\documentclass[preprint,12pt]{elsarticle}

\usepackage{amsmath, amsfonts, amsthm, amssymb, here, dsfont, hyperref}
\usepackage{color,multirow,here, subcaption,caption}
\usepackage{mathtools}
\usepackage[utf8]{inputenc} 
\usepackage[T1]{fontenc}
\usepackage{lmodern}
\usepackage{here,dsfont, hyperref}
\usepackage{enumitem}
\usepackage{url}
\usepackage{bbm}
\usepackage{appendix}
\usepackage{stmaryrd}
\usepackage{physics}
\usepackage{algorithm}
\usepackage{stmaryrd}
\usepackage{algpseudocode}
\usepackage{comment}
\usepackage{natbib}
\usepackage{xcolor}
\usepackage{booktabs} 
\usepackage{currfile-abspath}
\usepackage{subcaption}
\usepackage{stmaryrd} 
\usepackage{lmodern}

\newcommand{\modif}{} 

\newcommand{\w}{\widehat}

\renewcommand{\P}{\bb{P}}

\newcommand{\R}{\bb{R}}

\newcommand{\ov}[1]{\overline{#1}}

\newcommand{\bb}[1]{\mathbb{#1}}
\newcommand{\s}[1]{\mathcal{#1}}

\newcommand{\cov}{\mathrm{Cov}}

\newcommand{\limloiT}{\xrightarrow[T \to \infty]{\s{L}}}
\newcommand{\ps}[1]{\bb{P}_{#1}-\text{a.s}}

\newcommand{\abso}[1]{\left| #1 \right|}
\newcommand{\scalar}[2]{\langle #1, #2 \rangle}
\newcommand{\indic}{\ensuremath{\mathds{1}}}

\newcommand{\indep}{\perp \!\!\! \perp}

\newenvironment{changemargin}[2]{%
  \begin{list}{}{%
    \setlength{\topsep}{0pt}%
    \setlength{\leftmargin}{#1}%
    \setlength{\rightmargin}{#2}%
    \setlength{\topmargin}{0pt} 
    \setlength{\headheight}{13pt} 
    \setlength{\headsep}{8pt}
    \setlength{\listparindent}{\parindent}%
    \setlength{\itemindent}{\parindent}%
    \setlength{\parsep}{\parskip}%
    \setlength{\footskip}{50pt} 
    \setlength{\textheight}{600pt} 
  }%
  \item[]}{\end{list}}

\newcommand{\wN}[2]{\widetilde{N}_{#1}^{#2}}
\newcommand{\oN}[2]{\overline{N}_{#1}^{#2}}
\newcommand{\N}[2]{N_{#1}^{#2}}

\newcommand{\lamb}[2]{\lambda_{#1}^{#2}}
\newcommand{\wlamb}[2]{\widetilde{\lambda}_{#1}^{\, #2}}
\newcommand{\olamb}[2]{{\overline{\lambda}}_{#1}^{\, #2}}

\newcommand{\Lamb}[2]{\Lambda_{#1}^{#2}}

\newcommand{\T}[2]{T_{#1}^{#2}}

\newcommand{\wT}[2]{\widetilde{T}_{#1}^{\, #2}}

\newcommand{\wMLE}{\widetilde{\theta}_T}
\newcommand{\MLE}{\w{\theta}_T}
\newcommand{\thetastar}{\theta^{\star} }
\newcommand{\alphastar}{\alpha^{\star}}
\newcommand{\betastar}{\beta^{\star}}
\newcommand{\mustar}{\mu^{\star}}

\newcommand{\bastar}[2]{g_{ {{#1}^{\star}_{#2} }}}
\newcommand{\bas}[2]{g_{ #1_#2 }}

\newcommand{\indexset}[2]{\in \llbracket #1,\ldots, #2 \rrbracket}

\newcommand{\Hz}{\ensuremath{\mathrm{H_0}}}
\newcommand{\Hu}{\ensuremath{\mathrm{H_1}}}

\newtheorem{ass}{Assumption}[section]
\newtheorem{theo}[ass]{Theorem}

\newtheorem{prop}[ass]{Proposition}

\newtheorem{coro}[ass]{Corollary}
\newtheorem{lemme}[ass]{Lemma}
\newtheorem{rem}[ass]{Remark}

\newtheorem{assumption}{Assumption}
\newtheorem{assumptionalt}{Assumption}[assumption]
\newenvironment{customAss}[1]{
  \renewcommand\theassumptionalt{#1}
  \assumptionalt
}{\endassumptionalt}

\newtheorem{theoalt}{Theorem}[assumption]
\newenvironment{customTheo}[1]{
  \renewcommand
  \thetheoalt{#1}
  \theoalt
}{\endtheoalt}

\usepackage{thmtools}

\makeatletter
\newtheoremstyle{nopoint}%
  {3pt}   
  {3pt}   
  {\itshape}  
  {}      
  {\bfseries} 
  {}      
  {1em}   
  {}      
\makeatother

\theoremstyle{nopoint}

\begin{document}

\begin{frontmatter}

\title{Hawkes process with a diffusion-driven baseline: \\long-run behavior, inference, statistical tests}

\author[1]{Maya Sadeler Perrin \corref{cor1}}
\author[2]{Anna Bonnet}
\author[2,3]{Charlotte Dion Blanc}
\author[1]{Adeline Samson}

\cortext[cor1]{Corresponding author maya.sadeler@univ-grenoble-alpes.fr}

\affiliation[1]{
    organization={Univ. Grenoble Alpes, CNRS, Grenoble INP$^1$, LJK, 38000 Grenoble},
    country={France}
}

\affiliation[2]{
    organization={Sorbonne Université, CNRS, Laboratoire de Probabilités, Statistique et Modélisation,
    LPSM, UMR 8001},
    city={Paris},
    postcode={F-75005},
    country={France}
}

\affiliation[3]{
    organization={CMAP, Ecole Polytechnique, IP Paris, Palaiseau},
    country={France}
}

\fntext[label2]{Institute of Engineering Univ. Grenoble Alpes}

\begin{abstract}
Event-driven systems in fields such as neuroscience, social networks, and finance often exhibit dynamics influenced by continuously evolving external covariates. Motivated by these applications, we introduce a new class of multivariate Hawkes processes, in which the spontaneous rate of events is modulated by a diffusion process. This framework allows the point process to adapt dynamically to continuously evolving covariates, capturing both intrinsic self-excitation and external influences. In this article, we establish the probabilistic properties of the coupled process, proving stability and ergodicity under moderate assumptions. Classical functional results, including law of large numbers and mixing properties, are extended to this diffusion-driven setting. Building on these results, we study parametric inference for the Hawkes component: we derive consistency and asymptotic normality of the maximum likelihood estimator in the long-time regime, and derive stronger convergence results under additional assumptions on the covariate process. We further propose hypothesis testing procedures to assess the statistical relevance of the covariate. Simulation studies illustrate the validity of the asymptotic results and the effectiveness of the proposed inference methods. Overall, this work provides theoretical and practical foundations for diffusion-driven Hawkes models.
\end{abstract}

\begin{keyword}
Hawkes process with covariate Stochastic baseline,  Diffusion-driven intensity,  Ergodicity,  Mixing,  Functional theorems,  Maximum Likelihood Estimation,  Statistical testing
\end{keyword}
\end{frontmatter}

\section{Introduction}
Many systems exhibit event-driven dynamics, where the intensity of occurrences evolves through self-excitation and external time-varying covariates. For instance, in biology, mutation events in a cell population may depend on past mutations and stochastic environmental factors, while {\modif in neuroscience, the spike train of a neural network depends on external or latent inputs such as sensory stimuli, the spatial position of the animal, or environmental conditions}. In these frameworks, the probability of new events occurring is a functional of the history of the system and continuously evolving stochastic covariates. To describe this type of dependence, we consider a Hawkes process with stochastic baseline intensity, in which the spontaneous rate of activity is driven by a diffusion process. The covariate process $X(t)\in\bb{R}^m$ is modelled as the solution of the stochastic differential equation (SDE)
\begin{equation}
\label{eq:SDE_multidim}
    dX(t) = b(X(t)) dt + \sigma(X(t)) dW(t), 
    \qquad X(0)\sim \pi_0,
\end{equation}
where $W(t)$ is a standard $l$-dimensional Brownian motion with $l\le m$, $b: \bb{R}^m \to \bb{R}^m$  is the drift function, $\sigma: \bb{R}^m \to \bb{R}^{m \times l}$ is the diffusion coefficient, and $X(0)$ is distributed according to $\pi_0$. The point process $N = (N_1,\dots,N_d)$ describing event arrival has an intensity $\lamb{}{} = (\lamb{1}{},\dots,\lamb{d}{})$, which takes the form of 
\begin{equation}
\label{eq:modelLambda_general}
\lamb{i}{}(t):= g_i(X(t^-)) + \sum_{j=1}^d \int_{(0,t^-)} h_{ij}(t-s) d\N{j}{}(s), \quad i=1,\dots,d,
\end{equation}
where $g_i:\bb{R}^m\to\bb{R}_+$ is a state-dependent baseline intensity. The kernel $h_{ij}$ quantifies the influence of past events of component $j$ on the current intensity of component $i$ and characterizes how the occurrence of an event increases the probability of future events, and how this excitation decays over time.

The coupled system $(X(t),N(t))$ thus forms a diffusion–point process pair, in which the diffusion $X(t)$ modulates the spontaneous emission event rate of $N(t)$. {\modif  The main application that motivates this work comes from neuroscience, where the goal is to model particular neurons of the hippocampus, known as place cells \cite{zhang2023cerebellum}. Those neurons are known to have a firing behavior that depends on the animal position: such neuron fires preferentially when the animal occupies a specific region of the space, called the place field of the cell. In this setting, $X(t)$ thus represents the position of the animal at time $t$, the baseline $g_i$ encodes the place field of the cell and the kernel $h$ accounts for the dependence of the neuron on its own past activity, such as bursting.

Let us highlight that the SDE model on $X$ is mainly motivated by this application in neuroscience. This choice also facilitates the theoretical analysis that we perform here. In particular, it fits within the framework of Markov processes, which is central when establishing long-run properties of stochastic processes. However, most arguments employed in this paper do not rely directly on the specific SDE model assumed for $X$, but rather on the properties it induces. Nevertheless, we use this model for two reasons. First, since the process (X, N) is complex to study, a simple model for $X$ is a reasonable first step to introduce the Hawkes process with a continuous covariate. Furthermore, a large part of this work concerns the estimation of $N$. The aim is therefore to have a model in which the assumptions related to $X$ can be easily verified in practice. This is the case for SDEs, where many models have already been studied and theorised, so most assumptions are already verified in this set-up.

Keeping this example in mind, the goal here is to study this new diffusion-driven Hawkes process, both from a probabilistic and an inferential points of view, and to derive some statistical inference and robust tests to assess whether the intensity depends on the covariate.
}

\subsection*{State of the Art}

Introduced by \cite{hawkes1971point} in 1971, linear Hawkes processes are reference models for self-exciting point phenomena. In their classical formulation, each component $N_i$ of the process has a predictable intensity depending on a constant baseline $\mu_i>0$ (i.e., $g_i(x) = \mu_i$ in Equation \eqref{eq:modelLambda_general}) and a kernel $h_{ij}$ describing the influence of past events. For a general kernel $h_{ij}:\bb{R}_+\to\bb{R}_+$, existence, uniqueness, stability, and stationarity of the multivariate Hawkes process have been established \citep{hawkes1971point, daley2003introduction, bremaud1996stability}, and functional limit theorems have been derived for the asymptotic distribution of event counts \citep{zhu2013central}. Many works deal with inference for these processes. In general, approaches can be divided into nonparametric methods, such as kernel smoothing, projection, or moment-based techniques \citep{reynaud2013spike, bacry2015hawkes}, and parametric methods, including Bayesian inference \citep{sulem2024bayesian}, Maximum Likelihood Estimation \cite{ogata1978estimators,clinet2017statistical} and penalized estimation such as Lasso for high-dimensional or sparse interaction structures \citep{dion2024erm, kwan2024likelihood}.

A simplification occurs for exponential kernels, where the intensity $\lamb{}{}(t)$ admits a Markovian representation \citep{clinet2017statistical,laub2021elements}. This structure eases the probabilistic analysis and recursive likelihood computation, leading to stronger results both on the probabilistic behavior of the process —such as exponential ergodicity with explicit control of regeneration times \citep{cattiaux2022limit}—and on statistical inference, notably with central limit theorems for the maximum likelihood estimator, and moment convergence \citep{ogata1978estimators, clinet2017statistical}. {\modif In this paper, we focus on this exponential setting for most of our results. Some of the probabilistic results hold for more general kernels, and we specify for each property whether the exponential form is required.}

Exponential kernel is common, especially in neuroscience where many studies focusing on neuronal activity (studies not necessarily based on Hawkes processes) rely on an exponential decay of the impact of past spikes on the current neuronal activity \cite{bonnet2022neuronal, reynaud2013spike, chevallier2019mean}. Furthermore, this specification can be seen as a reasonable first step toward more general kernel step as recent studies suggest that general integrable kernels can be approximated by a linear combination of exponentials \cite{khabou2025markov}.

Classical Hawkes models remain autonomous: their dynamics depend solely on past events. To account for external influences, several generalizations have been proposed. Marked Hawkes processes include event-specific covariates or marks \citep{bonnet2024testing, zhuang2002stochastic, clinet2021asymptotic, chiang2022hawkes}, although the covariates are typically observed only at event times \citep{davis2024fractional}. Spatio-temporal extensions have also been considered \citep{bernabeu2025spatio}, but do not handle continuously evolving covariates. Other models build on a framework similar to Equations \eqref{eq:SDE_multidim},\eqref{eq:modelLambda_general}, allowing for time-varying coefficients in the intensity. These approaches, however, have been explored from a rather different perspective, focusing primarily on the high-frequency regime, where the observation window $T$ is fixed and the number of events $N$ goes to infinity \citep{clinet2017statistical, potiron2025high, kwan2024likelihood, cai2024latent}. This includes univariate exponential Hawkes processes with stochastic coefficients \citep{clinet2017statistical, potiron2025high}, as well as non-exponential kernels with deterministic, time-dependent baselines \citep{kwan2024likelihood}. Although these works provide consistent and asymptotically normal estimators, they do not investigate other probabilistic properties of the model—such as stability, ergodicity, or mixing behavior. Those properties are needed in our framework as we work in the long-time regime.

The low-frequency (long-duration) regime has received less attention, with some contributions addressing nonparametric inference for indicator-type kernels \citep{cai2024latent}. Recent works integrate covariates directly into the Hawkes process, but typically no explicit model is specified for the covariate, and the effect of its dynamics is not analyzed. Moreover, covariates are often not continuous in time or treated as deterministic. For instance, \cite{li2024multivariate} estimates the effect of the covariate $X$ through a log-linear function of the baseline intensity. A similar model is used in \cite{mohler2018improving}, which focuses on parametric estimation with penalization to improve inference and stability of the estimated parameters.
In \cite{boumezoued2025calibration}, the authors consider marked Hawkes processes, where part of the baseline intensity is driven by unobserved homogeneous Poisson components. However, these models rely on latent baselines. That differs from our framework where the external covariate $X(t)$ is continuous, fully observed, and modulates the intensity of the Hawkes process.

Given the variety of existing Hawkes models, we are interested in testing the relevance of the covariate in the baseline. Even for standard Hawkes processes,
statistical testing procedures are scarce. Classical goodness-of-fit methods, such as those based on the time-change theorem \citep{papangelou1972integrability}, require repeated events \citep{reynaud2014goodness} and have no straightforward extension to the long-duration regime. The only theoretically grounded procedure in this setting is the test proposed by \cite{baars2025asymptotically}, which accounts for the asymptotic fluctuations of the estimated compensator. In practice, however, implementing this test is complex and computationally demanding, so practical tools for assessing covariate relevance in such models remain limited.

\subsubsection*{Contribution and paper organization}

Our main contributions are as follows:
\begin{itemize}
    \item We establish the probabilistic properties of the process $(X(t),N(t))$, showing that under moderate assumptions, it admits stable and ergodic dynamics. We also extend  {\modif the law of large numbers}
    for the point process $N$.
    \item  We extend to our framework the classical mixing theorem known for standard Hawkes processes. We propose two results depending on the assumptions on the covariate process and on how it impacts the intensity.
    \item { \modif We study statistical inference for the model parameters in the long-time (low-frequency) regime, proving consistency and central limit theorems for the Maximum Likelihood Estimator (MLE) as well as the convergence of moments under general ergodicity conditions on $X(t)$.}
    \item We develop hypothesis testing procedures to assess the relevance of the covariate, enabling formal evaluation of whether including $X(t)$ significantly enhances data description.
\end{itemize}

The remainder of the paper is organized as follows. In Section \ref{sec:definition}, we introduce the Hawkes process with stochastic baseline intensity, define the coupled process $(X(t),N(t))$, and establish existence results. Section \ref{section:theorie} studies the long-time behavior of the process: exponential ergodicity of the intensity, convergence towards a stationary version, and a law of large numbers for $N$. Section \ref{section:MLE} focuses on statistical inference: we study the MLE in the long-time regime and establish its consistency and asymptotic normality. This section also contains the mixing properties of the intensity on which these results rely. Section \ref{section:test} introduces hypothesis testing for the relevance of the covariate. Finally, Section \ref{section:simu} illustrates our theoretical conclusions using simulations, presenting the behavior of the process, the performance of the MLE, and the testing procedures. All the code used to perform simulations is available \href{https://gitlab.com/ljk3941346/hawkesstochasticbaselineprocesses.jl}{here}.

\section{Hawkes process with stochastic time-dependent baseline}
\label{sec:definition}
Before introducing the model, let us first define some notation.

\subsection{Notation}  
\label{subsection:notation}

Throughout the paper, we work on a probability space $(\Omega, \s{F}, \bb{P})$ with canonical filtration
\[
\s{F}_t:= \sigma\qty(X(0), W(s), N(s), s \le t),
\]
where $X$ is the covariate process, driven by a noise process $W$, and $N = (N_1, \dots, N_d)$ is the $d$-dimensional counting process. For any stochastic process $Z$, $\s{F}^Z$ denotes its canonical filtration, for instance $\s{F}^W$ or $\s{F}^N$. Let $(T_n)_{n \ge 1}$ denote the jump times of $N$, with $T_0:= 0$. The true Hawkes intensity of component $N_i$ at time $t$ is denoted by $\lamb{i}{\star}(t)$, {\modif such that $$\lamb{}{\star}(t) d t=\bb{E}\left[d N(t) \mid \s{F}_t^{(N,X)}\right].$$}
The full intensity vector is $\lamb{}{\star}(t):= (\lamb{1}{\star}(t), \dots, \lamb{d}{\star}(t))^\top \in \bb{R}_+^d$ and the associated compensator
\[
\Lamb{}{\star}(T):= \int_0^T \lamb{}{\star}(t)  dt
\]
is the unique predictable increasing process such that $M(T) := N(T) - \Lamb{}{\star}(T)$ is a martingale with respect to $(\s{F}_t)_{t \ge 0}$, see \emph{e.g.} \cite{daley2003introduction, daley2008introduction}.

For a differentiable function $f: \bb{R}^n \times \bb{R}^p \to \bb{R}$ depending on a state variable $x$ and a parameter $\eta$, we denote by $\nabla_x f_\eta(x)$ the gradient with respect to $x$, and by $\nabla_\eta f_\eta(x)$ the gradient with respect to $\eta$. Higher-order derivatives with respect to $\eta$ are denoted $\nabla_\eta^j f_\eta(x)$. Finally, we denote $ \s{D}(E, \bb{R})$ the set of functions $\psi \in  \s{C}^1(E, \bb{R})$ such that $(u,v,w) \mapsto \nabla \psi(u,v,w)$ has polynomial growth, and $ \s{C}_b(E, \bb{R})$ the set of continuous bounded functions on $E$.

Finally, we denote by $\norm{ . }$ the Euclidean norm.

\subsection{Model description}

We consider a multivariate Hawkes process  influenced by an external covariate that evolves continuously over time. The joint process $(X(t), N(t))$ combines a diffusion for the covariate, with a self-exciting point process whose dynamics depend both on the current state of the covariate and on the history of past events. The covariate $X(t) \in \bb{R}^m$ is modeled as the solution of a SDE driven by a standard $l$-dimensional Brownian motion given in Equation \eqref{eq:SDE_multidim}. To ensure that this diffusion is well-defined, we impose the following regularity condition on its drift and diffusion coefficients.

\begin{ass} \label{ass:existtence_SDE}
The coefficient functions $b:\R^m\rightarrow \R^m$ and $\sigma:\R^{m}\rightarrow \R^{m\times l}$ are continuous and globally Lipschitz.
\end{ass}
Under Assumption \ref{ass:existtence_SDE}, the SDE defined in Equation \eqref{eq:SDE_multidim} admits a unique global strong solution. Conditionally on $X(t)$, the multivariate Hawkes process $N$ is then defined through its intensity for each component $i$:
\begin{equation}
\label{eq:modelLambda}
\lamb{i}{\star}(t):=  g_i^\star(X(t^-)) + \sum_{j=1}^d \int_{(0,t^-)} h_{ij}^\star(t-s)  d\N{j}{}(s) .
\end{equation}

We further require a stability condition ensuring that self and cross-exciting effects remain bounded.
\begin{ass} \label{ass:sta_hawkes}
The kernel matrix $K:= \qty(\norm{h_{ij}^\star}_{\s{L}_1})_{i,j}$ satisfies $\rho(K)<1$, where $\rho(K)$ denotes its spectral radius.
\end{ass}
Assumptions \ref{ass:existtence_SDE} and \ref{ass:sta_hawkes} guarantee that the joint process $(X(t), N(t))$ is well-posed, ensuring existence and uniqueness of $X(t)$ and stability/non-explosion of $N(t)$. These conditions are always assumed throughout this work. {\modif
For most of the results stated in this article, we rely on the exponential form for kernel functions $h_{ij}^\star$. Some results can be generalized to other kernels. We thus state this exponential form as an assumption that appears in the statements of different theorems when needed.

\begin{ass}
\label{ass:exp_kernel} There exist $\alphastar_{ij}, \betastar_{ij}>0$ such that:
\begin{equation}
\label{eq:exp_kernel}
    h_{ij}(t) = \alphastar_{ij} e^{-\betastar_{ij} t}.
    \end{equation}    
\end{ass}
}

\subsection{Existence result}

In order to prove the existence of the process, let us exhibit a thinning representation that applies to the Hawkes component of the process defined by Equations \eqref{eq:SDE_multidim},\eqref{eq:modelLambda}.

{\modif
\begin{ass}[Boundedness of the random baseline]
\label{ass:bounded_baseline1}
 It is assumed that there exist $g_{+}>0$ and $x_0 \in \bb{R}^m, r \in (0, +\infty]$ such that for any initial condition $x \in\s{B}(x_0,r)$, $X(t) \in\s{B}(x_0,r)  \quad \ps{x}$ and $g^{\star}_i(\s{B}(x_0,r)) \subset (0, g_+]$.
\end{ass}

This condition ensures that there is an upper bound for the process $g^{\star}_i(X(t))$. When it comes to verify it, one can consider two approaches. Either the baselines are bounded, in which case the condition is automatically satisfied, or one must ensure that the SDE admits bounded trajectories and that these trajectories remain in regions of space where $g$ is strictly positive. 
In both cases, $g^{\star}_i(X(t))$ is bounded and strictly positive, which ensures the existence of a thinning procedure. 
}

\begin{theo}[Existence]
\label{theo:existence}
Suppose that Assumption \ref{ass:bounded_baseline1} holds. 
Then, the process $(\lamb{}{\star}(t))$ admits a Thinning procedure representation given by
\begin{align*}
& dX(t) = b(X(t)) dt + \sigma(X(t)) dW(t), \\
& \lamb{i}{(n+1)}(t) = g_{i}^\star (X(t^-)) + \sum_{j=1}^d \int_{(0,t)} h_{ij}^{\star}(t-s)  d\N{j}{(n)}(s), \\
& d\N{j}{(n+1)}(t) = \N{j}{P} \qty( [0,\lamb{j}{(n+1)}(t)] \times dt ),
\end{align*}
with $\lamb{}{(0)} = 0$, $\N{}{(0)} = \emptyset$, $\N{j}{P}$ a homogeneous unitary Poisson random measure on $\bb{R}_+^2$ with intensity measure equal to the Lebesgue measure and where $\N{j}{P}\qty( [0,\lamb{j}{(n+1)}(t)] \times dt )$ denotes the number of points of $\N{j}{P}$ in $[0,\lamb{j}{(n+1)}(t)] \times [t,t+dt]$.
\end{theo}

\begin{proof}
    Assumption \ref{ass:existtence_SDE} ensures the existence of a unique strong solution on $[0,T]$ for the SDE, for any $T>0$. We can therefore work conditionally to the Brownian filtration. We then fall back to the case of a Hawkes process with a deterministic baseline function which is bounded, thanks to Assumption \ref{ass:bounded_baseline1}, and thus apply Theorem 1 of \cite{cai2024latent} which studies a similar set-up. 
\end{proof}


\section{Asymptotic behavior of the process in the long-time regime}
\label{section:theorie}

In this section, we study the asymptotic properties of the point process 
$N(t)$ and its intensity $\lamb{}{\star}(t)$. We focus on two main aspects: ergodicity and limit theorems. These results characterize the long-term behavior of the process and extend classical results obtained for Hawkes processes with constant baselines \cite{ogata1978estimators, bremaud1996stability, clinet2017statistical, bacry2013some}. {\modif Another long-term property of interest is the mixing property of the process, which we choose to present in Section \ref{subsec:ergodic_mixing} since it serves to study the behavior of the Maximum Likelihood Estimator.}

\subsection{Exponential ergodicity of the intensity process}

For this section, we always assume Assumption \ref{ass:exp_kernel}. Indeed, to establish the ergodicity of the intensity, we introduce a Markovian representation of the system, which requires an exponential form for the function $h_{ij}^\star$. We can then rely on results from the ergodic theory of Markov processes.

\subsubsection{Markovian representation}

Consider the joint process $Z(t):= (X(t), Y(t)) \in \bb{R}^m \times \bb{R}^{d^2}$ with $Y$ defined by 
\begin{equation}
\label{eq:kernelY}
    Y_{ij}(t):= \int_{(0,t)} \alphastar_{ij} e^{-\betastar_{ij}(t - s)}   d\N{j}{}(s), 
    \quad Y(t) = \qty(Y_{ij}(t))_{1 \leq i, j \leq d},
\end{equation}
so that
\[
\lamb{i}{\star}(t) = g_{i}^\star(X(t^{-})) + \sum_{j=1}^d Y_{ij}(t).
\]
Our goal is to analyze the asymptotic behavior of $Z$ by applying the ergodic theory for Markov processes developed in \cite{meyn1993stability}. The dynamics of $Z$ is given by the coupled system of equations:
\begin{equation} \label{eq:markov_eq}
\begin{cases}
dX(t) = b(X(t))   dt + \sigma(X(t))   dW(t), \\
dY_{ij}(t) = -\betastar_{ij} Y_{ij}(t)   dt + \alphastar_{ij}   d\N{j}{}(t),
\end{cases}
\end{equation}
{
\modif where the equation on $Y$ can be obtained by application of the derivation formula for a product, writing that $Y_{ij}(t) =  \alphastar_{ij} e^{-\betastar_{ij}t}\mathrm{A}_{ij}(t)$ with $\mathrm{A}_{ij}(t) =\int_0^t e^{\betastar_{ij}s} d\N{j}{}(s)$.
}

Equation \eqref{eq:markov_eq} shows that $Z$ is Markovian w.r.t. the filtration $(\s{F}_t)$.

\subsubsection{Ergodicity property}

To ensures the ergodicity of the diffusion-driven baseline process, we impose structural assumptions on the diffusion $X(t)$. We begin with classical conditions, which guarantee that the SDE admits smooth transition densities and suitable Lyapunov functions.

\begin{ass} Consider the following conditions.
\label{ass:ergodicitySDE}
\begin{itemize}
\item \textbf{Hörmander's condition.} \label{ass:hormander}
Let $\sigma_1,\ldots,\sigma_m$ denote the columns of the diffusion matrix $\sigma$ and set $\sigma_0:= b$ for the drift vector field.  
We say that the Hörmander condition holds if the Lie Algebra generated by $\{ \sigma_0, \ldots ,\sigma_m \} $ span $\bb{R}^m$.

\item \textbf{Dissipative drift and bounded diffusion.} Suppose that the function $b$ satisfies \[
\langle b(x), x \rangle \leq -\gamma \norm{x}^q, \quad \text{for all } \norm{x} \geq \s{R},
\]
for some constants $ \s{R} > 0$, $\gamma > 0$, and $q \geq 1$, and the diffusion matrix
\[
a(x):= \sigma(x)\sigma(x)^\top
\]
is bounded on $\bb{R}^m$.
\end{itemize}
\end{ass}

Assumption \ref{ass:ergodicitySDE} has two classical consequences for the diffusion $X$.
First, Hörmander's condition ensures that the
transition kernel $P_t^X(x,dy)$ admits a smooth density $p_t(x,y)$ which is $\s{C}^\infty$ in $t$, $x$ and $y$,
with respect to the Lebesgue measure (see \cite{rogers2000diffusions}). 
Second, the dissipativity and ellipticity assumptions guarantee the existence of norm-like Lyapunov functions.  

{\modif \begin{lemme}
    Suppose Assumption \ref{ass:ergodicitySDE} holds and define the following function
\begin{equation}
\label{eq:lyapunov_function}
V_1(x)=\exp \qty(\xi\norm{x}),
\end{equation}
with $\xi>0$ small enough. Then, $V_1$ satisfies the drift condition
\begin{equation}
    \label{eq:lyap_condition}
    \s{A}^X V_1 \le -c_1 V_1 + d_1 \indic_{K_1}
\end{equation}
for suitable constants $c_1,d_1>0$, a compact set $K_1\subset\bb{R}^m$ and where $\s{A}^X$ denotes the infinitesimal generator of $X$.

Furthermore, the SDE admits an invariant measure, denoted $\pi_X$.
\end{lemme}
}

See \cite{amorino2021invariant} for a proof. These two properties — smooth transition densities and existence of a Lyapunov function — are the usual conditions required to apply the ergodic theory of \cite{meyn1993stability} to SDE. 
They assure the diffusion $X$ is ergodic and therefore admits a unique invariant probability measure, denoted $\pi_X$. These assumptions are relatively easy to verify and have already been established for a variety of SDE models \cite{mattingly2002ergodicity,gobet2002lan}.

In addition to the ergodicity condition on $X$, we also need to strengthen Assumption \ref{ass:bounded_baseline1} to ensure the ergodicity of the process $\lamb{}{\star}$. We require that the functions $g_{i}^\star$ are not only upper bounded but also bounded away from zero. 

{\modif 
\begin{ass}[Boundedness of the random baseline] \label{ass:bounded_baseline2}
There exist constants $\bas{}{-}$ and $\bas{}{+}$ and $x_0 \in \bb{R}^m, r \in (0, +\infty]$ such that for any initial condition $x \in\s{B}(x_0,r)$, the process remains almost surely within this interval: $X(t) \in \s{B}(x_0,r) $ $\bb{P}_x$-a.s., and for all $i \in \llbracket 1, d \rrbracket$, we have
    \[
    0 < \bas{}{-} \leq g_{i}^\star(x) \leq g_{+}, \quad \text{for all } x \in [x_1, x_2].
    \]
\end{ass}
}

We now establish the exponential ergodicity of the process $Z$.
Let us define the transition kernel $P_t^{Z}((x,y),\cdot)$:
\[
P_t^{Z}((x,y),(A,B)):= \bb{P}\qty(X(t) \in A,  Y(t) \in B  \middle|  X(0) = x,  Y(0) = y).
\]
\begin{theo}
\label{theo:exp_ergodicity}
Under Assumptions \ref{ass:exp_kernel},\ref{ass:sta_hawkes},\ref{ass:ergodicitySDE},\ref{ass:bounded_baseline2}, the process $Z$ is exponentially ergodic. In particular, there exists an invariant measure $\pi$ and constants $B < +\infty$ and $\eta < 1$ such that for all $x \in \s{B}(x_0,r)$ and all $y >0$,
\begin{equation}
    \label{def:exp_ergodicity}
    \norm{P_t^{Z}((x,y),\cdot) - \pi(\cdot)}_{V} \leq B V(x,y) \eta^t,
\end{equation}
where $\norm{\mu}_V:= \sup_{\abs{g} \leq V} \left|\int g(s)  \mu(ds)\right|$ denotes the $V$-norm (a weighted extension of the total variation norm), and function $V: (x,y) \in \bb{R}^m \times \bb{R}^{d^2} \mapsto V_1(x) + e^{\langle m, y \rangle}$, with $V_1$ defined in Equation \eqref{eq:lyapunov_function}. 
\end{theo}

The proof can be found in \ref{proof:exp_erg}. This theorem ensures that the process $Z$ is exponentially ergodic. In particular, $Z(t)$ converges geometrically fast to its invariant measure $\pi$ in the $V$-norm. Furthermore, it satisfies Birkhoff's ergodic theorem: for any function $\psi$ dominated by $V$, we have
\[
\frac{1}{T} \int_0^T \psi(Z(s)) ds \xrightarrow[T \to \infty]{} \pi(\psi) \quad \text{almost surely and in } \s{L}^1(\bb{P}).
\]

This result ensures that the intensity vector $\lamb{}{\star}(t)$ is itself ergodic. Indeed, to establish ergodicity it is sufficient to apply Birkhoff's theorem to all functions in $\s{L}^1$, which guarantees convergence of time averages. For any polynomial function $\psi$, we can write $\psi(\lamb{}{\star}(t)) = \psi \circ \nu(Z(t))$, where $\nu$ maps $Z(t)$ to $\lamb{}{\star}(t)$. {\modif Since $\nu(x,y) = \qty(g_{i}^\star(x) + \sum_{j=1}^d y_{ij})_{1\le i \le d}$ and since the functions $g_{i}^\star$ are bounded, $\psi \circ \nu$ can be upper bounded by a polynom in $\norm{y}$. As a result, $\psi \circ \nu(z)$ can be upper bounded by $V$ (as $V$ is the sum of an exponential function of $\norm{y}$ and an exponential function in $\norm{x}$). Thus, for every $z \in \bb{R}^m \times \bb{R}^{d^2}$, $\psi \circ \nu(z) \leq V(z)$. It follows that $\lamb{}{\star}(t)$ inherits the ergodicity property of $Z(t)$.}

\subsection{ Asymptotic limit theorem}

In this section, we present asymptotic results for the counting process and its associated intensity. These results, stated as a Proposition and a Theorem, establish the stability of $\lamb{i}{\star}(t)$ and a Law of Large Numbers (LLN) for $N(t)$.

\begin{prop}
\label{prop:stationary_intensity_true_process}
Suppose Assumptions \ref{ass:ergodicitySDE} and \ref{ass:bounded_baseline2} hold, and define the stationary version of the intensity as
\begin{equation}
\label{eq:stationary_intensity}
\olamb{}{\star}(t) 
= g_{i}^\star(\ov{X}(t^-)) 
+ \sum_{j=1}^d \int_{(-\infty,t)} h_{ij}^\star(t-s)   d\oN{j}{}(s),
\end{equation}
where  $\ov{X} = (\ov{X}(t))_{t\in\bb{R}}$ is distributed according to the stationary law $\pi_X$ of the SDE and where $\oN{j}{} = (\oN{j}{}(t))_{t\in\bb{R}}$ denotes the stationary Hawkes process.

Then the following holds.
\begin{enumerate}
    \item The intensity converges towards the stationary one, meaning 
    \[
    \lamb{i}{\star}(t) \xrightarrow[t \to \infty]{\s{L}} \olamb{i}{\star}.
    \]
\item  Moreover, if Assumption \ref{ass:exp_kernel} holds, there exist constants $c>0$ and $C>0$ such that,
    \[
    \bb{E}\Big[\big| \lamb{i}{\star}(t) - \olamb{i}{\star}(t) \big|\Big]
    \le C e^{-ct}.
    \]
\end{enumerate}

\end{prop}

The proof is given in \ref{proof:stationary_intensity_true_process}. The first part follows from the stability results in \cite{bremaud1996stability}, extending classical Hawkes process stability to the covariate-dependent setting. {\modif This proposition involves the stationary process $\ov{N}$, constructed on $\bb{R}$ as in \cite{bremaud1996stability}, with a baseline driven by a
stationary version $\ov{X}$ of the covariate. Since it starts at $-\infty$, its intensity depends on the whole past of the process, and no initial condition is needed.

Note that this result does not rely on the exponential form of the kernel, but is simply a consequence of the stationarity condition stated in Assumption \ref{ass:sta_hawkes}. It is thus true for a general kernel $h^\star = (h_{ij}^\star)$  which verifies $\norm{h^\star}_{\s{L}^1}<1$. The second part of Proposition \ref{prop:stationary_intensity_true_process} quantifies the exponential rate of convergence to the stationary intensity, when the kernel is exponential. This is a consequence of the exponential ergodicity of $X$ and the exponential kernel.} We now generalize the usual law of large number for Hawkes processes with a constant baseline, proven in \cite{bacry2013some}. The main difference is that the limiting behavior now involves the stationary mean of the covariate-dependent baseline:
\begin{equation}
\label{eq:barmu}
    \bar\mu:= \int g^{\star}(x) \pi_X(dx).
\end{equation}
The structure of the LLN remains unchanged, with the baseline replaced by its stationary average.

\begin{theo}
\label{theo:LLN_CLT}
Suppose that Assumptions \ref{ass:bounded_baseline1},\ref{ass:ergodicitySDE} hold. Let $\bar\mu$ be defined by Equation \eqref{eq:barmu}, and recall $K=\qty( \norm{h_{i j}^{\star}}_{\s{L}_1})_{1\le i,j\le d}$. Then, the following holds.
  \[ 
  \sup_{v\in[0,1]} \norm{\frac{1}{T}N(Tv) - v (I-K)^{-1}\bar\mu } \xrightarrow[T\to\infty]{} 0,\quad \text{almost surely and in } L^2(\bb{P}).
  \]
\end{theo}

As announced, we find the same structure as in the classical Hawkes process with a constant baseline, as proven by \cite{bacry2013some}.

{\modif 
\paragraph{Sketch of proof of Theorem \ref{theo:LLN_CLT}} The proof of this theorem, which can be found in \ref{proof:LLN_CLT}, is largely inspired by \cite{bacry2013some}. 

Write $N(t)=m(t)+V(t)$, where $m(t)=\bb{E}(N(t))$ and $V$ is the centered process. The function $m$ satisfies a Volterra equation, and $V$ satisfies:
\[
V = M + \Delta G + \Psi\ast\qty(M+\Delta G),
\]
where $\Delta G(t)=\int_0^t\qty(g^\star(X(s))-\bb{E}\qty(g^\star(X(s))))\,ds$ denotes the fluctuations of the baseline around its mean. $\Psi$ is the resolvent kernel of $h^\star$, defined as the solution of $\Psi(t) = h^\star(t) + \int_0^t h^\star(t-s)\Psi(s)ds$. The Volterra equation and the ergodicity of $X$ give $m(Tv)/T\to v(I-K)^{-1}\bar\mu$ as $T$ goes to infinity. The term $V(Tv)/T$ vanishes, by Doob's inequality on $M$ and using the ergodic theorem for $\Delta G$, so that $N$ and $m$ share the same limit.
}

{\modif

\subsection{Discussion and relaxation of assumptions}
\label{sec:discussion_assumption}
The results presented depend on assumptions about either the baseline or the covariate. These assumptions were chosen to ensure the paper is consistent and coherent, particularly with regard to the statistical section (see Section \ref{section:MLE}).  For this reason, we discuss here their potential extension to a more general set-up. 

\subsubsection*{Boundedness condition}
The results of this section rely on either Assumption \ref{ass:bounded_baseline1} or \ref{ass:bounded_baseline2}, which state that the baseline is upper and lower bounded. However, those conditions are not minimal. For example, we expect that the lower bound assumption can be relaxed: the process $g^\star_i(X(t))$ has to be strictly positive at some point so that the first jump can happen. Thus, it would suffice to assume that $g^\star$ is lower bounded on a region of the space visited by $X$. 

Yet, adapting the proof requires adding other conditions on the covariate to prove the ergodicity and especially the $T$-chain property. And, as the assumptions (especially in the inference section) can already be tedious to follow, we chose to keep global coherence for these assumptions. In particular a lower bound on the baselines is required for the inference of the process, and refinement of the assumption for the probabilistic results would therefore not change the scope of the following ones. 

}
\subsubsection*{Ergodicity assumptions on $X$}
\label{subsec:assumptions_X_ergo_general}
The value of the assumptions on $X$ lies in their comprehensiveness, and also in the ease with which they can be verified. Indeed, these assumptions are standard in the literature on SDEs and have therefore already been studied for many diffusion models.

Yet they could be replaced by more general assumptions, which would allow us to consider a slightly broader context. Indeed the proof relies on the general Markovian stability conditions stated in \cite{meyn1993stability}. This set-up only requires that $X(t)$ is a left-continuous Markov process satisfying the following properties.
\begin{enumerate}
\item \textbf{Lyapunov drift condition.}
There exists a norm-like Lyapunov function $V_1$ such that
\[
\s{A}^X V_1 \le -c_1 V_1 + d_1 \mathbf{1}_{K_1},
\]
where $\s{A}^X$ denotes the infinitesimal generator, $c_1,d_1>0$ are constant and $K_1$ is a compact set of $\bb{R}^m$.

\item \textbf{$T$-chain condition.}
There exists a transition kernel $Q_{1}$ such that for all $x\in\bb{R}^m$, for all Borel set $A$,
\[
P_T^X(x,A)\ge Q_{1,T}(x,A)
\]
with $Q_{1,T}$ verifying 
\begin{itemize}
    \item for every Borel set $A$, the map $x\mapsto Q_{1,T}(x,A)$ is lower semi-continuous
    \item $Q_{1,T}(x,\bb{R}^m)>0$ for all $x$.
\end{itemize}
\item \textbf{Accessible point.}
There exists $x_0\in\bb{R}^m$ such that for every $z$, every neighborhood $\s{U}$ of $x_0$ satisfies
\[
P_T^X(z,\s{U})>0
\quad\text{for some }T>0 .
\]

\end{enumerate}

Note that these general conditions are automatically satisfied under Assumption \ref{ass:ergodicitySDE}. Existence of a density for $P^X_t( x,   \cdot   )$ leads to
\[
Q_{1,t}(x,A)
= \int_A p_t(x,y) dy
\]
which is lower semi-continuous by Fatou's lemma. Furthermore, it ensures that every point is accessible. Similarly, the dissipative and ellipticity condition {\modif ensure} the existence of a Lyapunov function of the form
$ V_1(x)=\exp(\xi\norm{x}) $ satisfying the required drift inequality. 

Throughout the paper, our main results are stated under Assumption \ref{ass:ergodicitySDE} for clarity and simplicity. However, each of these results remains valid if one replaces Assumption \ref{ass:ergodicitySDE} with the more general Markovian conditions listed above.


\section{Inference via likelihood maximization}
\label{section:MLE}

{\modif In this section, we investigate the inferential properties of the Maximum Likelihood Estimator (MLE) of the point process component of the model. This section is divided into several subsections. The first subsection defines the parametric model as well as the associated likelihood. Then, we state the main result, which consists of a CLT for the estimator and we present a sketch of proof of this theorem. The following subsection is devoted to some intermediate results that are required to prove the CLT and that represent interesting properties for the parametric family of estimation.

\subsection{Statistical model}
From now on, we only consider an exponential form for the kernel, thus assuming that Assumption \ref{ass:exp_kernel} holds. We work in a parametric framework and suppose the family of models used for estimation is the following,
\begin{equation}
\label{eq:modelEstim_lambda}
\lamb{\theta,i}{}(t):=  \bas{\mu}{i}(X(t^-)) + \sum_{j=1}^d \int_{(0,t^-)} \alpha_{ij} e^{-\beta_{ij}(t-s)}  d\N{j}{}(s) .
\end{equation}
The parameter $\theta$ is composed of three types of parameters $\mu$, $\alpha$, and $\beta$, where $\mu$ parametrizes the baseline functions while $\alpha$ and $\beta$ parametrize the interaction functions:
\[
\theta:= \qty( (\mu_i)_{i=1}^d, (\alpha_{ij})_{i,j=1}^d, (\beta_{ij})_{i,j=1}^d ).
\]
The parameters space $\Theta$ is assumed to be a Cartesian product
$
\Theta:= \Theta_1 \times \Theta_2,
$
with $\Theta_1$ a compact representing all admissible values for $\mu = (\mu_i)_{i=1}^d$, and $(\alpha, \beta)=((\alpha_{ij})_{i,j}, (\beta_{ij})_{i,j}) \in \Theta_2$ with $\Theta_2:= [\underline{\alpha}, \ov{\alpha}]^{d^2} \times [\underline{\beta}, \overline{\beta}]^{d^2}$ for some strictly positive constants $\underline{\alpha}, \overline{\alpha}, \underline{\beta}, \ov{\beta}$.
Finally, we assume the existence of a true parameter $\thetastar \in \Theta$ such that
\[
\lamb{}{\star}(t):= \lamb{\thetastar}{}(t)
\]
denotes the true intensity vector. For any parameter $\theta \in\Theta$, the associated compensator is denoted $\Lamb{\theta}{}(\cdot)$. Note that for all $\theta,$ the jumps of $\lamb{\theta}{}$ are the same as those of $\lamb{\thetastar}{}$. Those are the jumps of the process $N$, that are observed in practice. This constitutes the framework for the parametric estimation of the point process. 

Concerning the estimation for the covariate, the specific structure of the model allows us to dissociate the estimation of the point process component from that of the diffusion component: the interaction between the two components is unidirectional meaning that the intensity $\lamb{}{\star}$ depends on $X$, while $X$ is not impacted by $\lamb{}{\star}$. As a result, the estimation of $\thetastar$ is independent of any kind of inference on the model for $X$, and, estimating a model for $X$ does not require any knowledge of $\thetastar$.  Estimation on $X$ itself relies on standard techniques (see \cite{kutoyants2013statistical} for example) and does not constitute a contribution of this work. We thus do not discuss it further, and only present the inference of the point process.

\subsection{Assumptions}

Before stating the theorem, we present here the required assumptions.
The first requirement is a regularity condition with respect to the parameter $\mu$, which ensures the differentiability and continuity of the likelihood function with respect to the estimated parameters.

\begin{ass}[\textbf{k}]{\normalfont (Regularity of $g$)}
\label{ass:regularity_g}
Assume that for all $i \indexset{1}{d}$ and all $\mu$, $x \mapsto \bas{\mu}{i}(x)$ is continuous, and that $\mu \mapsto \bas{\mu}{i}(x)$ is $ \s{C}^k$ for all $x$.
\end{ass}

The second assumption is an identifiability condition, necessary to guarantee that the parameter $\theta$ can be uniquely recovered from the intensity trajectories. It also ensures that the Fisher information matrix is non-degenerate.

\begin{ass}[Identifiability]
\label{ass:identifiability_baseline}
Assume that for all $i \indexset{1}{d}$
\begin{enumerate} [ref=\theass.\arabic*]
    \item  The mapping $\mu \mapsto \bas{\mu}{i}$ is one-to-one.
\item \label{ass:identifiability_densityXbar} For all $\mu \ne \mustar$, $\pi_X \qty(\qty(\bas{\mu}{i} - \bastar{\mu}{i})^{-1}(\bb{R}^{d_i}\backslash \{0\}))>0$ and $\pi_X\qty((\nabla_\mu   \bas{\mu}{i})^{-1}\qty(\bb{R}^\star))>0$.    
    \item  \label{ass:identifiability_elliptic}$\sigma \sigma^T (.)$ is positive definite and for any $\mu,\mustar$, 
$$\qty( \nabla_x (\bas{\mu}{i} - \bastar{\mu}{i})(\ov{X}) = 0 ) \Rightarrow  \qty( \text{ for all } x \in \bb{R}^m,    \nabla_x (\bas{\mu}{i} - \bastar{\mu}{i})(x)=0)$$
\end{enumerate}  
with $\ov{X} \sim \pi_X$, the invariant measure associated with $X$.
\end{ass}

Assumption \ref{ass:identifiability_baseline} strengthens Assumption \ref{ass:ergodicitySDE} which already imposes Hörmander's condition together with a dissipative drift. However, for the identifiability analysis developed below, the boundedness of $a(x)$ is no longer sufficient. In particular, to ensure identifiability of the parameter $\mu$, we require that $\sigma^T\sigma$ is semi-definite positive, as stated in Assumption \ref{ass:identifiability_elliptic}. To further clarify the need for this specific item, we invite the reader to consult Section \ref{section:identifiability}, where it is discussed in detail.

\begin{rem}
 In the specific setting of one-dimensional SDEs, i.e., when $X(t) \in \bb{R}$, the existence of a density {\modif ensures}
 that Assumption \ref{ass:identifiability_densityXbar} and the second point of Assumption \ref{ass:identifiability_elliptic} hold. Indeed, the stationary distribution of this SDE is explicitly known and admits a strictly positive density on $\bb{R}$, see \cite{kutoyants2013statistical}. The stationary measure then assigns a positive mass to every region of the state space, including those required to distinguish between the different basis functions $\bas{\mu}{i}$. 
\end{rem}

Next assumption ensures the existence of moments for the process $\bas{\mu}{i}\qty(X(t))$. This condition involves two parameters, $(p, j)$, where $p$ specifies the moment order and $j$ the derivative order. Later, different combinations of $(p, j)$ are required, depending on the specific results we aim to prove.
\begin{ass}[\textbf{p,j}]{\normalfont (Moment existence)}
\label{ass:moment_g(X)}
We assume  that for all $i \in \llbracket 1, d\rrbracket $ we have the following 
\[ \sup_{t\geq 0} \bb{E}\qty(\norm{\sup_{\mu \in \Theta_1} \nabla_\mu^j \bas{\mu}{i}\qty(X\qty(t))}^p) < + \infty.\]
\end{ass}

\begin{rem}
Assumption \ref{ass:ergodicitySDE} ensures that $X$ admits moments of all orders, as the Lyapunov function  \ref{eq:lyapunov_function} dominates any polynomial function $x \mapsto |x|^p$. As a result, Assumption \ref{ass:moment_g(X)} is satisfied whenever $\bas{\mu}{i}(x)$ is a polynomial function of $x$ with coefficients uniformly bounded in $\mu$.  
For example, $\bas{\mu}{1}(x) = \mu_{11} x^T + \mu_{12}$ satisfies Assumption \ref{ass:moment_g(X)}.
\end{rem}

\begin{ass}[Boundedness of the random baseline]
\label{ass:bounded_baseline3}
We assume that there exist two strictly positive constants $g_{-}$ and $g_{+}$ and $x_0 \in \bb{R}^m, r \in (0, +\infty]$ such that for any initial condition $x \in\s{B}(x_0,r)$, the process remains in this ball almost surely. Namely, $X(t) \in\s{B}(x_0,r)$ $\bb{P}_x$-a.s., and for all $x \in \s{B}(x_0,r)$:
    \begin{align*}
        & g_{-} \leq  \inf_{\mu \in \Theta_1 } \min_{1\leq i \leq d}\bas{\mu}{i}(x)\\
        &  \sup_{\mu \in \Theta_1 } \max_{1\leq i \leq d}\bas{\mu}{i}(x)   \leq g_{+}
    \end{align*}
\end{ass}

 We now turn to the asymptotic analysis of the MLE.

\subsection{Likelihood and Asymptotic behavior of the MLE}

{\modif 
For the marginal process $(N(t))_{t \in [0,T]}$, the conditional log-likelihood function is (see \cite{daley2003introduction}):
\begin{align*}
\ell_T^{(N\vert X)}(\theta) &=  \sum_{i=1}^d \qty(\int_0^T \log\qty(\lamb{\theta,i}{}(t) )   d\N{i}{}(t) - \int_0^T \lamb{\theta,i}{}(t)   dt).
\end{align*}
We thus define the Maximum Likelihood Estimator (MLE) as
\begin{equation}\label{eq:mle}
\w{\theta}_T \in \arg \max_{\theta \in \Theta}   \ell_T^{(N\vert X)}(\theta).
\end{equation}

\begin{theo}[Asymptotic normality of the MLE]
\label{theo:AsymptoticNormality}
Suppose Assumptions \ref{ass:exp_kernel},\ref{ass:ergodicitySDE}, \ref{ass:bounded_baseline3},\ref{ass:regularity_g}(4),\ref{ass:identifiability_baseline}. Suppose furthermore that Assumption \ref{ass:moment_g(X)}(p,j) hold for all $p>1$ and $j \in \{1,2,3,4\}$. Let $\Gamma$ denote the Fisher information matrix, defined by
\[
\Gamma := \sum_{i=1}^d \bb{E}\qty[
\frac{\nabla_\theta \olamb{\thetastar,i}{}(0)\, \nabla_\theta \olamb{\thetastar,i}{}(0)^\top}{\olamb{\thetastar,i}{}(0)} ],
\]

and denote $\xi \sim  \s{N}(0, Id)$. Recall that $\MLE$ is defined by Equation \eqref{eq:mle}.

If $\int V_1(x)  \pi_0(dx) < \infty$, where $\pi_0$ is the law of $X(0)$ and $V_1$ is defined in Equation \eqref{eq:lyapunov_function}, the following holds. 
\[
\bb{E}\qty[\psi\qty(\sqrt{T} (\MLE - \theta^\star))]
 \xrightarrow[T \to \infty]{} 
\bb{E}\qty[\psi\qty(\Gamma^{-1/2} \xi)],
\]
for all $\psi\in D(\bb{R}^p, \bb{R})$, where $p$ is the dimension of $\MLE$.
\end{theo}
}

The theorem states a stronger result than the usual central limit theorem, since the functions $\psi$ can be polynomial and not only bounded. It should be noted that this convergence effectively implies convergence in distribution, by Prokhorov’s theorem. Indeed, the convergence of moments ensures that the sequence of measures is tight. Furthermore, any subsequence that converges in distribution must converge to the normal distribution, since the Gaussian distribution is uniquely determined by its moments. Thus, convergence in moments is a stronger version of the classical asymptotic normality that provides additional information.

The proof of this theorem is given in \ref{proof:AsymptoticNormality}. It relies on Theorem 3.14 of \cite{clinet2017statistical}, which provides conditions for the convergence of the MLE in the general context of point processes. These conditions are probabilistic properties of the process $\lamb{\theta}{}$ and its derivatives with respect to $\theta$. To better understand the result and its prerequisite, we present here a sketch of the proof of \cite{clinet2017statistical}.

\paragraph{Sketch of proof} The proof uses the theory of M-estimators, and proceeds in three steps: consistency, asymptotic normality, and convergence of moments. All steps involve the asymptotic behavior of the process
\begin{equation}
\label{eq:U}
U_{\theta,i}(t) := \qty(\lamb{\thetastar,i}{}(t),\, \lamb{\theta,i}{}(t),\, \nabla_\theta \lamb{\theta,i}{}(t)),
\end{equation}
which gathers the true intensity, the intensity associated with a parameter $\theta$, and its gradient.

\underline{Consistency.} It is the consequence of the fact that log-likelihood ratio, rescaled by $T$, converges uniformly in $\theta$ to a deterministic limit. To exhibit this limit, the likelihood is split into a martingale term and a drift term. The martingale term vanishes uniformly in $\theta$, thanks to Sobolev's inequality combined with the Burkholder--Davis--Gundy inequality. The drift term converges thanks to ergodicity of the process $U_{\theta , i}$. A non-degeneracy condition, namely that the limit vanishes only at $\thetastar$, then ensures that $\thetastar$ is its unique maximum, which yields consistency of $\MLE$.

\underline{Asymptotic normality.} Once consistency is known, one performs a Taylor expansion of the score $\theta\mapsto\frac{1}{\sqrt T}\partial_\theta l_T(\theta)$ between $\thetastar$ and $\MLE$. This results in two terms that need to be controlled. The first one, $\frac{1}{\sqrt T}\partial_\theta l_T(\thetastar)$, converges to a centered Gaussian law by a martingale central limit theorem. The second one, $\frac{1}{T}\partial^2_\theta l_T(\xi_T)$, with $\xi_T$ between $\MLE$ and $\thetastar$, converges to the Fisher information matrix $\Gamma$. This is again shown using Sobolev's inequality, Burkholder--Davis--Gundy, and ergodicity, but this time the convergence has to be uniform on a shrinking neighborhood of $\thetastar$, so that it still holds once $\xi_T$ is plugged in. Combining the two terms and inverting $\Gamma$ then gives the central limit theorem for $\MLE$.

\underline{Convergence of moments.}
The final layer of the proof of \cite{clinet2017statistical} strengthens the convergence in the CLT as it ensures that $\sqrt T(\MLE-\thetastar)$ converges in moment. It requires a stronger version of ergodicity, with a rate of convergence. In this set-up, the Quasi-Likelihood random field around $\thetastar$ satisfies a polynomial-type large deviation inequality, which controls the contribution of values of $\theta$ far from $\thetastar$ and upgrades the convergence in distribution into a convergence of moments for $\MLE$.

This sketch of proof highlight the required property, as it relies on the following points.
\begin{enumerate}[label=(P\arabic*)]
    \item \label{item:P1} The convergence of $U_{\theta,i}(t)$ towards a stationary version, at an exponential rate.
    \item \label{item:P2} An ergodic theorem for $U_{\theta,i}$, used for consistency and asymptotic normality.
    \item \label{item:P3} A mixing property with an explicit rate, uniform in $\theta$, used for the convergence of moments.
\end{enumerate}
These results, that constitute the core of the proof, are interesting properties on their own. For this reason, all those properties are stated in the following Section (Section \ref{subsec:ergodic_mixing}).

\subsection{Some useful results: asymptotic properties of the intensity}
\label{subsec:ergodic_mixing}

We start by extending the stationary behavior of the process proved in Proposition \ref{prop:stationary_intensity_true_process}, to the parametric family $\qty( \lamb{\theta}{})$, which shows condition \ref{item:P1}.

\begin{prop}
\label{prop:stationary_intensity_param_family}
Suppose Assumptions \ref{ass:ergodicitySDE}, \ref{ass:bounded_baseline3} hold, and define for $\theta \in\Theta$ and any $i \in \llbracket 1,d\rrbracket$, the stationary version of the intensity as
\begin{equation}
\label{eq:stationary_intensity_param}
\olamb{\theta,i}{}(t) 
= \bas{\mu}{i}(\ov{X}(t^-)) 
+ \sum_{j=1}^d \int_{(-\infty,t)} \alpha_{ij} e^{-\beta_{ij}(t-s)}   d\oN{j}{}(s),
\end{equation}
where  $\ov{X} = (\ov{X}(t))_{t\in\bb{R}}$ is distributed according to the stationary law $\pi_X$ of the SDE and where $\oN{j}{} = (\oN{j}{}(t))_{t\in\bb{R}}$ denote the stationary Hawkes process, defined in Equation \eqref{eq:stationary_intensity}.

\begin{enumerate}
    \item The following holds,
\[
    U_{\theta,i}(t) \xrightarrow[t \to +\infty]{\s{L}}\qty(\olamb{\theta^\star,i}{},  \olamb{\theta,i}{},  \nabla_\theta \olamb{\theta,i}{}).
\]
\item Moreover, there exist constants $c>0$ (independent of $\theta$) and $C>0$ (depending on $\theta$) such that,
\[
\bb{E}\Big[\big| \lamb{\theta,i}{}(t) - \olamb{\theta,i}{}(t) \big|\Big]
\le C  e^{-ct}.
\]
\end{enumerate}
\end{prop}

The proof is given in \ref{proof:prop_stationary_intensity}.
We now state the mixing properties of the process $U_{\theta,i}(t)$. These properties are useful for the theoretical analysis of the model and are required to establish the asymptotic normality of the maximum likelihood estimator (see Section \ref{section:MLE}). In particular, mixing implies ergodicity, which is necessary to ensure consistency and CLT for the MLE. Studying the mixing behavior of $U_{\theta,i}(t)$ therefore extends Theorem \ref{theo:exp_ergodicity} which establishes ergodicity only for the true intensity $\lamb{\thetastar}{}(t)$.
The mixing property, stated here thus proves \ref{item:P2} and \ref{item:P3}.

Formally, a process $\left(U(t)\right)_{t \ge 0}$ taking values in a space $E$ is said to be $C$-mixing, for a set of functions $C := \{ \phi : E \to \bb{R} \}$, if for any $\phi, \psi \in C$, 
\[
\rho_u:= \sup_{t \ge 0} \operatorname{Cov}\big[\psi(U(t)), \phi(U(t+u))\big] \xrightarrow[u \to \infty]{} 0.
\]

Depending on assumptions on the baseline functions and the covariate process $X(t)$, two distinct mixing regimes can be identified. Under moderate boundedness conditions, the process satisfies a weak mixing property with respect to continuous and bounded functions. When stricter boundedness assumptions are imposed on $\bas{\mu}{i}(X(t))$, a strong mixing property applies.

}

\begin{theo}[Mixing Properties of the Intensity]
\label{theo:mixing} 
The augmented intensity process $U_{\theta,i}(t)$, given in Equation \eqref{eq:U}, satisfies the following mixing properties:
\begin{itemize}
    \item \textbf{Weak Mixing:} If Assumptions \ref{ass:ergodicitySDE},\ref{ass:bounded_baseline2} hold then for all functions $\psi,\phi \in \s{C}_b(E, \bb{R})$ and for all $\theta \in \Theta$, we have for all $i$: 
    \[ 
   \rho_{\theta,i}(u) =   \sup_{t\geq 0} \cov \qty(\psi\qty(U_{\theta,i}\qty(t)),  \phi\qty(U_{\theta,i}\qty(t+u))) \xrightarrow[u \to \infty]{} 0 
    \]
    \item \textbf{Strong Mixing:} If Assumptions \ref{ass:ergodicitySDE},\ref{ass:bounded_baseline3} and \ref{ass:moment_g(X)}(p,1) are verified for all $p>1$, then and for $\psi, \phi \in  \s{D}(E, \bb{R})$, there exists $r<1$ such that for all $i$:
    $$\rho_{i}(u):= \sup_{\theta \in \Theta} \sup_{t\geq 0} \cov \qty(\psi\qty(U_{\theta,i}\qty(t)),  \phi\qty(U_{\theta,i}\qty(t+u))) = o\qty(r^u) $$
\end{itemize}
\end{theo}

The proof of this theorem, presented in \ref{proof:mixing}, follows the approach of \cite{clinet2017statistical} and \cite{potiron2025high}, who establish mixing properties for Hawkes models with constant baseline. The method is based on truncating the memory of the process to neglect the influence of the distant past, which becomes negligible due to the exponential decay of the memory kernel.

We obtain two types of asymptotic behavior, depending on the assumptions made on the covariate process $X$ and the baseline functions $\bas{\mu}{i}$. In all cases, the ergodicity of $X$ is required to ensure the stability of the process and to derive asymptotic results. The strong exponential mixing property holds under Assumptions \ref{ass:bounded_baseline3} and \ref{ass:moment_g(X)}(p,1), which require uniform boundedness of $\bas{\mu}{i}(X(t))$ and its derivative $\nabla_\mu \bas{\mu}{i}(X(t))$ in $L^p$ for all $p>1$. In the strong case, the mixing property is uniform in $\theta$, contrary to the weak case. Under the weaker Assumption \ref{ass:bounded_baseline2}, the process remains mixing, but the result is not uniform in $\theta$; however, this setting relies on weaker assumptions on the baseline functions and still ensures mixing for bounded test functions.

These two types of mixing properties are important because they can lead to different levels of convergence for the MLE (see \cite{clinet2017statistical}). Indeed, the weak mixing property is usually sufficient to establish the classical version of the asymptotic normality of the MLE. The strong mixing property allows for the control of the expectations of test functions in $ \s{D}(E,\bb{R})$, leading to moment convergence at all order, as stated in Theorem \ref{theo:AsymptoticNormality}.


\subsection{Discussion and relaxation of assumptions}
\label{sec:discussion_inference}
So far, we have made numerous assumptions, particularly in the context of Theorem \ref{theo:AsymptoticNormality}. We therefore devote this subsection to discuss some of these assumptions.

\subsubsection*{Boundedness condition}

All the stated theorems assume that the process $\bas{\mu}{i}\qty(X(t))$ is, in one way or another, bounded (see Assumptions \ref{ass:bounded_baseline1}, \ref{ass:bounded_baseline2}, \ref{ass:bounded_baseline3}). This condition generalizes a classical assumption of standard Hawkes processes, in which the constant representing the baseline belongs to a compact set of $\bb{R}_+^*$, bounded both from above and from below.
    
The main advantage of this assumption is that it facilitates the dissociation of the behavior of the Hawkes component from the SDE component of the model. This allows us to frame the behavior of the Hawkes component between two Hawkes processes with a constant baseline, the behavior of which has already been studied \cite{ogata1988statistical} \cite{clinet2017statistical}. Another major advantage of this assumption is that it ensures the intensity remains strictly positive at all times. Otherwise, it would be necessary to set an upper bound on $\sup_t \bb{E}\qty( \frac{1}{\lamb{\theta}{}(t)}\indic_{\lamb{\theta}{}(t)\ne 0} )$, which is far more complicated.  

From an application point of view, the relevance of this assumption depends on the context, but in practice many stochastic processes are naturally bounded by constraints. For example, in classical applications of Hawkes models: in finance, buying and selling prices are bounded both from below and from above due to money stocks; in ecology, SDE models describing animal movements are subject to physical constraints.

\subsubsection*{Identifiability conditions}
\label{section:identifiability}

We aim to clarify here some points raised in Assumption \ref{ass:identifiability_baseline}. Accurate parameter identification from observed intensity trajectories requires distinguishing between two sources of variation: the contribution of the variable $X(t)$ through the term $\bas{\mu}{i}(X(t))$, and that of the memory of past events $(Y_{ij}(t))_{ij}$.

The objective of the third point of Assumption \ref{ass:identifiability_elliptic} is therefore to ensure that these two intensity components do not somehow merge. For that, we exploit their structural difference: $\bas{\mu}{i}\qty(X\qty(t))$ evolves according to a SDE (thanks to the regularity of $g$), while $Y_{ij}$ are jump processes, piecewise deterministic and very regular between jumps (even $C^\infty$). If the two contributions became indistinguishable, this would mean that the diffusive part $\bas{\mu}{i}\qty(X\qty(t))$ no longer varies randomly, i.e., its quadratic variation would be zero. Now, this quadratic variation is expressed by:
\[
\nabla_x \bas{\mu}{i}\qty(X\qty(t)) \sigma\sigma^T(X(t)) \nabla_x \bas{\mu}{i}\qty(X\qty(t))^T.
\]
Assumption \ref{ass:identifiability_elliptic} then ensures that this expression can only be canceled if the function $\bas{\mu}{i}$ no longer depends on $X(t)$, which places us back in the usual framework of a Hawkes process with a constant baseline, a known and studied framework.

This remark underlines that what truly matters is not the non-degeneracy of $\sigma\sigma^T$, but rather the impossibility of confusing a stochastic diffusion with a deterministic jump process, suggesting that Assumption \ref{ass:identifiability_elliptic} can in fact be relaxed.  A more general but less intuitive way to formulate this idea is: if $\ps{}$ the diffusion process $ \bas{\mu}{i}(X(t^{-}))$ is equal to the jump process $\sum_j Y_{ij}(t)$ then both the diffusion and jump process are identically null.

This shows that the class of admissible SDEs is actually wider than initially considered, and can include some types of hypoelliptic diffusions i.e with a degenerate diffusion matrix $\sigma\sigma^T.$

\section{Model validation and statistical testing}
\label{section:test}

In this section, we present the statistical tests that follow from the asymptotic convergence of the likelihood estimator, established by Theorem \ref{theo:AsymptoticNormality}, as well as certain properties of the intensity of the model (such as the ergodic property).

\subsection{Testing the baseline parameters}
\label{section:test_param}
We first introduce tests designed to assess the value of the model coefficients. These tests rely both on the asymptotic normality of the estimators and on the existence of a consistent estimator of the Fisher information matrix. Both aspects are detailed in the following corollary.

\begin{coro}
\label{coro:test}
Let $\hat{I}_T$ be defined by either
\begin{equation}
\label{eq:Ihat}
-\frac{\nabla_\theta^2 \ell_T^N(\MLE)}{T}
\quad\text{or}\quad
\frac{1}{T}\sum_{i=1}^d \int_0^T \lamb{\MLE,i}{}(t)^{-2}\, \nabla_\theta \lamb{\MLE,i}{}(t)\, \nabla_\theta \lamb{\MLE,i}{}(t)^\top\, d\N{i}{}(t).
\end{equation}
Then, under the assumptions of Theorem \ref{theo:AsymptoticNormality}, $\hat{I}_T$ is a consistent estimator of $\Gamma$ and is asymptotically non-degenerate. As a consequence,
\[
\sqrt{T}\, \hat{I}_T^{1/2} \qty(\MLE - \thetastar) \limloiT \s{N}(0, I_p).
\]
\end{coro}
Both expressions in Equation \eqref{eq:Ihat} are closely related since, at the true parameter $\thetastar$,
\[
\bb{E}\qty(-\frac{\nabla_\theta^2 \ell_T^N(\thetastar)}{T})
= \frac{1}{T} \sum_{i=1}^d \bb{E}\qty(\int_0^T \lamb{\thetastar}{i}(t)^{-2}\, \nabla_\theta \lamb{\thetastar}{i}(t)\, \nabla_\theta \lamb{\thetastar}{i}(t)^\top\, d\N{i}{}(t)).
\]
However, the second expression is often preferred in practice, as it only requires the first-order derivatives of the intensity.

Corollary \ref{coro:test} allows us to build tests on the parameters of the model. This corresponds to an asymptotic test for the value of a parameter of a normal distribution with unknown variance.
{\modif 
\begin{algorithm}
\caption{Test for $\Hz: c^\top\thetastar=c_0$ vs $\Hu: c^\top\thetastar\neq c_0$}
\label{Test:coeff_wald}
    \begin{algorithmic}
    \Require Sample $\s{S}:= \{ (t, X(t)): t \in \s{T} \}$, with $\s{T}$ such that $(T_i)_{i \in \llbracket 1, N(T)\rrbracket} \subset \s{T}$
    \begin{enumerate}
        \item Compute the MLE $\MLE$;
        \item Compute $\hat{I}_T$ as in Equation \eqref{eq:Ihat};
        \item Compute the test statistic
        \[
        \s{Z}_c := \frac{\sqrt{T}\,\qty(c^\top\MLE - c_0)}{\sqrt{c^\top\hat{I}_T^{-1}c}} ;
        \]
        \item Reject $\Hz$ if $\abs{\s{Z}_c} > q_{1-\alpha/2}$, where $q_{1-\alpha/2}$ is the $1-\alpha/2$ quantile of the standard Gaussian distribution.
    \end{enumerate}
    \end{algorithmic}
\end{algorithm}

Algorithm \ref{Test:coeff_wald} describes the procedure for a two-sided test which is of asymptotic level $\alpha$. To test a one-sided alternative, one can use the same statistic and reject $\Hz$ when $\s{Z}_c>q_{1-\alpha}$ (right-sided test) or when $\s{Z}_c<q_{\alpha}$ (left-sided test).
This generic procedure covers two cases.
\begin{itemize}
    \item \emph{Value of one coefficient.} With $c=e_k:=(0,\ldots,0,1,0,\ldots,0)^\top$, the $k$-th element of the canonical basis, one can test a given value of a coefficient of $\thetastar$. Namely, the null hypothesis is $\Hz:\thetastar_k=c_k$, for a given $c_k$.
    \item \emph{Equality of two coefficients.} With $c=e_k-e_l$, where $e_k$ and $e_l$ are the $k$-th and $l$-th elements of the canonical basis, and $c_0=0$, one can test $\Hz:\thetastar_k=\thetastar_l$.
\end{itemize}

Although this type of procedure is not new (see, for example, \cite{bonnet2024testing}), its interest here lies in its application to our model: testing whether the baseline depends on the covariates. To illustrate this, we consider the following example, which is then studied in detail in Section \ref{section:simu}. Suppose that the family of baselines used during estimation is of the form

\[
    \bas{\mu}{i}(x) = \mu_{i,2} + (\mu_{i,1}-\mu_{i,2})\exp\qty(-5\norm{x-x_1}^2), \qquad i\in\{1,2\},
\]

where $x_1\in\bb{R}^2$ is fixed. The baseline of the $i$-th component does not depend on $X$ if and only if $\mu_{i,1}=\mu_{i,2}$. Testing $\Hz:\mustar_{i,1}=\mustar_{i,2}$ with Algorithm \ref{Test:coeff_wald} therefore allows to test whether the covariate has an effect on the $i$-th component. If $\Hz$ is not rejected, we retrieve a Hawkes process with constant baseline. Otherwise, the covariate modulates the event rate of this component.
}


\subsection{Test on the model} 

We now introduce a test aimed at assessing the adequacy of the proposed model.  
A classical approach for evaluating the fit of a point process model relies on the time-change theorem \cite{papangelou1972integrability}, which states that the increments of the true compensator follow an i.i.d.\ exponential distribution with unit mean.  
Let $(T_i)$ denote the jump times of the process $N$. The theorem ensures that the random variables $E_i$ defined as
\[E_i:= \Lambda_{\thetastar}(T_i) - \Lambda_{\thetastar}(T_{i-1}), \qquad i \in \bb{N},
\]
are i.i.d. from exponential distribution with rate $1$, that we denote $ \s{E}(1)$.

This motivates a goodness-of-fit procedure based on a plug-in version of this property.  
If the estimated intensity $\lambda_{\widehat{\theta}_T}(t)$ is close to the true one, then the transformed increments
\[
\widehat{E}_i:= \Lambda_{\widehat{\theta}_T}(T_i) - \Lambda_{\widehat{\theta}_T}(T_{i-1})
\]
should be approximately exponentially distributed with parameter 1. Yet there was no theoretical guarantee for this test, and empirical evidence suggests that it does not perform well: the resulting p-values tend to be too large, leading to overestimation of model adequacy (see, e.g., \cite{bonnet2024testing}).

A first theoretical justification of this phenomenon was proposed in \cite{baars2025asymptotically} which proved that the process $\qty( \frac{1}{\sqrt{T}}\qty(\Lamb{\MLE}{}(uT) - \Lamb{\thetastar}{}(uT)) )_{u \in [0,1]}$ converges to the process $\qty( u \bb{E}\qty(\nabla_\theta \olamb{\theta}{}) \Gamma^{-1/2}  \s{N}(0, I_d))_{u \in [0,1]}$. Building on this result, \citep{baars2025asymptotically} also proposed a test for model adequacy. Although this test is theoretically valid in our context, its practical use remains limited. Indeed, the mentioned test statistics involve computing the compensator between two points, $\tau_{i-1}$ and $\tau_i$, which are not jump points. In our case, the compensator itself must already be approximated, due to the term $\int \bas{i}{\mu}(X(t))dt$, which introduces numerical errors and can be computationally demanding. Therefore, computing the integral of the compensator adds yet another source of error and can be quite long. As a result, making the test operational would require both a large amount of data and a long observation window, which is not feasible in practice.

We now turn to the following theorem, which extends the previous framework by introducing a correction to the compensator increments. Unlike earlier results, it establishes that these corrected increments converge to i.i.d.\ exponential random variables. In particular, the theorem provides a theoretical guarantee ensuring the convergence of an associated estimator under the null hypothesis $\Hz$.

\begin{theo}
\label{theo:gof_corrige}
Suppose Assumptions \ref{ass:ergodicitySDE},  \ref{ass:bounded_baseline3}, \ref{ass:regularity_g}(4), and \ref{ass:identifiability_baseline} hold, together with Assumption \ref{ass:moment_g(X)}$(p,j)$ for all $p>1$ and $j \in \{1,2,3,4\}$, and assume $\int V_1(x)  \pi_0(dx) < \infty$, where $\pi_0$ is the law of $X(0)$ and $V_1$ is defined in Assumption \ref{ass:ergodicitySDE}.

Let $\wN{}{}$ and $\N{}{}$ be two independent stationary multivariate Hawkes processes with stationary law $\N{}{}$, and denote by $(\wT{j}{})$ and $(\T{j}{})$ their associated jump times. Define for all $ i \indexset{1}{d}$ the stationary intensities
\begin{align*}
    &\wlamb{\theta,i}{}(t) = \bas{\mu}{i}(\widetilde{X}(t)) + \sum_{j=1}^d \int_{-\infty}^{t^-} \alpha_{ij} e^{-\beta_{ij}(t-s)}  d\wN{j}{}(s), \\
    & \lamb{\theta,i}{}(t) = \bas{\mu}{i}(X(t)) + \sum_{j=1}^d \int_{-\infty}^{t^-} \alpha_{ij} e^{-\beta_{ij}(t-s)}  d\N{j}{}(s),    
\end{align*}
where $X, \widetilde{X} \sim \pi_X$ are distributed according to the stationary law of the SDE. Define the MLE $\wMLE$ computed using only the process $\wN{}{}$. The corrected increments are defined as follows
\begin{equation}
\label{eq:Eicorrected}
    \w{E}_i(T):= \Lamb{\wMLE}{}(T_i) - \Lamb{\wMLE}{}(T_{i-1}) + \frac{1}{\sqrt{T}} (T_i - T_{i-1})   \widetilde{\chi}_T   \widetilde{I}_T   \widetilde{X},
\end{equation}

with 
$$ \widetilde{\chi}_T:= \frac{1}{T} \int_0^T \nabla_\theta \wlamb{\wMLE}{}(t)   dt,    \text{ } \widetilde{I}_T:= \frac{1}{T} \sum_{i=1}^d \int_0^T \wlamb{\wMLE,i}{-2}(t) \nabla_\theta \wlamb{\wMLE,i}{}(t) \nabla_\theta \lamb{\wMLE,i}{}(t)^\top   d\wN{i}{}(t).$$

Then,
$$(\w{E}_i(T))_{i\in \bb{N}} \xrightarrow[T \to \infty]{ \s{L}}\bigotimes_{j=1}^\infty  \s{E}(1),$$

that is, the sequence of compensated increments converges in law to i.i.d.\ standard exponential random variables.
\end{theo}

The proof of this Theorem can be found in \ref{proof:BaarsTest}. The assumptions required in Theorem \ref{theo:gof_corrige} coincide with those used to establish moment convergence. They are needed here to control the expectations of the estimators of the Fisher information matrix and of $\widehat{\chi}_T$. In addition, ergodicity of the process {\modif ensures} the convergence of time averages toward their expectations, which is essential for the asymptotic validity of the test.

It is worth noting that Theorem \ref{theo:gof_corrige} assumes two independent samples $(N, \widetilde{N})$ for technical reasons. This assumption simplifies the theoretical proof, since it avoids dealing with the joint law of the estimator $\MLE$ (which depends on $N$) and the jump times $(T_i)$ themselves. 

In practice, we are mainly interested in situations with few observations—sometimes as few as one, as is common in neuroscience or biology. When many independent repetitions are available, as in some financial applications, the existing test of \cite{reynaud2014goodness} can be applied. Here, we use the same observed realization of $N$ for both estimation and testing, and we show that the correction remains numerically valid even in this single-observation setting.

The corresponding testing procedure presented in Algorithm \ref{algo:GofCorr}.

\begin{algorithm}[!htbp]
\caption{Corrected GOF procedure for testing $\Hz:  N \stackrel{ \s{L}}{=} \N{\thetastar}{} $ for some $\thetastar \in \Theta$}
\label{algo:GofCorr}
\begin{algorithmic}
\Require Sample $\s{S}:= \{ (t, X(t)): t \in  \s{T} \},$ with $\s{T}$ such that $ (T_i)_{ i \in \llbracket 1, N(T)\rrbracket } \subset \s{T}$
\begin{enumerate}
    \item Compute $\MLE$ MLE estimator of $\thetastar$.
    \item Compute 
    \begin{align*}
       & \widehat{\chi}_T:= \frac{1}{T} \int_{(0,T)} \nabla_\theta \lamb{\MLE}{}(t) dt, \\
        &\widehat{I}_T:= \frac{1}{T} \int_{(0,T)} \nabla_\theta \lamb{\MLE}{}(t)
     \nabla_\theta \lamb{\MLE}{}(t)^\top 
    \lamb{\MLE}{}(t)^{-2}  d\N{}{}(t).
    \end{align*}
    \item Compute the compensated increments
    $
    \Delta_i \Lamb{\MLE}{}:= 
    \int_{(T_{i-1},T_i)} \lamb{\MLE}{}(t) dt.
    $
    \item Draw $x \sim  \s{N}(0,I_d)$ and form the corrected variables
    $
    \hat{E}_i(T):= 
     \Delta_i \Lamb{\MLE}{}
    + \frac{1}{\sqrt{T}}(T_i-T_{i-1})   
    \widehat{\chi}_T \widehat{I}_T^{-1/2} x.
    $
    \item Compare $\qty(\w{E}_i(T))_{i \indexset{1}{N(T)-1} }$ to an exponential distribution of parameter 1.
\end{enumerate}
\end{algorithmic}
\end{algorithm}
This goodness-of-fit test is illustrated in the following section.

\section{Numerical illustration}
\label{section:simu}
The aim of this section is to illustrate the behavior of the considered point process, as well as its asymptotic convergence, while highlighting some observations related to the relaxation of assumptions on the SDE. With this in mind, we present three simulation frameworks.

\subsection{Benchmark Simulation: Well-Specified OU-Driven Hawkes Process}

We first focus on a two-dimensional Hawkes process driven by a bivariate Ornstein-Uhlenbeck process $X(t)$. This example satisfies all the conditions required by the theorems presented above, and its main purpose is to illustrate the properties previously demonstrated.

\paragraph{Simulation set-up} We consider a two-dimensional Ornstein–Uhlenbeck process for the covariate $X(t)$. Namely, we take $dX(t) = -\xi X(t) dt + \sqrt{2|\xi|} dW(t)$, with $\xi^\star = 0.05$ and initial condition $X_0 = (0, 0) \in \bb{R}^2$. The kernel parameters are $\alphastar = \begin{pmatrix} 0.3 & 0.4 \\ 0.5 & 0.4 \end{pmatrix}$ and $\betastar = (0.8, 1.5)$. Baseline functions for the Hawkes process are $\bastar{\mu}{1}(x) = \mustar_{12} + (\mustar_{11} - \mustar_{12}) \exp(-5 \norm{x - x_1}^2)$ and $\bastar{\mu}{2}(x) = \mustar_{22}$ with $x_1 = (0.1, 0.1)$, $\mustar_{11} = 0.5$, $\mustar_{12} = 0.8$ and $\mustar_{22} = 0.7$.

\paragraph{Estimation set-up}  For the estimation, we consider the following class of baseline functions: $\bas{\mu}{i}(x) = \mu_{i,2} + (\mu_{i,1} - \mu_{i,2}) \exp(-5 \norm{x - x_1}^2) $ where $x_1 = (0.1, 0.1)$. Hence, the estimation procedure allows for the possibility that the second component of the Hawkes process also depends on the covariate $X$ (as the estimation phase can lead to a non-constant $\bas{\hat{\mu}}{2}$).  Let us recall that we do not seek to compute the parameter defining the diffusion, but only $\MLE$.

After estimation, we test the null hypothesis 
$\Hz: \mustar_{i,1} = \mustar_{i,2}$ for $i \in \{1,2\}$ 
using Algorithm \ref{Test:coeff_wald}, in order to evaluate the dependence of each process on the covariate.

\begin{figure}[!htbp]
    \centering
    \begin{subfigure}[b]{0.32\textwidth}
        \centering
        \includegraphics[width=\textwidth,draft=false]{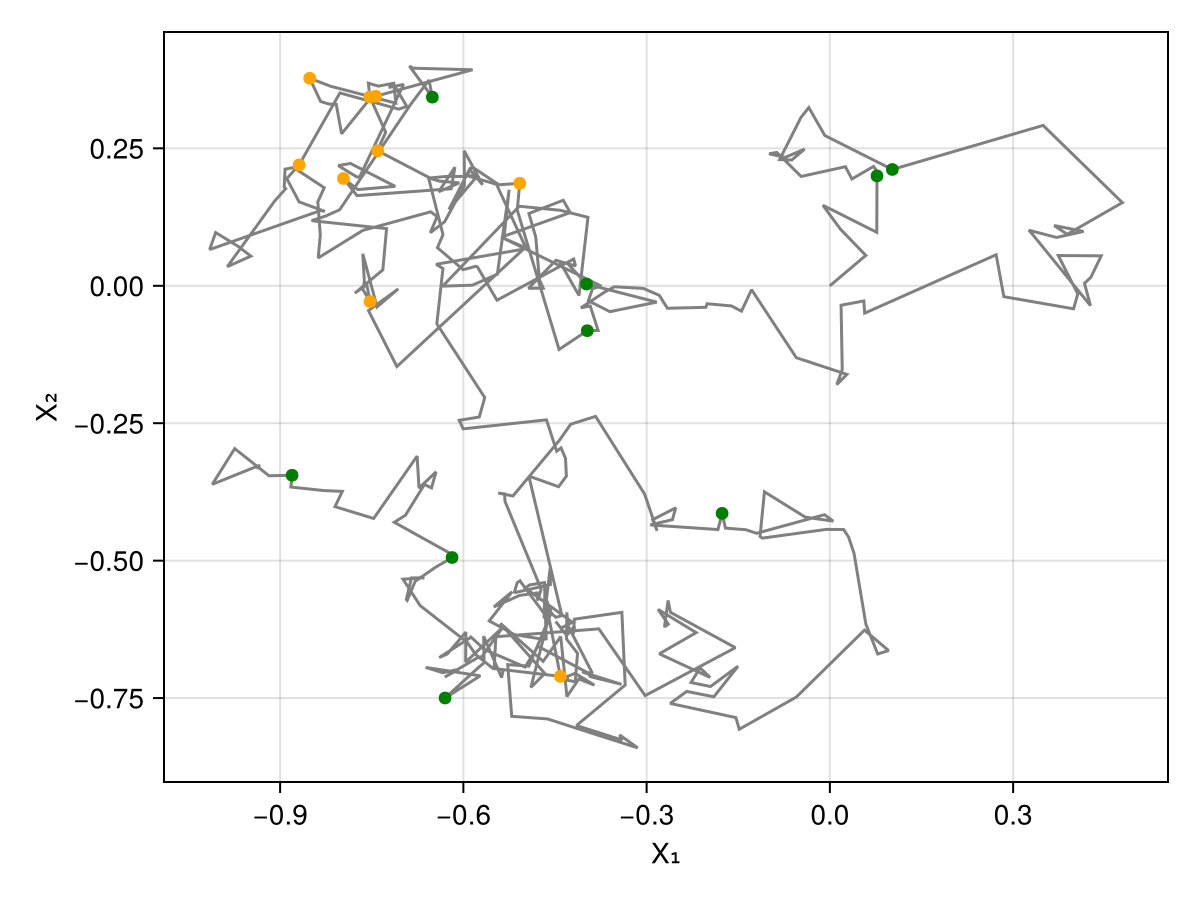}
        \caption{Example of trajectory of $X$}
        \label{fig:simu1Process}
    \end{subfigure}
    \begin{subfigure}[b]{0.32\textwidth}
        \centering
        \includegraphics[width=\textwidth,draft=false]{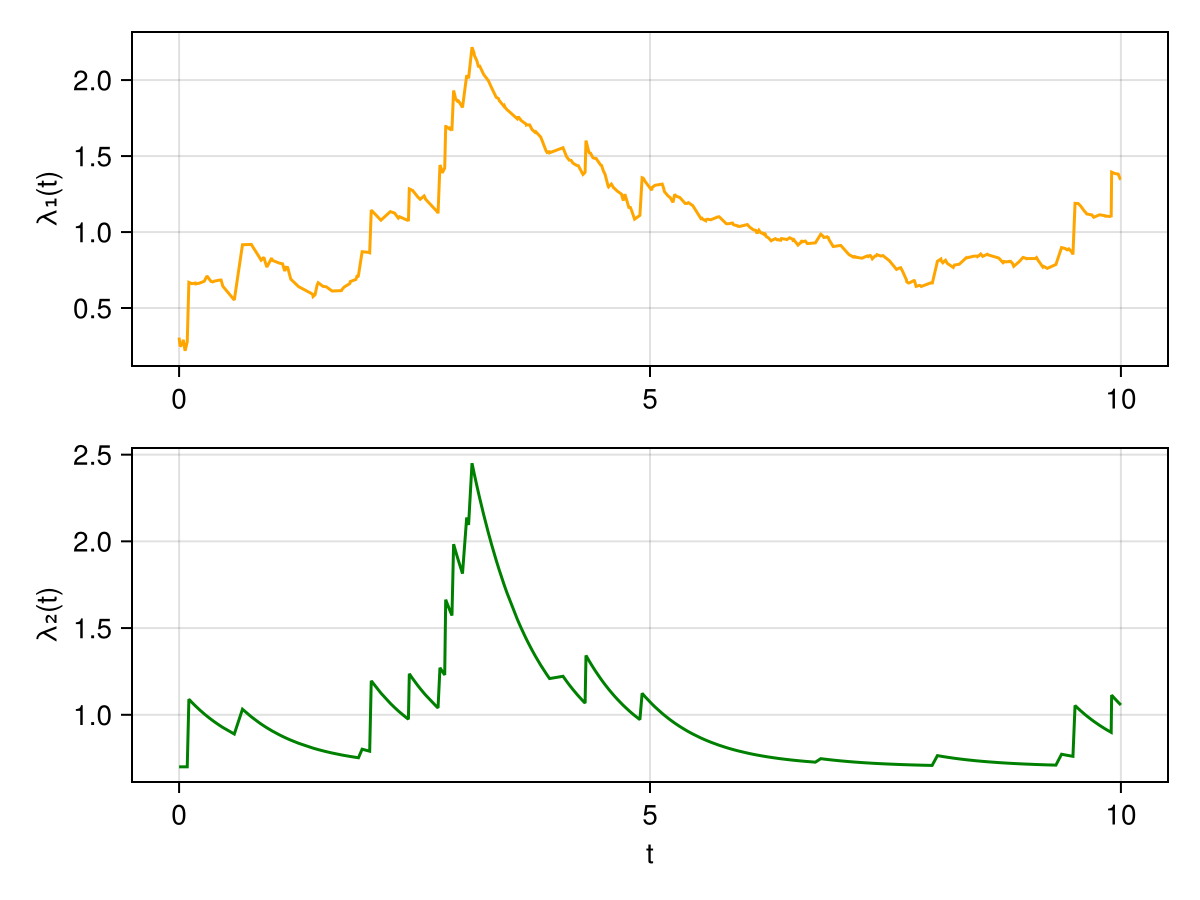}
        \caption{Intensity process $\lamb{}{}(t)$}
        \label{fig:simu1Intensity}
    \end{subfigure}
    \begin{subfigure}[b]{0.32\textwidth}
        \centering
        \includegraphics[width=\textwidth,draft=false]{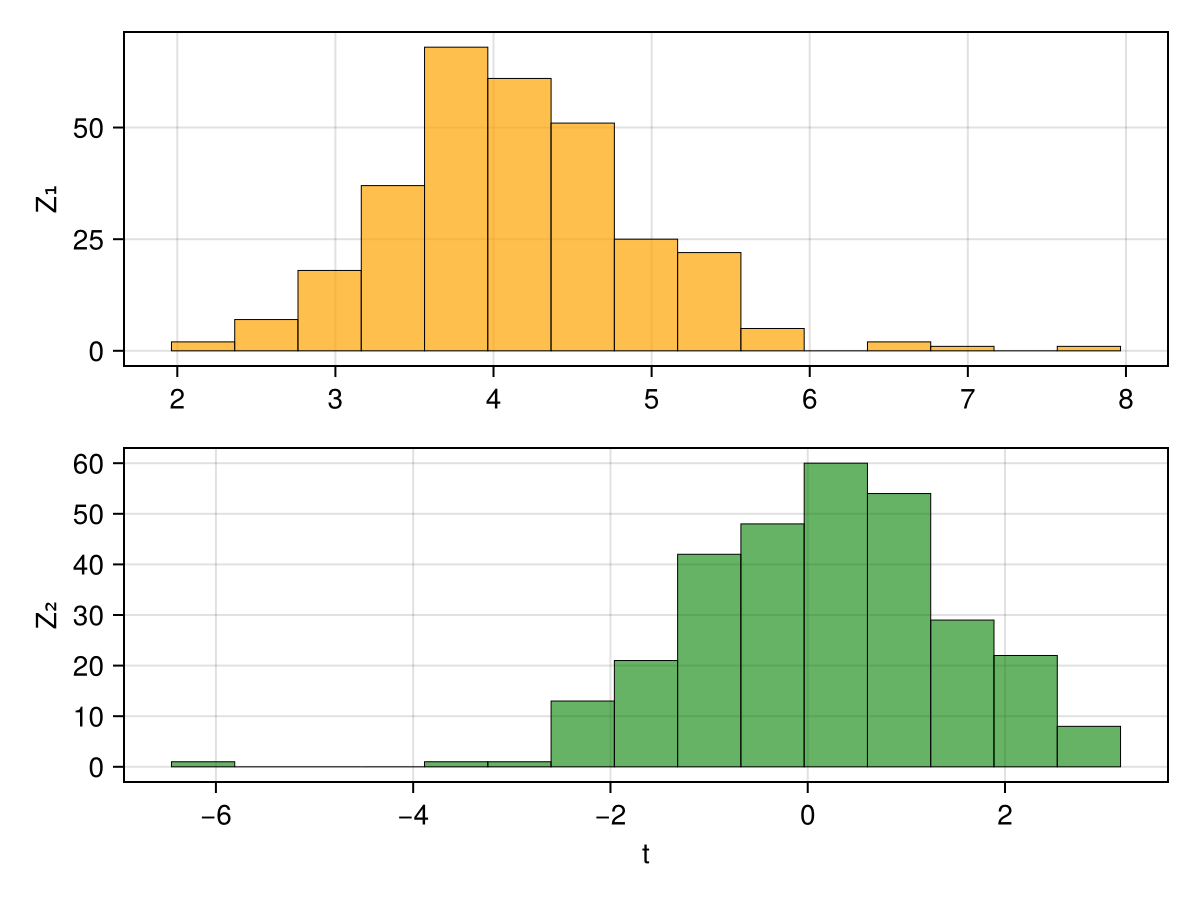}
        \caption{Test statistic distribution}
        \label{fig:simu1Stat}
    \end{subfigure}
    \caption{Trajectory of the covariate process $X$, with colored dots indicating event times —yellow for $N_1$, green for $N_2$— along the path (Figure \ref{fig:simu1Process}), together with the associated intensity functions (Figure \ref{fig:simu1Intensity}). 
    Figure \ref{fig:simu1Stat} shows the distribution of the test statistic for 
    $\Hz: \mustar_{i,1} = \mustar_{i,2}$ with $i \in \{1,2\}$.  For Figure \ref{fig:simu1Process}--\ref{fig:simu1Intensity}  $T= 10$ whereas for Figure \ref{fig:simu1Stat} $T=3000$ and $n=300$ repetitions were used.}

    \label{fig:Simu1}
\end{figure}
Figure \ref{fig:simu1Process},\ref{fig:simu1Intensity} shows both the trajectory of the covariate process $X$ and the intensity functions of the bivariate Hawkes process. This figure illustrates the influence of $X$: the intensity of process $N_1$ exhibits larger irregularity due to its dependency on the covariate, while the intensity of $N_2$, whose baseline function is constant, corresponds to the classical behavior of a Hawkes process without covariates.  

Figure \ref{fig:Simu1} also displays the distribution of the test statistic under the null hypothesis 
$\Hz: \mustar_{i,1} = \mustar_{i,2}$, with $i \in \{1,2\}$, illustrating the asymptotic normality established in Theorem \ref{theo:AsymptoticNormality}. As shown in Figure \ref{fig:simu1Stat}, both estimators are asymptotically normal but not centered at the same value. The statistic $Z_2$, which corresponds to the test for $N_2$, is centered at zero, which is consistent with the fact that the simulations were performed without any dependence between the process and the covariate. In contrast, $Z_1$ is not centered at zero, reflecting the dependence of the baseline of process $N_1$ on $X$.

\subsection{Stress-Test Simulation: Hypoelliptic Diffusion via the Kramers Oscillator}
\label{subsec:simu_HypoDiffusionhawkes}

In this example, we are interested in hypoelliptic diffusion, i.e., diffusion with a degenerate diffusion matrix. This case aims at illustrating remark \ref{section:identifiability} and shows that estimation and testing can be performed in this setting.

\paragraph{Simulation set-up}

The covariate is a two-dimensional process $(X(t), V(t))$ evolving according to the Kramers oscillator dynamics (see \cite{pilipovic2025strang}): 
\begin{align*}
    & dX(t) = V(t) dt \\
    & dV(t) = (-\eta_\star V(t) + a_\star X(t) - b_\star X(t)^3) dt + \sqrt{\sigma_\star^2} dW(t) , 
\end{align*}
 with parameters $\eta_\star = 0.65$, $a_\star = 1$, $b^\star = 0.6$, $\sigma_\star^2 = 0.1$, and initial condition $X_0 = (1.0, 0.0)$. The Hawkes process is a one-dimensionnal Hawkes process with baseline function being form $g_{\mustar}(x,v) = \mustar_1 + \mustar_2 \tfrac{(1 - x)^2}{2}$, with $\mustar_1 = 0.2$ and $\mustar_2 = 0.5$. The kernel parameters are $\alphastar = 0.6$ and $\betastar = 2.0$.

\paragraph{Estimation set-up} For the estimation phase, we consider the following class of baseline functions: $g_{\mu}(x,v) = \mu_1 + \mu_2 \tfrac{(1 - x)^2}{2}$.

\begin{figure}[!htbp]
    \centering
    \includegraphics[scale=0.40,draft=false]{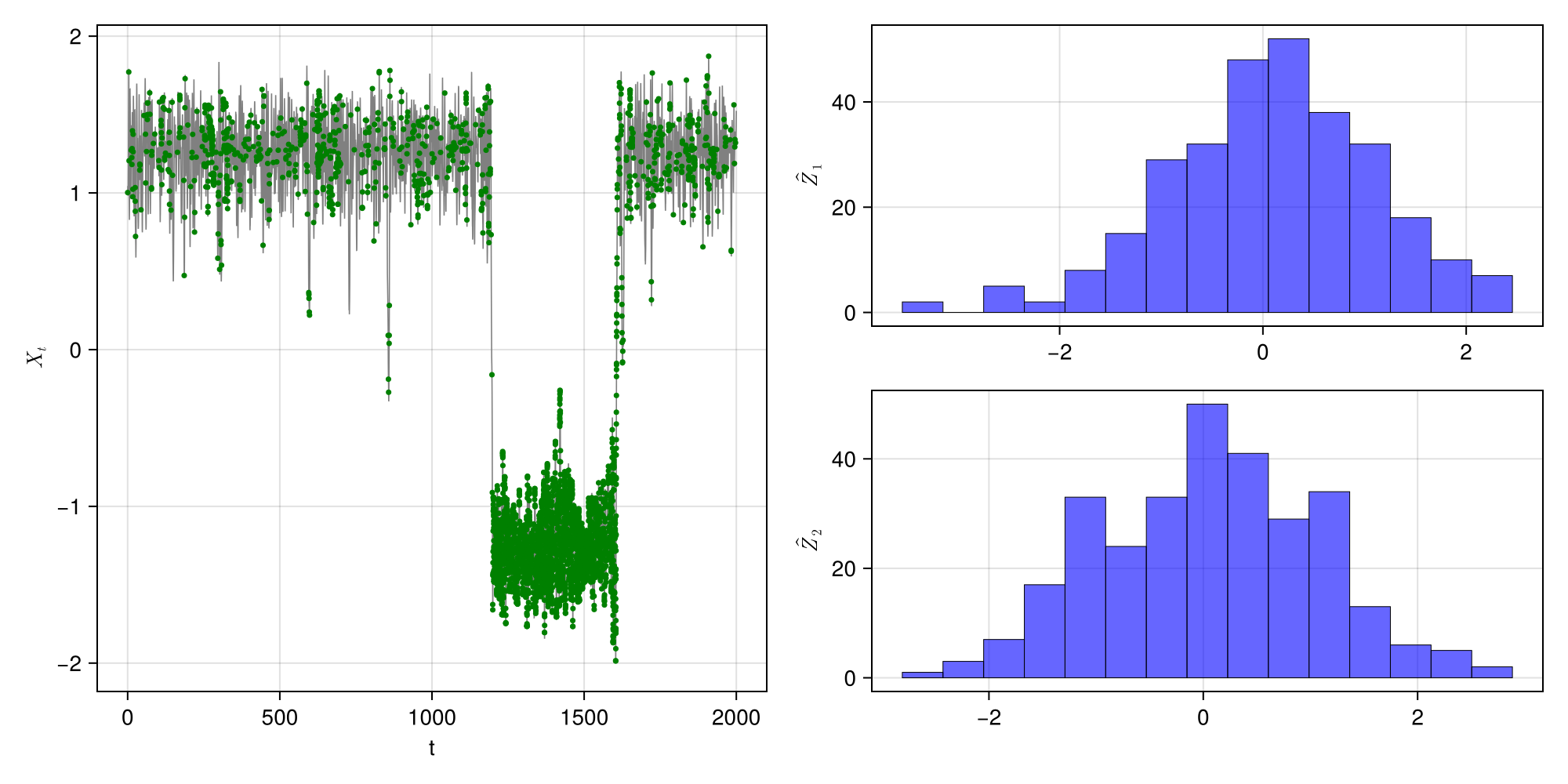}
    \caption{Left: trajectory of the process $X$ over time, with green dots indicating the event times for $N_1$ along the path. Right: histograms of the estimators under the null hypotheses $\Hz: \mu_1^\star = 0.2$ (top) and $\Hz: \mu_2^\star = 0.5$ (bottom). For this simulation $T= 2000$ and $n=300$ repetitions were used.}
    \label{fig:Process2Estimation}
\end{figure}

On the left graph of Figure \ref{fig:Process2Estimation}, we observe the first component of the process $(X,V)$, as well as the event times of the point process $N$, represented by green dots. It is clear that events tend to occur when $X(t)$ moves away from the level $x=1$, which is consistent with the shape of the baseline function. On the right graph, the histograms of the estimators associated with the coefficients $\mu_1$ and $\mu_2$ exhibit a Gaussian shape, despite the fact that the underlying SDE process does not satisfy all the conditions stated in the assumptions of Theorem \ref{theo:AsymptoticNormality}. 

To better understand this behavior, note that although the diffusion system has a degenerate diffusion matrix, it is still ergodic (see, e.g., \cite{mattingly2002ergodicity, pilipovic2025strang}). This degeneracy, however, prevents the  identifiability conditions (e.g Assumption \ref{ass:identifiability_elliptic}) from being met. As explained in Section \ref{section:identifiability}, this does not affect the asymptotic behavior of the estimators, and the normal-approximation distribution remains reasonable in this setting, as illustrated in Figure \ref{fig:Process2Estimation}.

\subsection{Benchmark simulation: Goodness-of-Fit Test for Model Adequacy}

In this section, we illustrate the behavior of the goodness-of-fit procedure in the context of Algorithm \ref{algo:GofCorr}. We focus here on the impact of the estimation step on the empirical distribution of the corrected inter-event times $(\widehat{E}_i(T))$.

\paragraph{Simulation set-up}  We consider a two-dimensional Ornstein–Uhlenbeck process for the covariate $X(t)$. Namely we take  $dX(t) = -\xi X(t) dt + \sqrt{2|\xi|} dW(t)$, with $\xi^\star = 0.05$ and initial condition $X_0 = (0, 0) \in \bb{R}^2$. The jump component is one-dimensional with kernel parameters $\alphastar = 0.8$ and $\betastar = 0.9$. The baseline function for the Hawkes process is $\bastar{\mu}{1}(x) = \mustar_{12} + (\mustar_{11} - \mustar_{12}) \exp(- \norm{x - x_1}^2)$ with $x_1 = (0.1, 0.1)$, $\mustar_{11} = 0.5$, $\mustar_{12} = 1.0$.

\paragraph{Estimation set-up}  For the estimation, we consider the following class of baseline functions: $\bas{\mu}{i}(x) = \mu_{i,2} + (\mu_{i,1} - \mu_{i,2}) \exp(-  \norm{x - x_1}^2) $ where $x_1 = (0.1, 0.1)$.

\begin{figure}[!htbp]
    \centering
    \includegraphics[scale = 0.5,draft=false]{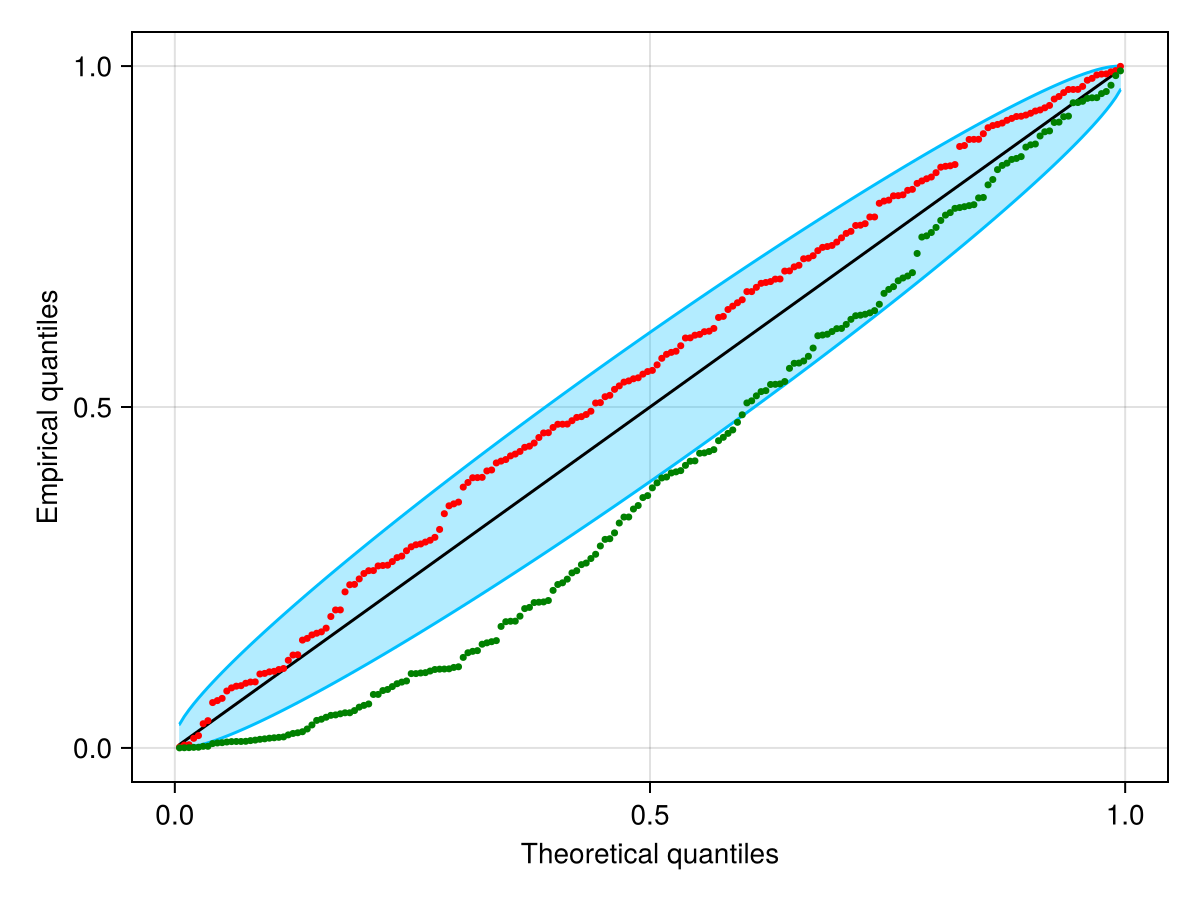}
    \caption{Illustration of Algorithm \ref{algo:GofCorr}. Green points show the distribution of p-values when the corrected increments $(\w{E}_i(T))$, defined in Equation \eqref{eq:Eicorrected}, are tested directly against a unit exponential distribution using a KS test. Red points correspond to p-values obtained after randomly subsampling a subset of size $N(T)^{2/3}$ from $(\w{E}_i(T))$ and applying the KS test to this subset. The p-values are compared to a standard uniform distribution using a QQ plot: the theoretical quantiles of the uniform distribution are shown as a black line, with 95\% confidence intervals represented as the blue region.}
    \label{fig:subsampling_KS}
\end{figure}

In Figure \ref{fig:subsampling_KS}, we present the distribution of the p-values obtained from the test described in Algorithm \ref{algo:GofCorr} and examine whether they follow a uniform distribution, as expected under correct model specification.
The test results, represented by the green dots, show that we reject the hypothesis of a uniform law for the p-values. This suggests that the test statistic converges at a slower rate than the expected $\sqrt{N(T)}$. The issue therefore lies in a mismatch between the convergence rate of the test statistic and that of the Kolmogorov–Smirnov test. Addressing this properly would require the development or the study of nonparametric tests, which is beyond the scope of the present work. For this reason, we instead focus on a simple and practical adjustment that restores a test with an appropriate rejection rate. 

One way to correct this problem is to subsample the $(\w{E}_i(T))$ before applying the nonparametric test. This approach relaxes the critical limits of the test and restores the uniformity of the p-values under $\Hz$. This idea is not new and has already been used in the literature on goodness-of-fit testing. For instance, \cite{baars2025asymptotically} introduced a clustering procedure in which a hyperparameter $c$ controls the aggregation of points before applying the KS test to the resulting clustered increments. Our approach can be viewed as a simplified variant of this method, where random subsampling replaces explicit clustering. Another related example is Test 4 in \cite{reynaud2014goodness}, which evaluates model adequacy for point processes observed over multiple independent repetitions. In their setting, subsampling is performed across repetitions, and the authors show that when the ratio between the subsample size and the total number of repetitions tends to zero, the p-values recover uniformity under $\Hz$.

Figure \ref{fig:subsampling_KS} illustrates the impact of our subsampling procedure on the distribution of p-values. In particular, the red dots represent the distribution of p-values obtained when a random subsample of size $N(T)^{2/3}$ is drawn from $(\w{E}_i(T))$, and the Kolmogorov-Smirnov test is applied to this reduced sample. It should be noted that the size of the subsample is chosen arbitrarily, as there is no clear theoretical argument for this particular exponent. Nevertheless, this choice was made in light of the recommendations in \cite{reynaud2014goodness}, since our subsampling procedure is based on the same principle.

\section{Conclusion and discussion}

{\modif
In this work, we introduced a generalization of the Hawkes process in which the baseline intensity is driven by a covariate process, modeled as the solution of an SDE. We extended several results known for standard Hawkes processes to this setting: exponential ergodicity of the intensity, convergence towards a stationary version, a law of large numbers for the counting process, and mixing properties of the intensity. Building on these results, we studied parametric inference in the long-time regime and established the consistency and asymptotic normality of the maximum likelihood estimator, together with the convergence of its moments. Finally, we proposed testing procedures to assess the relevance of the
covariate and the adequacy of the model, and illustrated these results through simulation studies.

The assumptions made in this work open several perspectives. A central element of our framework is the exponential kernel, which ensures that the joint process $Z=(X,Y)$ is Markovian and allows us to rely on the ergodic theory of Markov processes developed in \cite{meyn1993stability}. However, not all of our results depend on this structure. In the probabilistic part, the convergence of the intensity towards its stationary version  and the law of large numbers hold for general kernels satisfying the stability condition: their proofs only use this condition and the ergodicity of $X$. The exponential form is used to obtain the ergodicity of $Z$ and the exponential rates of convergence. For the inference part, the mixing properties of the intensity, on which the asymptotic study of the MLE relies, are based on a truncation of the memory of the process, whose effect vanishes because the kernel decays exponentially. Similar truncation arguments have been used for Hawkes processes with general kernels (see Lemma 13 of \cite{potiron2025high}), which suggests that part of our results could be extended beyond the exponential case, at the cost of additional assumptions on the decay of the kernel.

A similar remark applies to the model for the covariate. The choice of an SDE is mainly motivated by the application in neuroscience described in the introduction, where $X$ represents the position of the animal and the intensity describes the activity of a place cell. This choice also facilitates the theoretical analysis, since it fits within the framework of Markov processes, and since the required assumptions can be checked for many classical diffusion models. Yet, most of our arguments do not rely on the specific form of the SDE, but on the properties it induces. As detailed in Section \ref{sec:discussion_assumption}, the ergodicity of $Z$ only requires $X$ to be a Markov process satisfying a Lyapunov drift condition, a $T$-chain property and an accessibility condition. Moreover, the proof of the law of large numbers only uses that $X$ admits a stationary distribution and satisfies an ergodic theorem. Our results could therefore be stated for more general covariate processes, including some hypoelliptic diffusions, as suggested by the simulations of Section \ref{section:simu}.
}

Another important perspective raised by our study concerns the use of goodness-of-fit tests for corrected increments. One perspective would therefore be to characterize the convergence rate of the test statistic, which would allow the KS test to be adjusted or develop new nonparametric procedures that are adapted to the test of an asymptotic distribution (in contrast to the KS-test which is non-asymptotic).

\subsection*{Acknowledgements}
This work is part of the 2022 DAE 103 EMERGENCE(S) - PROCECO project supported by Ville de Paris.

This paper was completed while C. Dion-Blanc was affiliated
with the Centre de Math\'ematiques Appliqu\'ees (CMAP) at Ecole Polytechnique. She is grateful to Sylvie 
M\'el\'eard for the opportunity to join her team, supported by the European Union (ERC, SINGER, 101054787).


\bibliographystyle{elsarticle-num}
\bibliography{biblio.bib}

\newpage

\appendixheaderon
\begin{changemargin}{-2.5cm}{-2cm}

\begin{appendix} 

\renewcommand{\theass}{\Alph{section}.\arabic{ass}}


\appendixheaderon
\section{Notations ans Technical results}
\label{section:proof}
This section is dedicated to the proofs of the theoretical results set out in this article.

{\modif 
\subsection{Notation}
\label{paragraph:notation}
We begin the proofs by introducing some additional notation and a few preliminary results that will be used throughout the proofs.

First, for the point process $\N{i}{}$, we denote by $M_i$ the stochastic process defined by $M_i(t) = N_i(t) - \int_{0}^{t}\lamb{\thetastar;i}{}(s)ds$ and $M$ denotes the vector $(M_i)$. We denote $\norm{.}_{l_p}$ the p-norm on $\bb{R}^d$ whereas $\norm{.}_{\s{L}^p}$ is the p-norm for a function in $L^p$. 

For two functions $f,g : \bb{R} \to \bb{R}$, we denote $(f\ast g)(t) = \int_{0}^{t^-} f(t-s) g(s)ds$.

More generally, for two matrix-valued functions $f = (f_{ij})_{1\le i,j\le d}$ and $g = (g_{ij})_{1\le i,j\le d}$, we define their (matrix) convolution $f\ast g$ componentwise by

$$(f\ast g)_{ij}(t) = \sum_{k=1}^{d} (f_{ik}\ast g_{kj})(t) = \sum_{k=1}^{d}\int_{0}^{t^-} f_{ik}(t-s) g_{kj}(s) ds, \qquad 1\le i,j\le d,$$

so that $f\ast g$ is again a $d\times d$ matrix-valued function. 

We denote by $f^{\ast (k)}$ the $k$-fold (matrix) convolution of $f$ with itself, i.e., for every $k \geq 1$,
$$f^{\ast (k)} = f^{\ast (k-1)} \ast f, \qquad \text{with } f^{\ast (1)} = f.$$

For $h^\star$ the (matrix-valued) kernel of the Hawkes process, we define 

\begin{equation}
\label{eq:def_resolvant_kernel}
    \Psi = \sum_{k=1}^{+\infty} (h^\star)^{\ast (k)}
\end{equation} 
This series is well defined (since $\rho(K)<1$), and in particular $\Psi$ belongs to $L^1$, i.e., each entry $\Psi_{ij}$ belongs to $L^1(\bb{R}_+)$.

Recall that $(T_n)_{n \geq 1}$ denotes the sequence of event times corresponding to the jumps of the multivariate counting process $N = (N_1, \dots, N_d)$, and that the inter-arrival times are defined by $\Delta T_n:= T_n - T_{n-1}$, while $K_n \in \{1, \dots, d\}$ indicates the component in which the $n$-th jump occurs, that is, $N_{K_n}$ jumps at time $T_n$. By convention, we denote $T_0 := 0$.
}

\subsection{Technical results}

We first collect two propositions that will be used throughout the subsequent proofs.

{\modif
\begin{prop}
\label{lemma:speed_cv_stationnary}
 If $\int V_1(x)\pi_0(dx)< + \infty$, with $V_1$ defined in Equation \eqref{eq:lyapunov_function} and $\pi_0$ the law of $X(0)$ then, for all $\mu \in \Theta_1$,
$$\bb{E}\qty(\vert \bas{\mu}{i}(X(t)) - \bas{\mu}{i} (\ov{X}(t))\vert) \leq \norm{\bas{\mu}{i}}_\infty C_1 r_1^t,$$

where $C_1>0$ and $r_1 \in (0,1)$ are two constant that do not depend on $\mu$.
\end{prop}

}

\begin{proof}
We know that 
$$\bb{E}\qty(\vert \bas{\mu}{i}(X(t)) - \bas{\mu}{i} (\ov{X}(t))\vert) \leq  \norm{\bas{\mu}{i}}_{\infty}\norm{  \s{L}(X(t)) - \pi_X}_{TV}  = \norm{\bas{\mu}{i}}_{\infty} \norm{\pi_0 P_t^X - \pi_X }$$
where $\pi_0$ is the law of $X(0)$. And thanks to Assumption \ref{ass:ergodicitySDE}, there exist $\Tilde{C}>0$ and $0< r_1<1$ such that for all $x$, $\norm{P^X_t(x,.) - \pi_X(.)}_{TV} \leq \Tilde{C} V_1(x) r_1^t$.
As a result, $$\bb{E}\qty(\vert \bas{\mu}{i}(X(t)) - \bas{\mu}{i} (\ov{X}(t))\vert) \leq \norm{\bas{\mu}{i}}_{\infty} \int \norm{P_t^X(x,.) - \pi_X(.)}_{TV} \pi_0(dx)  \leq \norm{\bas{\mu}{i}}_{\infty}   \qty( \int V_1(x) \pi_0(dx) )\Tilde{C} r_1^{t}.$$

Denoting $C_1:= \widetilde{C} \int V_1(x) \pi_0(dx) < \infty $, the previous inequality concludes the proof.
\end{proof}

\begin{prop}
\label{prop:bounded_moment_general_lambda}
Consider the process  $\wlamb{i}{}(t) = \widetilde{m}(t) + \sum_{j=1}^d \int_{(0,t)} \widetilde{h}_{ij}(t-s) d\N{j}{}(s)$ with $\widetilde{h}$ a positive function and $\widetilde{m}$ an adapted stochastic process verifying:
    \begin{enumerate}
    \item  $\bb{P}-a.s \text{  }, \quad  \widetilde{m}(t) >0 $ for all $t$,
    \item   For some $q \in \bb{N} \backslash \{0,1\}$, $\widetilde{h}_{ij} \in L^q$ and $\sup_{t\in \R^+} \bb{E}\qty(\widetilde{m}(t)^q)< + \infty$.
        \end{enumerate}
        
Then for all $i \indexset{1}{d}$ and all $l \in \llbracket 1,q\rrbracket$ the following holds 
$$\sup_{t\in \bb{R}_+} \bb{E}\qty(\abs{\wlamb{i}{}(t)}^l) < +\infty .$$

\end{prop}

\begin{proof}

To prove this Proposition, we start by proving the result for $\lamb{}{\star}(t)$, the true intensity of the process and then extend it.

Hence, we start by proving that for all $l \geq 1$ and all $i\in\{1\ldots d \}$, $$\sup_{t \in \bb{R}} \bb{E}\qty(\abs{\lamb{i}{\star}(t)}^l ) < + \infty $$.

    The proof is done by induction on $l\in \llbracket1, p\rrbracket$ and relies on an expression for the intensity in terms of the martingale associated with the point process. We therefore begin by detailing this equality and its properties before moving on to the induction itself.

{\modif
    First, note that
    \begin{align*}
         \lamb{i}{\star}(t) &= m^\star_i(t) + \sum_{j=1}^d \int_{0}^{t^-} h^{\star}_{ij}(t-s)d\N{j}{}(s) =m^\star_i(t) + \sum_{j=1}^d \int_{0}^{t^-} h^{\star}_{ij}(t-s)dM_j(s) + \sum_{j=1}^d \int_{0}^{t^-} h^{\star}_{ij}(t-s)\lamb{j}{\star}(s)ds,
    \end{align*}
    In a vectorial form :
    $$\lamb{}{\star}=  r +  h^{\star}\ast \lamb{}{\star},$$
    where $r(t)$ is the vector whose components are $(m^\star_i(t) + \sum_{j=1}^d \int_{0}^{t^-} h^{\star}_{ij}(t-s)dM_j(s))_i$. Thus, $\lamb{}{\star}$ verifies a Volterra's equation. As $\rho(K)<1$, the solution of this equation can be expressed as 
    \begin{equation}
    \label{eq:lambda_representation_martingale}
        \lamb{}{\star}(t) = \Psi \ast r(t) + r(t)
    \end{equation} 
where we recall, $\Psi$ is defined by Equation \eqref{eq:def_resolvant_kernel}. We can now start the induction. Using Equation \eqref{eq:lambda_representation_martingale}, we have

\begin{align*}
    \bb{E}\qty(\lamb{}{\star}(t)) = \bb{E}\qty(r(t)) + \bb{E}\qty(\Psi \ast r(t)) &  = \bb{E}\qty(  r(t) ) + \int_{0}^{t^-} \Psi(t-s) \bb{E}\qty(r(s)) ds 
\end{align*}

And $(\bb{E}\qty( r(t)) )_i= \bb{E}\qty( m^\star_i(t)) + \sum_{j=1}^d \bb{E}\qty( \int_0^{t^-} h^{\star}(t-s)dM_j(s) ) =  \bb{E}\qty( m^\star_i(t))$, where the last inequality is true because $M(t)$ is a martingale with $0$ as mean and initial condition. As $\sup_t  \bb{E}\qty( m^\star_i(t))< + \infty$, we deduce that $ \sup_t \bb{E}\qty(\lamb{}{\star}(t))<+\infty.$  This concludes the initialization.

\underline{Induction.} Let's fix $l\in \llbracket2,p\rrbracket$ and suppose that $\sup_t \bb{E}\qty(\abs{\lamb{i}{\star}(t)}^{l-1} ) < +\infty.$

Using Equation \ref{eq:lambda_representation_martingale}, combined with the fact that  $x,y,z >0$ and all $ p \geq 1$, $(x+y)^l < (1+z)^{l-1}x^l + ( 1 + z^{-1})^{l-1}y^l$, we have the equations (which has to be understood componentwise).

\begin{align*}
   \bb{E}\qty( \abs{\lamb{}{\star}(t)}^l) &= \bb{E}\qty( \qty( \Psi \ast r(t) + r(t))^l ) \leq   \qty(1+z)^{l-1} \bb{E}\qty(\qty(\Psi \ast r(t))^l) + (1+z^{-1})^{l-1} \bb{E}\qty( r(t)^l ) 
\end{align*}
To conclude the proof, we show that $\sup_t \bb{E}\qty( r(t)^l )<+ \infty$ and $\sup_t \bb{E}\qty(\qty(\Psi \ast r(t))^l )<+ \infty$.

\textit{Let us start with the first inequality e.g. let us show that $\sup_t \bb{E}\qty( r(t)^l )<+ \infty$.}
For all $i$, $r_i(t) = m^\star_i(t) + \sum_{j=1}^d \int_0^{t^-} h^{\star}(t-s) dM_j(s)$ so that
\begin{align*}
    \bb{E}\qty(r_i(t)^l) &= \bb{E}\qty( \qty( m^\star_i(t) + \sum_{j=1}^d \int_0^{t^-} h^{\star}(t-s) dM_j(s))^l) \\
    & \leq \qty( 1+z)^{l-1} \bb{E}\qty( m^\star_i(t)^l ) + \qty( 1+z^{-1})^{l-1}\bb{E}\qty( \qty(\sum_{j=1}^d \int_0^{t^-} h^{\star}(t-s) dM_j(s))^l) \\
    & \leq \qty( 1+z)^{l-1} \bb{E}\qty( m^\star_i(t)^l ) + d^{l-1} \sum_{j=1}^d \bb{E}\qty(\qty( \int_0^{t^-}  h^{\star}(t-s) dM_j(s))^l)
\end{align*}
where the last inequality is true because for all $x\in \bb{R}^d,$ $\norm{x}_{l_1}^l \leq d^{l-1} \norm{x}_{l_l}^l$. As $\sup_t \bb{E}\qty( m^\star_i(t)^l ) < \infty$.
\\

We are left with showing that  $\qty( \int_0^{t^-}  h^{\star}(t-s) dM_j(s))^l) < + \infty . $ We rely on the Kunita's inequality \cite{applebaum2009levy} that ensures that there exists a constant $C_l$, independent of $t$ such that 
$$\bb{E}\qty( \qty( \int_0^{t^-}  h^{\star}(t-s) dM_j(s))^l) \leq C_l \bb{E}\qty( \int_0^{t^-}  \qty(h^{\star})^{l}(t-s) \lamb{i}{\star}(s)ds )+ C_l \bb{E}\qty( \qty( \int_0^{t^-}  \qty(h^{\star})^{2}(t-s) \lamb{i}{\star}(s)ds )^{l/2}).$$
As a result
\begin{align*}
   & \bb{E}\qty( \qty( \int_0^{t^-}  h^{\star}(t-s) dM_j(s))^l) \\
   & \leq C_l \bb{E}\qty( \int_0^{t^-}  \qty(h^{\star})^{l}(t-s) \lamb{i}{\star}(s)ds )+ C_l \bb{E}\qty( \qty( \int_0^{t^-}  \qty(h^{\star})^{2}(t-s) \lamb{i}{\star}(s)ds )^{l/2}) \\
   & = C_l \int_0^{t^-}  \qty(h^{\star})^{l}(t-s) \bb{E}\qty(\lamb{i}{\star}(s))ds + C_l \bb{E}\qty( \qty( \int_0^{t^-}  \qty(h^{\star})^{2-\frac{4}{l}}(t-s) \qty(h^{\star})^{\frac{4}{l}}(t-s) \lamb{i}{\star}(s)ds )^{l/2}) \\
   & \leq C_l \sup_t \bb{E}\qty(\lamb{i}{\star}(t)) \norm{h}_{l_l}^l + C_l \bb{E}\qty[ \qty( \qty(\int_0^{t^-} \qty(h^{\star})^{2}(t-s)ds)^{ \qty(l/2-1)} \qty(\int_0^{t^-}\qty(h^{\star})^2(t-s) \lamb{\thetastar,i}{l/2}(s)ds)^{2/l} )^{l/2} ]\\
   & \leq  C_l \sup_t \bb{E}\qty(\lamb{i}{\star}(t)) \norm{h}_{l_l}^l + C_l \bb{E}\qty[ \qty( \qty(\int_0^{t^-} \qty(h^{\star})^{2}(t-s)ds)^{ \qty(l/2-1)} \int_0^{t^-}\qty(h^{\star})^2(t-s) \lamb{\thetastar,i}{l/2}(s)ds)] \\
   & \leq C_l \sup_t \bb{E}\qty(\lamb{i}{\star}(t))\norm{h}_{l_l}^l + C_l \norm{h}_2^{l/2-1} \sup_t \bb{E}\qty(\lamb{\thetastar,i}{l/2}(t))\norm{h}_2^2,
\end{align*} 
where we applied Hölder inequality on the fourth line of equations.
As a result, the induction assumption insure that $\sup_t \bb{E}\qty( r_i(t)^l)<+\infty.$

\textit{Let us now show that $\sup_t \bb{E}\qty(\qty(\Psi \ast r(t))^l )<+ \infty$ }

For each $i \in \llbracket 1, d\rrbracket$, 
\begin{align*}
   \qty(  \qty(\Psi \ast r)_i(t))^l & = \qty(\sum_{k=1}^d \int_0^{t^-}   \Psi_{ik}(t-s) r_k(s) ds )^l \leq  d^{l-1} \sum_{k=1}^d \qty(\int_0^{t^-}   \qty( \Psi_{ik}(t-s) r_k(s)ds))^l \\
   & \leq d^{l-1} \sum_{k=1}^d  \qty(\int_0^{t^-}   \Psi_{ik}^{(1-1/l)(1-1/l)^{-1}}(t-s)ds )^{l-1} \qty(\int_0^{t^-}   \Psi_{ik}(t-s) r_k^l(s) ds ) \\
   & \leq d^{l-1} \sum_{k=1}^d  \norm{\Psi_{ik}}_{\s{L}_1}^{l-1} \qty(\int_0^{t^-}   \Psi_{ik}(t-s) r_k^l(s) ds ),
\end{align*}
where we applied Hölder with $\frac{1}{p}= 1-1/l$ once again at the third line of the previous equations. 
As a result, 
\begin{align*}
    \bb{E} \qty(  \qty(\Psi \ast r)_i(t))^l & \leq  d^{l-1} \sum_{k=1}^d  \norm{\Psi_{ik}}_{\s{L}_1}^{l-1} \qty(\int_0^{t^-}   \Psi_{ik}(t-s) \bb{E} (r_k(s)^l) ds ) \leq d^{l-1} \sum_{k=1}^d  \norm{\Psi_{ik}}_{\s{L}_1}^{l} \sup_t \bb{E} (r_k(t)^l)  < +\infty
\end{align*}
where the last inequality holds because $\Psi_{ik} \in \s{L}^1$ for each $i,k$ and because we already proved that $\sup_t \bb{E} (r_k(t)^l) < + \infty$.
This concludes the induction and thus the proof of the Proposition \ref{prop:bounded_moment_general_lambda} in the case where we consider the intensity of the process. }

We now turn to the proof of the general statement of Proposition \ref{prop:bounded_moment_general_lambda}. Let us take $l \geq 1 $ with $l \in \llbracket 1, p \rrbracket$.

\begin{align*}
    \bb{E} \qty(\abso{\widetilde{\lambda}(t)}^l) & \leq \qty(1 + z)^{l-1}\abso{\widetilde{m}(t)}^l + ( 1 + z^{-1} )^{l-1}\qty( \sum_{j=1}^d \int_{(0,t)}\widetilde{h}_{ij}(t-s)d\N{j}{}(s))^l\\
    &  \leq \qty(1 + z)^{l-1}\abso{\widetilde{m}(t)}^l + ( 1 + z^{-1} )^{l-1} d^{l-1} \sum_{j=1}^d \qty(\int_{(0,t)}\widetilde{h}_{ij}(t-s)d\N{j}{}(s))^l
\end{align*}

Furthermore,
\begin{align*}
    & \bb{E}\left[\left(\int_{(0, t)} \widetilde{h}_{ij}(t-s) d\N{j}{}(s)\right)^l\right]\\ 
    = &  \bb{E}\left(\left(\int_{(0, t)}\widetilde{h}_{ij}(t-s) dM_j(s)+\int_{(0, t)} \widetilde{h}_{ij}(t-s) \lamb{j}{}(s)ds\right)^l\right) \\
\leq & \left(1+z\right)^{l-1} \bb{E}\left(\left(\int_{(0, t)} \widetilde{h}_{ij}(t-s) dM_j(s)\right)^l\right)+(1+z^{-1})^{l-1} \bb{E}\left(\left(\int_{(0, t)} \widetilde{h}_{ij}(t-s) \lamb{j}{}(s) ds\right)^l\right)
\end{align*}
And the Kunita inequality ensures that there exists $C_l$ such that $$\bb{E}\left[\left(\int_{(0, t)} \widetilde{h}_{ij}(t-s) dM_j(s)\right)^l\right] \leq C_l \bb{E}\left(\int_{(0, t)} \widetilde{h}_{ij}(t-s)^l \lamb{j}{}(s) ds \right)+C_l \bb{E}\left(\left(\int_{(0, t)}\widetilde{h}_{ij}(t-s)^2 \lamb{j}{}(s) ds\right)^{l/2}\right).$$

As a result we have shown the following inequality.
\begin{align*}
    \bb{E}\qty(\abso{\widetilde{\lambda}(t)}^l)  \leq & \left(1+z\right)^{l}\abso{\widetilde{m}(t)}^l + (1+ z^{-1})^{l-1}(1+z)^{l-1} C_l \sum_{j=1}^d \bb{E}\left[\int_{(0, t)} \widetilde{h}_{ij}^l(t-s) \lamb{j}{}(s) ds\right] \\
    & +(1+ z^{-1})^{l-1}(1+z)^{l-1} C_l \sum_{j=1}^d \bb{E}\left[\left(\int_{(0, t)} \widetilde{h}_{ij}(t-s) \lamb{j}{}(s) d s\right)^{l / 2}\right] \\
& +(1+z^{-1})^{l-1} \sum_{j=1}^d \bb{E}\left[\left(\int_{(0, t)} \widetilde{h}_{ij}(t-s) \lamb{j}{}(s) ds\right)^l\right]
\end{align*}
Using Holder inequality we thus have, 
\begin{align*}
     \sup_t \bb{E}\qty(\abso{\widetilde{\lambda}(t)}^l) & \leq \left(1+z\right)^{l}\sup_{t\in\bb{R}} \abso{\widetilde{m}(t)}^l  + C_l (1+ z^{-1})^{l-1}(1+z)^{l-1}\norm{h}_l^l \sum_{j=1}^d \sup_{t\in\bb{R}} \bb{E}\qty(\abso{ \lamb{j}{}(s)}) \\ 
    &  + C_l(1+ z^{-1})^{l-1}(1+z)^{l-1} \sum_{j=1}^d \sup_{t \in \bb{R}} \norm{\widetilde{h}_{ij}}_1 \bb{E}\qty( \abso{ \lamb{j}{}(t)}^{l/2}) \\
    & +(1+z^{-1})^{l-1}  \sum_{j=1}^d \norm{\widetilde{h}_{ij}}_1  \sup_{t \in \bb{R}}\bb{E}\qty( \abso{ \lamb{j}{}(t)}^{l}) \\
    & < + \infty,
\end{align*}
where the last inequality is due to the fact that $\widetilde{h}_{ij}\in \s{L}^q \supseteq \s{L}^l$ for all $1 \leq l \leq q$. This concludes the proof. 
\end{proof}

\section{Proof of Theorem \ref{theo:exp_ergodicity}}
\label{proof:exp_erg}
As mentioned in the main text, establishing the ergodicity of $Z$ requires only general Markovian properties for $X$.

We rely on Theorem 6.1 of \cite{meyn1993stability}, which establishes that, under the $T$-chain property and a suitable Lyapunov function $V$, a Markov process is geometrically ergodic.  
In particular, it ensures that under those conditions, there exist constants $C > 0$ and $ \eta \in (0,1)$ such that
$$\norm{ P^Z_T(z, \cdot) - \pi }_{\mathrm{TV}} \leq C V(z)   \eta^T, $$
where $P^Z_t(x, \cdot)$ denotes the transition kernel of the process and $\pi$ its unique invariant distribution. In the following, we verify that both the Lyapunov condition and the $T$-chain property hold for our process.

\subsection{The Lyapunov condition}
\label{proof:LyapunovCondition}

The first condition is the existence of a Lyapunov function that verifies the following bound:
\begin{equation}
\label{eq:Lyapunov}
    \s{A}_ZV(x,y) \leq -c V(x,y) + d \indic_{\s{K}}(x,y) \text{  with $c,d$ strictly positive constants and }\s{K} \text{ a compact set }
\end{equation}
where the infinitesimal generator is defined, when the limit exists, by
\[
A_Z f(z) = \lim_{t \to 0^+} \frac{\bb{E}_z \left[f(Z(t))\right] - f(z)}{t}.
\]

\paragraph{Infinitesimal generator computation}
We begin by computing the infinitesimal generator $ \s{A}_Z$ before explaining the form of the function $V$.

Let $f: \bb{R}^m \times \bb{R}^{d^2} \to \bb{R}$ be a function of class $C^2 $. Recall that $(X,Y)$ are defined by Equation \ref{eq:markov_eq}, so applying Itô's formula 
(see for example \cite{protter2005stochastic, applebaum2009levy} or  \cite[Chapter II, Theorem 5.1]{ikeda2014stochastic}), leads to:

\begin{align*}
& f\qty(X(t), Y(t)) - f\qty(x, y) \\
& \quad = \sum_{i=1}^m \int_0^t \nabla_{x_i} f\qty(X(s^{-}), Y(s^{-}))   dX_i(s) 
+ \frac{1}{2}\sum_{i,j=1}^m \int_0^t \nabla_{x_i x_j}^2 f\qty(X(s^{-}), Y(s^{-}))   d[X_i,X_j](s) \\
&\quad + \sum_{i,j=1}^d \int_0^t \nabla_{y_{ij}} f\qty(X(s^{-}), Y(s^{-}))   dY_{ij}(s) \\
&\quad + \sum_{j=1}^d \int_0^t \qty(f\qty(X(s), Y(s)) - f\qty(X(s^{-}), Y(s^{-})) 
- \sum_{i=1}^d \Delta Y_{ij}(s)   \nabla_{y_{ij}} f\qty(X(s^{-}), Y(s^{-}))) d\N{j}{}(s).
\end{align*}

Taking expectations and using the dynamics of $(X(t), Y(t)) $ given in Equation \eqref{eq:markov_eq}, we get:
\begin{align*}
& \bb{E}_{x,y}[f\qty(X(t), Y(t)) - f\qty(X(0), Y(0))]  \\
& \quad = \bb{E}_{x,y} \left[ \int_0^t \nabla_x f\qty(X(s^{-}), Y(s^{-}))^T b(X(s^{-})) ds  + \frac{1}{2} tr\qty[ (\sigma \sigma^T)(X(s^{-} ) \nabla^2_x f\qty(X(s^{-}), Y(s^{-})) ] ds \right] \\
&\quad + \bb{E}_{x,y} \left[ \sum_{i,j=1}^d \int_0^t \qty(
-\betastar_{ij} Y_{ij}(s^{-})   \nabla_{y_{ij}} f\qty(X(s^{-}), Y(s^{-}))) ds \right] \\
&\quad + \bb{E}_{x,y} \left[ \sum_{j=1}^d \int_0^t \qty(
f\qty(X(s^{-}), Y(s^{-}) + A_j) - f\qty(X(s^{-}), Y(s^{-}))) d\N{j}{}(s) \right],
\end{align*}
where $ A_j  \in \bb{R}^{d \times d} $ is the matrix whose only nonzero entries are in the $j $-th column, given by the vector $(\alphastar_{ij})_i $.
Then,  using the fact that $\bb{E}[h(s)   d\N{j}{}(s)] = \bb{E}[h(s)   \lamb{j}{\star}(s)   ds] $, with 

$$\lamb{k}{\star}(s) = g_{k}^\star(X(s)) + \sum_{j=1}^d Y_{kj}(s),$$

we obtain:
\begin{align*}
& \bb{E}_{x,y}[f\qty(X(t), Y(t)) - f\qty(X(0), Y(0))] \\
&\quad = \bb{E}_{x,y} \left[ \int_0^t \nabla_x f\qty(X(s^{-}), Y(s^{-}))^T b(X(s^{-})) ds  + \frac{1}{2} tr\qty[ (\sigma \sigma^T)(X(s^{-} )) \nabla^2_x f\qty(X(s^{-}), Y(s^{-})) ] ds \right] \\
&\quad + \bb{E}_{x,y} \left[ \sum_{i,j=1}^d \int_0^t \qty(-\betastar_{ij} Y_{ij}(s^{-})   \nabla_{y_{ij}} f\qty(X(s^{-}), Y(s^{-}))) ds \right] \\
&\quad + \bb{E}_{x,y} \left[ \sum_{j=1}^d \int_0^t \qty( f\qty(X(s^{-}), Y(s^{-}) + A_j) - f\qty(X(s^{-}), Y(s^{-}))) \qty(g_{j}^{\star}(X(s^{-})) + \sum_{i=1}^d Y_{ji}(s^{-})) ds \right].
\end{align*}
This gives the infinitesimal generator $ \s{A}_Z $ of the process $Z(t) = (X(t), Y(t)) $, by
\begin{eqnarray*}
\s{A}_Z f\qty(x,y)   &=&   \nabla_x f\qty(x,y)^T b(x) + \frac{1}{2} tr\qty[ (\sigma \sigma^T)(x ) \nabla^2_x f\qty(x, y)) ] \\
    & &  +  \sum_{j=1}^{d} (f\qty(x,y + A_j) - f\qty(x,y))\qty(g_{j}^{\star}(x) + \sum_{i=1}^{d} y_{ji}) - \sum_{i,j=1}^d \betastar_{ij} \nabla_{y_{ij}}f\qty(x,y) y_{ij}.
\end{eqnarray*}

\paragraph{Lyapunov condition} 

Denoting 
\[ \langle A, B \rangle:=\sum_{i,j} a_{ij} b_{ij} \text{ for all } A,B \in \bb{R}^{d^2}, \]
we define the function $V_2: \bb{R}^{d^2} \to \bb{R}_+ $ by
\begin{equation*}
V_2(y):= e^{\langle M, y \rangle},    
\end{equation*}
 where $M = (m_{ij}) \in \bb{R}^{d \times d} $ is a matrix of real parameters to be chosen later.  
We then consider the candidate Lyapunov function $V: \bb{R} \times \bb{R}^{d^2} \to \bb{R}_+ $ defined by
\[
V(x, y):= V_1(x) + V_2(y),
\]
where $V_1: \bb{R}^m \to \bb{R}_+ $ is the function given by in Equation \eqref{eq:lyapunov_function} in the context of Assumption \ref{ass:ergodicitySDE} or by Section \ref{subsec:assumptions_X_ergo_general} in the general Markovian context. Thus we have:

\begin{align*}
    \s{A}_Z V(x,y) & = \s{A}_X V_1(x) + \sum_{j=1}^d e^{\scalar{M}{y}} \qty(e^{\scalar{M}{A_j}} -1)\qty(g_{j}^{\star}(x) + \sum_{i=1}^{d} y_{ji}) - e^{\scalar{M}{y}}\sum_{i,j} \betastar_{ij} y_{ij} m_{ij}  \\
    & \leq  \s{A}_X V_1(x) + \sum_{j=1}^d e^{\scalar{M}{y}} \qty(e^{\scalar{M}{A_j}} -1)\qty(g_{+} + \sum_{i=1}^{d} y_{ji}) - e^{\scalar{M}{y}}\sum_{i,j} \betastar_{ij} y_{ij} m_{ij}
\end{align*}

where we used the upper bound of Assumption \ref{ass:bounded_baseline2}

The first term corresponds to the generator of the diffusion $X $, applied to $V_1 $. Furthermore, there exist constants $c_1 > 0 $ and $d_1 > 0 $ such that
\[
\s{A}_X V_1(x) \leq -c_1 V_1(x) + d_1 \indic_{\s{K}_1}(x).
\]
where $\s{K}_1$ is given by in Equation \eqref{eq:lyap_condition} in the context of Assumption \ref{ass:ergodicitySDE} or by Section \ref{subsec:assumptions_X_ergo_general} in the general Markovian context.
The second term, namely 

\[\s{A}_{Y \vert X } V(x,y):= \s{A}_Z V(x,y) -  \s{A}_X V_1(x):= \sum_{j=1}^d e^{\scalar{M}{y}} \qty(e^{\scalar{M}{A_j}} -1)\qty(g_{+} + \sum_{i=1}^{d} y_{ji}) - e^{\scalar{M}{y}}\sum_{i,j} \betastar_{ij} y_{ij} m_{ij},\]
and corresponds to the generator acting on $V_2 $ for a multidimensional Hawkes process with constant baseline intensity $g_{+} $. Following arguments from Proposition 4.5 of \cite{clinet2017statistical}, which establishes a Lyapunov property for a multidimensional Hawkes process with exponential kernel and constant baseline, there exist constant $c_2 > 0 $, $d_2 > 0 $, and a compact set $\s{K}_2 \subset \bb{R}^{d^2} $ such that
\[
 \s{A}_{Y \vert X} V(x, y) \leq -c_2 V_2(y) + d_2 \indic_{\s{K}_2}(y).
\]

Combining both inequalities concludes the proof: 
\[
 \s{A}_Z V(x, y) \leq - c V(x,y) + d \indic_{\s{K}_1\times \s{K}_2}(x,y).
\]

\subsubsection{T-chain property of the transition kernel}
\label{proof:Tchaindim2}
It is generally acknowledged that this type of proof is tedious and difficult to follow, particularly due to the large number of indices to manipulate. For clarity and readability, we therefore present the proof in the case of a Hawkes process of two components, that is, $d=2$. The proof for an arbitrary dimension $d \geq 1$ proceeds in the exact same way replacing $d=2$ with a generic $d$, and can be found in \ref{proof:Tchaingeneraldim}.

In the following, we thus consider a bidimensional Hawkes process having baseline functions $g_{1}^\star$ and $g_{2}^{\star}$ with an interaction matrix $\alphastar = (\alphastar_{ij})_{1 \leq i,j \leq 2}$ and a decay matrix $\betastar = (\betastar_{ij})_{1 \leq i,j \leq 2}$.

{\modif 
The proof requires the two following points.

\begin{enumerate}
    \item We start by showing that the transition kernel of $Z$ verifies
\[
    P^{Z}_T\qty((x,y), A \times B) \geq Q_T((x,y), A \times B),
\]
with $Q_T$ a nontrivial kernel such that $Q_T(., A\times B)$ is semi lower-continuous for any subset $A\times B$ of $  \bb{R}^m \times \bb{R}_+^{d^2}$. 
\item Secondly, we need to make sure that $Z$ admits a reachable point meaning that there exist a point $z^\star = (x^\star, y^\star)$ such that for every starting point $(x,y)$ and every neighbour of $z^\star$, denoted $\s{U}$ there exists a time $T$ such that $\P_T^Z((x,y), \s{U}) > 0$.
\end{enumerate}

Together, those properties ensure that $Z$ is an irreducible $Q$-chain. (see Proposition 6.2.1 of \cite{meyn2012markov}) and thus, that every compact set is a petite set (Theorem 6.2.5 of \cite{meyn2012markov}) so that Theorem 5.1 \cite{meyn1993stability} applies. 
We note that the notion of a petite set is quite specific to the theory of Markov processes, similar to that of $\psi$-irreducibility. For a precise definition, we refer the reader to \cite{meyn1993stability}.
}

We start by showing the first point stated. The main idea is to treat separately the component of the intensity that depends on the process $Y $, by first conditioning on the trajectory of $W(t) $ allowing us to consider $X(t) $ as non-random.

\begin{equation}
\label{eq:lower_bound_kernel}    
\begin{split}
    P^Z_T((x,y), A \times B) 
    &= \bb{E}_{(x,y)}\qty(\indic_A(X(T))  \indic_B(Y(T))) \\
    &= \bb{E}_{(x,y)} \qty(\indic_A(X(T))  \bb{E}_{(x,y)}\qty(\indic_B(Y(T))\vert \s{F}_T^W)) 
\end{split}
\end{equation}
where $\s{F}^T_W$ denoted the canonical filtration of the Brownian motion.

We thus start working on  
$\bb{E}_{(x,y)}\qty(\indic_B(Y(T))  \vert  \s{F}^T_W )$. 
In this setting, the deterministic part of the intensity, given by $g_{i}^{\star}(X(t))$, 
can be viewed as a bounded function of time. Using this boundedness, we recover a setting similar 
to a standard Hawkes process with constant baseline intensity, which allows us to apply 
existing results from the classical Hawkes process framework.

We employ techniques already used in \cite{clinet2017statistical} (proof of Lemma A.3) and \cite{dion2020exponential}, 
relying on the fact that the intensity is piecewise deterministic. Fixing the number of jumps that occurred before $T$, the value of $Y(T)$ becomes a deterministic function depending only on the initial condition and the jump times (see Equation \eqref{eq:markov_eq}), which simplifies the computation of the expectation.

We thus work on the event
\[
A_T:= \{ T_1 < T_2 < T_3 < T_4 < T < T_5 \} \cap \{ K_1 = 1, K_2 = 1, K_3 = 2, K_4 = 2 \}.
\]

In other words, we require that each component jumps $d=2$ times before time $T$, 
and that the jumps of component 1 occur before those of component 2.
Then, we have
\begin{align*}
    \bb{E}_{(x,y)}\qty(\indic_B(Y(T))  \vert  \s{F}^T_W ) 
    &\geq \bb{E}_{(x,y)}\qty(\indic_B(Y(T)) \indic_{A_T}  \vert  \s{F}^T_W ),
\end{align*}
and
\begin{align*}
    &\bb{E}_{(x,y)}\qty(\indic_B(Y(T)) \indic_{A_T}  \vert  \s{F}^W_T ) \\
   & = \int f^{ (\Delta T_{5},K_{5}, \dots, \Delta T_{1}, K_{1})}(t_5, k_5, \dots, t_1, k_1) 
      \indic_{B}(\gamma_y(T)) \indic_{t_1+\ldots +t_4< T } \indic_{t_1+\ldots +t_5 > T }    dt_1 \dots dt_5,
\end{align*}

where:
\begin{itemize}
    \item $f^{ (\Delta T_{1}, K_{1}, \dots, \Delta T_{5}, K_{5})}(t_5, k_5, \dots, t_1, k_1)$ 
    is the joint density of the inter-arrival times $\Delta T_i:=T_i-T_{i-1}$ and the marks $(K_i)$, conditional on the trajectory of $W$;
    \item $\gamma_y(T) \in \bb{R}^{d^2}$ denotes the value of $Y(T)$, conditional on the event
    $\{ Y(0) = y \} \cap \{ \Delta T_1 = t_1, K_1 = 1, \dots, \Delta T_4 = t_4, K_4 = 2 \}$ and conditional to $W$.
\end{itemize}

\paragraph{Computation of $\gamma_y(T)$}
Using the recursive expression of $Y$, we can write $\gamma_y(t)$ on each interval 
$[t_1 + \dots + t_{i-1},  t_1 + \dots + t_i]$ as a piecewise-deterministic function that depends on the initial condition $y$ and the jump times: 

\[ 
\begin{cases}

\begin{pmatrix}
y_{11} e^{-\betastar_{11} t} & y_{12} e^{-\betastar_{12} t} \\
y_{21} e^{-\betastar_{21} t} & y_{22} e^{-\betastar_{22} t}
\end{pmatrix}, \quad t \leq t_1 \\[0.8em]
\begin{pmatrix}
y_{11} e^{-\betastar_{11} t} + \alphastar_{11} e^{-\betastar_{11}(t-t_1)} & y_{12} e^{-\betastar_{12} t} \\
y_{21} e^{-\betastar_{21} t} + \alphastar_{21} e^{-\betastar_{21}(t-t_1)} & y_{22} e^{-\betastar_{22} t}
\end{pmatrix}, \quad  t_1 < t \leq t_1+t_2 
\end{cases}
\]
and so on, and for $ \sum_{i=1}^4 t_i < t\leq T$,
$$
\begin{pmatrix}
y_{11} e^{-\betastar_{11} t} + \alphastar_{11}\qty(e^{-\betastar_{11}(t-t_1)} + e^{-\betastar_{11}(t-t_1-t_2)}) & 
y_{12} e^{-\betastar_{12} t} + \alphastar_{12}\qty(e^{-\betastar_{12}(t-\sum_{i=1}^3 t_i)} + e^{-\betastar_{12}(t-\sum_{i=1}^4 t_i)}) \\
y_{21} e^{-\betastar_{21} t} + \alphastar_{21}\qty(e^{-\betastar_{21}(t-t_1)} + e^{-\betastar_{21}(t-t_1-t_2)}) & 
y_{22} e^{-\betastar_{22} t} + \alphastar_{22}\qty(e^{-\betastar_{22}(t-\sum_{i=1}^3 t_i)} + e^{-\betastar_{22}(t-\sum_{i=1}^4 t_i)})
\end{pmatrix}
$$
In a more compact form, these equations can be written
\begin{equation}
\label{eq:kernel_val}
\gamma_y(T) = \qty(y_{ij} e^{-\betastar_{ij} T} + \alphastar_{ij} 
\sum_{i=2(j-1)+1}^{2j} e^{-\betastar_{ij} (T - t_1 - \dots - t_i)})_{1 \leq i,j\leq 2}.
\end{equation}
We denote by $(\gamma_y^{kj}(t))_{1 \leq k,j \leq 2}$ the entries of the matrix $\gamma_y(t)$, and define, for each $k$,
\[
\gamma_y^k(s):= \gamma_y^{k1}(s) + \gamma_y^{k2}(s),
\]
which represents the total contribution to component $k$ obtained by summing over the two columns of $\gamma_y(s)$.

\paragraph{Computation of the joint density $f $}

We aim to make explicit the joint density $f $ of the random vector $(\Delta T_{5}, K_{5}, \dots, \Delta T_1, K_1)$. Using Bayes' formula, we have:

\begin{equation}    
\label{eq:JointDensity}
f^{(\Delta T_{5}, K_{5}, \dots, \Delta T_{1}, K_{1})}
= f^{\qty(\Delta T_5, K_5 \mid \Delta T_{<5}, K_{<5})} \prod_{k=0}^{2}   \prod_{i=2k+1}^{2(k+1)} 
   f^{\qty(\Delta T_i, K_i \mid \Delta T_{<i}, K_{<i})},
\end{equation}

where the notation $f^{\qty(\Delta T_i, K_i \mid \Delta T_{<i}, K_{<i})}$ corresponds to the conditional density of the pair $(\Delta T_i, K_i)$ given the previous history.

To compute the conditional terms $f^{\qty(\Delta T_i, K_i \mid \Delta T_{<i}, K_{<i})} $, we first consider the survival function of the $i$-th inter-arrival time given the past, and then derive the corresponding density from it. 
Working on the event $A_T$, for  $i \in \llbracket 1,2 \rrbracket$, the condition $\Delta T_i > t_i$ is equivalent to $N_1([T_{i-1}, T_{i-1} + t_i]) = 0,$  meaning that between two jump times, the process behaves as an inhomogeneous Poisson process with intensity $s \mapsto g_{1}^{\star}(X(s)) + \sum_{j=1}^2 \gamma_y^{1j}(s) =  g_{1}^{\star}(X(s)) + \gamma_y^{1}(s).$ 
Hence, we have

\begin{align*}
\bb{P}(\Delta T_i > t_i \mid \Delta T_{i-1}, \dots, \Delta T_1, K_1, \s{F}_T^W)  &= \bb{P}\qty(N_1([T_{i-1}, T_{i-1} + t_i]) = 0 \mid \Delta T_{i-1}, \dots, \Delta T_1, K_1, \s{F}_T^W ) \\
&= \exp \qty(- \int_{T_{i-1}}^{T_{i-1} + t_i} \qty[ g_{1}^{\star}(X(s)) + \gamma_y^{1}(s)]ds).
\end{align*}

Similarly, for $i \in \{3,4\}$, we have

\begin{align*}
\bb{P}(\Delta T_i > t_i \mid \Delta T_{i-1}, \dots, \Delta T_1, K_1, \s{F}_T^W) = \exp \qty(- \int_{T_{i-1}}^{T_{i-1} + t_i} \qty[ g_{2}^{\star}(X(s)) + \gamma_y^{2}(s) ]ds).
\end{align*}

Once again for $i=5$ we have 
\begin{align*}
\bb{P}(\Delta T_5 > t_5 \mid \Delta T_{4}, \dots, \Delta T_1, K_1, \s{F}_T^W) = \exp\qty(- \int_{T_{4}}^{T_{4} + t_5} \qty[ g_{2}^{\star}(X(s)) + \gamma_y^{k_5}(s) ]ds).
\end{align*}
where $k_5$ is the component that jumps at time $T_5$. Hence, for all $i$, the conditional density of $\Delta T_i$ is given by

$$t \mapsto \Bigl(g_{k}^\star(X(T_{i-1}+t)) + \gamma_y^k(T_{i-1}+t) )\exp \qty(- \int_{T_{i-1}}^{T_{i-1}+t} \qty[g_{k}^{\star}(X(s)) + \gamma_y^k(s) ] ds ),$$

where $k$ is the component that jump at time $T_k$: $k=1$ if $i \in \{1,2\}$ and $k=2$ if $i \in \{3,4\}$. For $i=5$, note that $k_5$ plays no role as we have no condition imposed on it.

For the conditional mark density, given $\Delta T_i, \dots, \Delta T_1, K_1, \s{F}_T^W $, the probability that the event occurs in component $k $ is:

\begin{eqnarray*}
\bb{P}(K_i = k \mid \Delta T_i, \dots, \Delta T_1, K_1, \s{F}_T^W) &=& \frac{\lamb{k}{}(T_i^-)}{\lamb{1}{}(T_i^-) + \lamb{2}{}(T_i^-)} \\
&=& \frac{g_{k}^\star (X(T_i^-)) + \gamma_y^k(T_i^-)}{ g_{1}^{\star}(X(T_i^-))+ \gamma_y^1(T_i^-) + g_{2}^{\star}(X(T_i^-)) +\gamma_y^2(T_i^-) }.
\end{eqnarray*}

We can then write the full conditional density as:

\begin{align*}
f^{\qty(\Delta T_i, K_i \mid \Delta T_{<i}, K_{<i})}(t_i,k_i..,t_1,k_1) &= 
\frac{g_{k_i}^{\star}(X(T_i^-)) + \gamma_y^{k_i}(T_i^-)}{ g_{1}^{\star}(X(T_i^-))+ \gamma_y^{1}(T_i^-) + g_{2}^{\star}(X(T_i^-)) +\gamma_y^2(T_i^-) } \\
&\quad \times \qty(g_{k_i}^{\star}(X(T_{i-1} + t_i)) + \gamma_y^{k_i}(T_{i-1} + t_i)) \\
& \times \exp\qty(- \int_{T_{i-1}}^{T_{i-1} + t_i} g_{k_i}^{\star}(X(s)) + \gamma_y^{k_i}(s)ds),
\end{align*}

where we recall that $k_i$ is equal to $1$ if $i\in \{1,2\} $ and to $2$ if $i\in \{3,4\} $, and is arbitrary for $i=5$.

Now, using the fact that the baseline intensities $\qty(\bastar{\mu}{k}(X(t)))_k $ are bounded above and below (Assumption \ref{ass:bounded_baseline2}), we can deduce a lower bound:

\begin{align*}
f^{\qty(\Delta T_i, K_i \mid \Delta T_{<i}, K_{<i})}(t_i,k_i..,t_1,k_1) & \geq \frac{ g_{-}+ \gamma_y^k(T_i^-)}{ g_{+}+ \gamma_y^1(T_i^-) +g_{+} +\gamma_y^2(T_i^-) } \\
&\quad \times \qty(g_{-}+ \gamma_y^{k}(T_{i-1} + t_i))  \times \exp\qty(- \int_{T_{i-1}}^{T_{i-1} + t_i} g_{+} + \gamma_y^{k_i}(s)ds)..
\end{align*}
Using Equation \eqref{eq:JointDensity}, we obtain an expression that is very similar 
to the one derived for a Hawkes process with constant baseline (see the third equation in the proof of  Lemma A.4 in \cite{clinet2017statistical}). The only difference lies in the fact that, instead of having two  constants $g_{+}$ and $g_{-}$, the reference only involves a single constant, namely the baseline rate of  the Hawkes process (which in our setting would correspond to the case  $g_{+} = g_{-}$). Nevertheless,  the rest of the argument is unchanged, since it relies on a change of variables in $\bb{R}^d$ associated with the mapping $(t_1,\dots,t_d) \mapsto \gamma_y(t)$, which remains valid in our framework.

Thus, the same arguments as in Lemma A.3 of \cite{clinet2017statistical} apply, showing that:

$$Q_{2,T}(y,B):= \int_{\bb{R}^{d^2}} f^{ (\Delta T_{5}, K_{5}, \dots, \Delta T_{1}, K_{1})}(t_{5}, k_{5},\dots, t_{1}, k_{1})   \indic_{B}(\gamma_y(T)) dt_1 \dots dt_{5}$$
is a kernel that is non-trivial and lower semi-continuous. Combining this with Equation \eqref{eq:lower_bound_kernel} and Assumption \ref{ass:bounded_baseline2}, we obtain:
$$P^{Z}_{T}\qty((x,y), A\times B) \geq Q_T \qty((x,y), A\times B):= Q_{1,T}(x,A) Q_{2,T}(y,B)$$
where $Q_{1,T}(x,A)= \bb{E}_x\qty(\indic_A(X(T)))$. This proves that $P $ is a $T $-chain.\\

It remains to show that $Z = (X,Y)$ admits a reachable point. {\modif
Let $(x,y) \in \bb{R}^m \times \bb{R}_+^{d^2}$ be a starting point, and let $\s{U} \times \s{V}$ be a neighbourhood of $(x^\star, 0)$.

On the event $\{N(T) = 0\}$, no jump occurs on $[0,T]$, so that $Y_{ij}(T) = y_{ij} e^{-\betastar_{ij} T}$ for all $i,j$. As $\s{V}$ is a neighbourhood of $0$, there exist $T_0$ such that for all $T > T_0$, $ \qty(y_{ij} e^{-\betastar_{ij} T})_{i,j} \in \s{V}$. Thus, for every $T \geq T_0$, $\{N(T) = 0\} \subset \{Y(T) \in \s{V}\}.$

And,
$$\bb{P}_{x,y}\qty(N(T) = 0 \mid \s{F}^W_T) = \exp\qty(- \int_0^T \sum_{k=1}^d \lamb{k}{\star}(t)  dt),
\qquad \text{with } \lamb{k}{\star}(t) = g_k^\star(X(t)) + \sum_{j=1}^d y_{kj} e^{-\betastar_{kj} t} \text{ on } \{N(T) = 0\}.$$

Using the upper bound $g_k^\star \leq g_+$ of Assumption \ref{ass:bounded_baseline2}, we get
$$\bb{P}_{x,y}\qty(N(T) = 0 \mid \s{F}^W_T) \geq \exp\qty(- d  g_+ T - \sum_{k,j=1}^d \frac{y_{kj}}{\betastar_{kj}}) > 0,$$
a lower bound that does not depend on the trajectory of $W$. Since $X$ does not depend on $N$, it follows that, for every $T \geq T_0$,
$$ P^Z_T\qty((x,y), \s{U} \times \s{V}) \geq \bb{E}_{x,y}\qty[\indic_{\s{U}}(X(T))  \bb{P}_{x,y}\qty(N(T) = 0 \mid \s{F}^W_T)] \geq \exp\qty(- d  g_+ T - \sum_{k,j=1}^d \frac{y_{kj}}{\betastar_{kj}})   \bb{P}_x\qty(X(T) \in \s{U}).$$
The point $(x^\star, 0)$ is therefore a reachable point for $Z$.
}

This concludes the proof.

\section{Proof of Proposition \ref{prop:stationary_intensity_true_process}}
\label{proof:stationary_intensity_true_process}

The first point is a consequence of \cite{bremaud1996stability}. Let us show the second point stated in Proposition \ref{prop:stationary_intensity_true_process}.

\vspace{\baselineskip}

{\modif 
Given that 
\begin{align*}
    \lamb{i}{\star}(t) - \olamb{i}{\star}(t) &= g_{i}^{\star}(X(t)) + \sum_{j=1}^d \int_{0}^{t^-} \alphastar_{ij} e^{-\betastar_{ij} (t-s)} d\N{j}{}- \qty( \bastar{\mu}{i} (\ov{X}(t)) + \sum_{j=1}^d \int_{-\infty}^{t^-} \alphastar_{ij} e^{-\betastar_{ij} (t-s)} d\oN{j}{} ) \\
    & = g_{i}^{\star}(X(t)) -  g_{i}^{\star}(\ov{X}(t)) + \sum_{j=1}^d \int_{0}^{t^-}\alphastar_{ij} e^{-\betastar_{ij} (t-s)} \qty( d\N{j}{}(s) - d\oN{j}{}(s) ) + \int^0_{- \infty}\alphastar_{ij} e^{-\betastar_{ij} (t-s) } \oN{j}{},
\end{align*}
we have

\begin{equation}
\begin{split}
\label{eq:inequality_lamb_lambbar}
    \abs{\lamb{i}{\star}(t) - \olamb{i}{\star}(t)} \leq \abs{g_{i}^{\star}(X(t)) -  g_{i}^{\star} (\ov{X}(t))} & +  \int^0_{- \infty}\alphastar_{ij} e^{-\betastar_{ij} (t-s) } \oN{j}{} \\
        & +\sum_{j=1}^d \int_{0}^{t^-}\alphastar_{ij} e^{-\betastar_{ij} (t-s)}  \abs{d\N{j}{}(s) - d\oN{j}{}(s)}.  
\end{split}
\end{equation}

}

Previous inequality leads to:
\begin{align*}
      \bb{E} \qty(\vert \lamb{i}{\star}(t) - \olamb{i}{\star}(t) \vert)  & \leq \bb{E}\qty(\vert g_{i}^{\star}(X(t)) - g_{i}^{\star} (\ov{X}(t))\vert)  + \bb{E}\qty(\sum_{j=1}^d \int_{(-\infty, 0)} \alphastar_{ij} e^{-\betastar_{ij} (t-s)}d\ov{N}_j(s)).\\
    &  \quad \quad \quad + \bb{E}\qty(\sum_{j=1}^d \int_{(0, t)} \alphastar_{ij} e^{-\betastar_{ij} (t-s)}\vert d\ov{N}_j(s)- d\N{j}{}(s)\vert) \\
    & = \bb{E}\qty(\vert g_{i}^{\star}(X(t)) - g_{i}^{\star} (\ov{X}(t))\vert) + \sum_{j=1}^d \int_{(-\infty, 0)} \alphastar_{ij} e^{-\betastar_{ij} (t-s)}\bb{E}\qty( \olamb{i}{\star}(s)) ds \\
    & \quad \quad \quad  + \sum_{j=1}^d \int_{(0, t)} \alphastar_{ij} e^{-\betastar_{ij} (t-s)} \bb{E}\qty( \vert\lamb{i}{\star}(s)- \olamb{i}{\star}(s) \vert)ds.
\end{align*}

Let us define 
{\modif $$f\qty(t):= \qty( \bb{E} \left[ \abs{ \lamb{i}{\star}(t) - \olamb{i}{\star}(t) } \right])_{ 1\leq i \leq d} \in \bb{R}^d$$ }
and 
$$r(t):=  \qty(  \bb{E} \qty(\Big| g_{i}^{\star}(X(t)) - g_{i}^{\star} (\ov{X}(t))\Big| ) + \sum_{j=1}^d \int_{(-\infty, 0)} \alphastar_{ij} e^{-\betastar_{ij}(t-s)} \bb{E} \qty(\olamb{\thetastar,j}{}(s) )   ds )_{1 \leq i \leq n} \in \bb{R}^d$$

then the previous inequality become:
$$f\qty(t) \leq ( \alphastar e^{-\betastar \cdot} \ast f )(t) + r(t),$$
which is to be understood componentwise.

This yields the same type of inequality as in the classical multivariate Hawkes process setting. By following the reasoning from the proof of Proposition 4.4 in \cite{clinet2017statistical}, we thus obtain that there exist $c,C>0$ such that: 
$$f(t) \leq  C e^{-ct }.$$

{\modif where the inequality is to be understood componentwise.}
This concludes the proof.

\section{Proof of Theorem \ref{theo:LLN_CLT}}
\label{proof:LLN_CLT}

The proof of this theorem is largely inspired by the proof of Theorems 1 and 2 in \cite{bacry2013some}, which establish analogous results in the case of a constant baseline. 
The main difference here is that we also need to handle the convergence of the baseline process. 
Following their approach, we define
$$
V(t):=N(t)-m(t) \in \bb{R}^d, 
\qquad 
\text{with} \quad m(t):=\bb{E}\qty[N(t)] \in \bb{R}^d.
$$
 
Then the following relations hold (where each equality are to be understood componentwise):
\begin{itemize}
    \item The function $m$ satisfies the following Volterra equation (see Chapter 4 of \cite{asmussen2003applied} for a deeper study of this equation):
    \begin{equation}
    \label{eq:renewal_m}
           m(t)=\int_0^t \bb{E}\qty[g^\star(X(s))] ds+\int_0^t h^\star(t-s) m(s) ds,
    \end{equation}
    so that, denoting $\ov{G}(t):=\int_0^t \bb{E}\qty(g^{\star}(X(s))) ds$, we get
    \begin{equation}
        m(t)= \ov{G}(t) + (h^\star \ast m)(t),
    \end{equation}
    leading to 
    \begin{equation}
    \label{eq:m_t}
        m(t) = \ov{G}(t) + \Psi \ast \ov{G}(t)
    \end{equation}
    with $\Psi$ the resolvent of the kernel $h^\star$, defined by Equation \eqref{eq:def_resolvant_kernel}.
    \item The centered process $V$ satisfies the renewal-type equation:
    \begin{equation}\label{eq:renewal_again}
    V(t)=M(t) + \Delta G(t) + h^\star \ast V(t) \qquad t\ge0,
    \end{equation}
    with {\modif $ \Delta G(t) = \int_0^t \qty( g(X(s)) - \bb{E}\qty( g(X(s)) ))ds$} and $M$ the martingale associated to $N$ (see \ref{paragraph:notation}).
    \item By iterating \eqref{eq:renewal_again}, we obtain the following (vectorial) representation:
    {\modif \begin{equation}
    \label{eq:U_renewal}
        V(t) =M(t)+ \Delta G(t) + \Psi \ast \qty( M + \Delta G)(t)
    \end{equation}}
\end{itemize}

\subsection*{Proof of the LLN property.}

The proof is divided into two parts.
\begin{enumerate}
    \item We start by showing that $ \frac{1}{T}\qty(N(Tv)-m(Tv))=\frac{1}{T}V(Tv)\xrightarrow[T\to\infty]{}0\quad\text{uniformly in }v\in[0,1],$  in $L^2$ and almost surely.
    \item We then prove that $ \frac{1}{T}m(Tv)\xrightarrow[T\to\infty]{} v (I-K)^{-1}\bar\mu \quad\text{uniformly in }v\in[0,1]. $
\end{enumerate}

 \vspace{0.3\baselineskip}

\noindent \underline{Let us start by the first point.} Namely, we show that
$$  \frac{1}{T}\qty(N(Tv)-m(Tv))=\frac{1}{T}V(Tv)\xrightarrow[T\to\infty]{}0 \quad\text{uniformly in }v\in[0,1] \text{ in } L^2 \text{ and } a.s   . $$

\paragraph{Convergence in $L^2$}
 
Using Equation \eqref{eq:U_renewal}, we get the following upper bound,
\begin{eqnarray*}
     \small
    \norm{ \frac{1}{T}V(Tv) }_{l_2} &\leq &\frac{1}{T} \qty( \sup_{v  \in [0,1] } \norm{M(vT)}_{l_2} +  \sup_{v  \in [0,1] } \norm{\Delta G(vT)}_{l_1} )  \qty(  1 + \int_{0}^{vT} \norm{\Psi(s)}_{l_2} ds )\nonumber\\
    &\leq &  \frac{1}{T} \qty(\sup_{v  \in [0,1] } \norm{M(vT)}_{l_2} + \sup_{v  \in [0,1] } \norm{\Delta G(vT)}_{l_2} ) \qty(  1 + \norm{\Psi}_{\s{L}_1} )
\end{eqnarray*}

So that, 

\begin{equation}
    \label{eq:upper_boundU}
     \small
    \norm{ \frac{1}{T}V(Tv)}_{l_2}^2 \leq \qty( \frac{2}{T^2}  \sup_{v  \in [0,1] } \norm{M(vT)}_{l_2}^2 + \frac{2}{ T^2}  \sup_{v  \in [0,1] } \norm{\Delta G(vT)}_{l_2}^2 )  \qty(  1 + \norm{\Psi}_{\s{L}_1} )^2
\end{equation}

Using Doob's inequality (with $p=2$) for the vector martingale $M$, we have
$$\bb{E}\qty(\sup_{v\in[0,1]}\norm{M(Tv)}_{l_2}^2) \le 4 \bb{E}\qty(\norm{M(T)}_{l_2}^2) =4\bb{E}\qty(\sum_{i=1}^d M_i(T)^2) = 4 \bb{E}\qty( \sum_{i=1}^d [M_i](T)) = 4 \bb{E}\qty(\sum_{i=1}^d N_i(T)) $$
As a result, 
$$\frac{1}{T^2 }\bb{E}\qty(\sup_{v\in[0,1]}\norm{M(Tv)}_{l_2}^2) \leq \frac{4}{T^2} \bb{E}\qty( \sum_{ i=1}^d \int_{0}^T \lamb{i}{\star}(t)dt) \leq \frac{4}{T} \sup_{t \ge 0} \bb{E}\qty(\norm{\lamb{}{\star}(t)}_{l_1}) \to 0 $$
 
where the convergence holds because $\bb{E}\qty(\norm{\lamb{}{\star}(t)}_{l_1})  < \infty$ thanks to Assumption \ref{ass:bounded_baseline1} and Proposition \ref{prop:bounded_moment_general_lambda}.

{\modif
For the S.D.E. term, write 
\begin{equation}\label{eq:decG}\frac1T\Delta G(vT) = A_T(v) - D_T(v)
\end{equation}
with $A_T(v) := \frac1T\int_0^{vT}\qty(g(X(s))-\pi_X(g)) ds$ and $D_T(v) := \frac1T\int_0^{vT}\qty(\bb{E}\qty(g(X(s)))-\pi_X(g)) ds$.

The ergodic theorem for $X$ gives $$\frac1t\int_0^t\qty(g(X(s))-\pi_X(g)) ds\to 0$$ a.s. so that $\sup_v\norm{A_T(v)}\to0$ a.s. Since $\sup_v\norm{A_T(v)}\le2\norm{g}_\infty$, dominated convergence gives $\bb{E}\qty(\sup_v\norm{A_T(v)}^2)\to0$. Then, exponential ergodicity gives for $s>0$, $$\norm{\bb{E}\qty(g(X(s)))-\pi_X(g)}\le C\eta^s,$$ hence $\sup_v\norm{D_T(v)}\le C/(T\abs{\log\eta})\to0$.
Using decomposition \eqref{eq:decG}, the bound \eqref{eq:upper_boundU}, and the martingale bound,  we obtain that $\bb{E}\qty(\sup_{v}\norm{\frac1T V(vT)}_{l_2}^2)\to0$ as $T$ goes to infinity.
}

\paragraph{Almost sure convergence}
 
Given Equation \eqref{eq:upper_boundU}, it is sufficient to prove that
$$\frac{1}{T}\sup_{v\in[0,1]}\norm{M(vT)}_{l_2} \xrightarrow[T \to \infty]{a.s} 0 .$$
Indeed, the S.D.E. term already converges almost surely: we showed that $\sup_v\norm{A_T(v)}\to0$ a.s. by the ergodic theorem, and $D_T$ is deterministic and tends to $0$. Hence only the martingale term remains.
According to the previous computations, for all $T\ge1$,

$$ \frac{1}{T^2} \bb{E}\qty(\sup_{v\in[0,1]} \norm{M(vT)}^2)\le \frac{1}{T} \frac{\bb{E}_{\pi_X}\qty(\norm{g^{\star}(\ov{X}(t))}_{l_\infty})}{1-\rho\qty(K)} .$$

Therefore, by Markov's inequality,
$$\bb{P}\qty(\frac{1}{T}\sup_{v\in[0,1] }\norm{M(vT)}>\epsilon) \le \frac{\bb{E}(\sup_v\norm{M(vT)}^2)} {\epsilon^2 T^2 }\le \frac{1}{\epsilon^2 T} \frac{\bb{E}_{\pi_X}\qty(\norm{g^{\star}(\ov{X}(t))}_{l_\infty})}{1-\rho\qty(K)} .$$
If we choose the subsequence $T_n=2^n$, then $\sum_n \bb{P}( T_n^{-1} \sup_{v \in [0,1]} \norm{M(vT_n)}>\epsilon) <\infty$, hence by Borel–Cantelli $T_n^{-1}\sup_v\norm{M(T_n v)}\to0$ almost surely. For any $T$ with $2^n \le T \le 2^{n+1}$, we have

$$\frac{1}{T}\sup_{v\in[0,1]}\norm{M(Tv)}=\frac{1}{T}\sup_{s\le T}\norm{M(s)} \le \frac{1}{2^n}\sup_{s\le 2^{n+1}}\norm{M(s)} \le 2\frac{1}{2^{n+1}}\sup_{v\in[0,1]}\norm{M(v2^{n+1})}$$

We thus obtain the desired convergence:
$$ \sup_{v\in[0,1]}\norm{\frac{1}{T}V(Tv)}\xrightarrow[T\to\infty]{}0 \quad\text{a.s.}$$

\vspace{\baselineskip}

\noindent \underline{We now turn to the proof of the following statement}:
$$ \frac{1}{T}m(Tv)\xrightarrow[T\to\infty]{} v (I-K)^{-1}\bar\mu \quad\text{uniformly in }v\in[0,1].$$
Recall that $\ov{G}(t):=\int_0^t \bb{E}\qty(g^{\star}(X(s))) ds$ and that the function $m$ satisfies the following equation:

$$m(t)=\overline{G}(t)+(\Psi\ast \overline{G})(t).$$

Therefore, for all $v\in[0,1]$, $$\frac{1}{T}m(Tv)=\frac{1}{T}\overline{G}(Tv)+\frac{1}{T}(\Psi\ast \overline{G})(Tv).$$

\emph{1. First term:} by the ergodicity assumption (e.g Assumption $\ref{ass:ergodicitySDE}$) and the fact that $g^{\star}$ is bounded (see Assumption \ref{ass:bounded_baseline1}), we have: 

$$\sup_{v\in[0,1]}\norm{\frac{1}{T}\overline{G}(Tv)-v\bar\mu}\xrightarrow[T\to\infty]{}0.$$

\emph{2. Second term} We write, by changing the variable $s=Tu$,
$$ \frac{1}{T}(\Psi\ast \overline{G})(Tv)=\int_0^v \Psi\qty(T(v-u))\overline{G}(Tu) du.$$ 

Let us use the uniform decomposition

$$ \frac{\overline{G}(Tu)}{T} = u\bar\mu + r_T(u),\qquad \text{with }\ \sup_{u\in[0,1]}\norm{r_T(u)} \xrightarrow[T\to\infty]{}0.$$ 
Then
\begin{align*}
\frac{1}{T}(\Psi\ast \overline{G} )(Tv)
&= T\int_0^v \Psi\qty(T(v-u))\qty(u\bar\mu + r_T(u)) du\\
&= T\int_0^v \Psi\qty(T(v-u))u du\ \bar\mu  +  T\int_0^v \Psi\qty(T(v-u))r_T(u) du.
\end{align*}

For the first term, we perform the integration change $w=T(v-u)$ and obtain
$$ T\int_0^v \Psi\qty(T(v-u))u du = \int_0^{Tv} \Psi(w)\qty(v-\tfrac{w}{T}) dw= v\int_0^{Tv}\Psi(w) dw - \frac{1}{T}\int_0^{Tv} w\Psi(w) dw.$$

Since $\psi\in L^1$, the first quantity converges to $v\int_0^\infty\psi(w) dw$ uniformly in $v$, and the second one tends to $0$ uniformly in $v$ by dominated convergence. The second term is controlled uniformly in $v$ by
$$\norm{T\int_0^v \Psi(T(v-u)) r_T(u)  du} \le \sup_{u}\norm{r_T(u)} \int_0^\infty\norm{\Psi(w)}  dw \xrightarrow[T\to\infty]{} 0 .$$

Thus
$$ \frac{1}{T}(\Psi\ast \ov{G})(Tv)\xrightarrow[T\to\infty]{} v\qty(\int_0^\infty\Psi(w) dw)\bar\mu \quad\text{uniformly in }v.$$

Noting the sum of convolutions, 

$$ \int_0^\infty\Psi(w) dw=\sum_{n\ge1} \int_0^\infty (h^\star)^{\ast n}(w) dw = \sum_{n\ge1} K^n = (I-K)^{-1}-I,$$

we finally obtain
$$\frac{1}{T}m(Tv)\xrightarrow[T\to\infty]{} v\qty(I +\int_0^\infty\Psi)\bar\mu = v\qty( Id - K)^{-1} \overline{\mu}.$$

\noindent\emph{Conclusion.}
Combining the two points above, we have for all $\epsilon>0$,
$$ \sup_{v\in[0,1]}\norm{\frac{1}{T}N(Tv)-v (I-K)^{-1}\bar\mu} \le \sup_{v}{\frac{1}{T}\qty(N(Tv)-m(Tv))}+ \sup_v{\frac{1}{T}m(Tv)-v (I-K)^{-1}\bar\mu}.$$
Each term on the right-hand side tends uniformly to $0$ as $T\to\infty$ (the first bound tends to $0$ in $L^2$ and almost surely, the second tends to $0$ uniformly). This establishes the uniform convergence announced, both almost surely and in the $L^2$ norm.

\section{Proof of Proposition \ref{prop:stationary_intensity_param_family}}
\label{proof:prop_stationary_intensity}
The first part of the proof is a consequence of \cite{bremaud1996stability}. Let us show the second part of the proof.  Using same reasoning as in the proof of Proposition \ref{prop:stationary_intensity_true_process} (see Equation \ref{eq:inequality_lamb_lambbar}), we have: 
\begin{align*}
    \bb{E} \left[ \abs{ \lamb{\theta,i}{}(t) - \olamb{\theta,i}{}(t) } \right] &  \leq \bb{E}\qty(\vert \bas{\mu}{i}(X(t)) - \bas{\mu}{i} (\ov{X}(t))\vert) \\
    & \quad \quad + \sum_{j=1}^d \int_{0}^{t^-} \alpha_{ij} e^{-\beta_{ij} (t-s)}\bb{E}\qty(  \abs{d\N{j}{}(s) - d\oN{j}{}(s)})+ \int_{-\infty}^0 \alpha_{ij} e^{-\beta_{ij} (t-s)}\bb{E}\qty(\olamb{\thetastar,j}{}) ds,\\
     & =  \bb{E}\qty(\vert \bas{\mu}{i}(X(t)) - \bas{\mu}{i} (\ov{X}(t))\vert) \\
     & \quad \quad + \sum_{j=1}^d \int_{0}^{t^-} \alpha_{ij} e^{-\beta_{ij} (t-s)}\bb{E}\qty(  \abs{ \lamb{\thetastar}{}(s) - \olamb{\thetastar}{}(s)})ds + \int_{-\infty}^0 \alpha_{ij} e^{-\beta_{ij} (t-s)}\bb{E}\qty(\olamb{\thetastar,j}{}) ds
\end{align*}

Thanks to the exponential speed of convergence of $\lamb{\thetastar,i}{}$ and $\bastar{\mu}{i}$ (respectively proved in Proposition \ref{prop:stationary_intensity_true_process} and Lemma \ref{lemma:speed_cv_stationnary}), we can deduce the stated result.

\section{Proof of Theorem \ref{theo:mixing}}
\label{proof:mixing}

The goal is to show that distant parts of the process become asymptotically independent.
To that end, we introduce a truncated version of the process $U$, that is:
\begin{eqnarray*}
\widetilde{U}_{\theta,i}(s,t) &:=& \left(\lamb{i}{\star}(t), \bas{\mu}{i}\qty(X\qty(t)) + \sum_{j=1}^d \int_{(s,t)} h_{\theta,i,j}(t-s)   d\N{j}{}(s), \nabla \bas{\mu}{i}\qty(X\qty(t)), \right.\\
&&\left. \quad \sum_{j=1}^d \int_{(s,t)} \nabla_{\theta}h_{\theta,i,j}(t-s)   d\N{j}{}(s) \right).
\end{eqnarray*}
The key step is to prove that the truncated and non-truncated processes share the same asymptotic behavior. By stability, this implies that the effect of the initial condition vanishes as time increases.

\subsection{Weak mixing property}

We start by proving the weak mixing property. Let us take $\psi , \phi \in \s{C}_b(E,\bb{R})$.

\begin{align*}
    \rho_{\theta,i}(u) &\leq \sup_{t\ge 0}\cov (\psi(U_{\theta,i}(t)), 
        \phi(U_{\theta,i}(t+u))-\phi(\widetilde{U}_{\theta,i}(t+\sqrt{u},t+u)) ) 
        && := A_1^\theta(u)\\
    &+ \sup_{t\geq 0 }\cov ( \psi(U_{\theta,i}(t)), 
        \phi(\widetilde{U}_{\theta,i}(t+\sqrt{u},t+u))  ) 
        && := A_2^\theta(u)
\end{align*}

\noindent \underline{Let us study the term $A_1^\theta(u) $.}
\begin{align*}
    A_1^\theta(u) & \leq \sup_{t \geq 0} \bb{V}\qty( \psi(U_{\theta,i}(t)) ) \bb{V}\qty(\phi(U_{\theta,i}(t+u))-\phi(\widetilde{U}_{\theta,i}(t+\sqrt{u},t+u)) ) \\
    & \leq \norm{\psi}_{L^{\infty}} \sup_{t \geq 0 } \bb{V}\qty(\phi(U_{\theta,i}(t+u))-\phi(\widetilde{U}_{\theta,i}(t+\sqrt{u},t+u)) ) \\
    & \leq \norm{\psi}_{L^{\infty}} \sup_{t \geq 0 } \bb{E}\qty( \norm{\phi(U_{\theta,i}(t+u))-\phi(\widetilde{U}_{\theta,i}(t+\sqrt{u},t+u)) }^2 ).
    \end{align*}
The first inequality follows from the Cauchy–Schwarz inequality, the second from the boundedness of $\psi$, and the third is a consequence of $\bb{V}(X) \leq \bb{E}\qty(X^2)$.
We now prove that $$\sup_{t \geq 0 } \bb{E}\qty( \norm{\phi(U_{\theta,i}(t+u))-\phi(\widetilde{U}_{\theta,i}(t+\sqrt{u},t+u)) }^2 ) \xrightarrow[]{} 0.$$
The proof is composed of two main steps. 
\begin{enumerate}
    \item First, we show that  $\displaystyle \sup_{t \ge 0} \bb{E}\qty[\norm{\widetilde{U}_{\theta,i}(t+\sqrt{u},t+u) - U_{\theta,i}(t+u)}] \xrightarrow[u \to \infty ]{} 0.$
    \item Then, we prove that for all $\epsilon > 0$, $\displaystyle \sup_{t \ge 0} \bb{P}(\norm{\phi(U_{\theta,i}(t+u))-\phi(\widetilde{U}_{\theta,i}(t+\sqrt{u},t+u))} > \epsilon)  \xrightarrow[u \to \infty ]{} 0.$
\end{enumerate}

\noindent  We start by the first point.

Since $\widetilde{U}_{\theta,i}(t+\sqrt{u},t+u)$ only depends on the process history up to $t+\sqrt{u}$. The exponential kernel implies
$$\sup_{t \ge 0} \bb{E}\qty[\norm{\widetilde{U}_{\theta,i}(t+\sqrt{u},t+u) - U_{\theta,i}(t+u)}] = \s{O}(e^{-\underline{\beta}\sqrt{u}/2}).$$

This ensures the convergence of $\sup_{t \ge 0} \bb{E}\qty[\norm{\widetilde{U}_{\theta,i}(t+\sqrt{u},t+u) - U_{\theta,i}(t+u)}]$.

\noindent Let's now turn to the second point

Let $\eta,\epsilon>0$. Since $U_{\theta,i}$ converges to a stationary process (see Proposition \ref{prop:stationary_intensity_param_family}), there exist a compact $K_\eta$ and $t_0\ge 0$ s.t. for all $t\ge t_0$, $u\ge 0$,
$$\bb{P}(U_{\theta,i}(t+u)\notin K_\eta \text{ or } \widetilde{U}_{\theta,i}(t+\sqrt{u},t+u)\notin K_\eta)\le \eta.$$
So for all $u$, for all $t\ge t_0$, we have 
 
\begin{align*} 
    & \bb{P}( \norm{\phi(U_{\theta,i}(t+u)) - \phi(\widetilde{U}_{\theta,i}(t+\sqrt{u}, t+u))}\ge \epsilon ) \\ 
    &\le \bb{P}( U_{\theta,i}(t+u) \notin K_\eta \ \cup\ \widetilde{U}_{\theta,i}(t+\sqrt{u}, t+u) \notin K_\eta ) \\ 
    &\quad + \bb{P}( U_{\theta,i}(t+u), \widetilde{U}_{\theta,i}(t+\sqrt{u}, t+u) \in K_\eta,\ \norm{\phi(U_{\theta,i}(t+u)) - \phi(\widetilde{U}_{\theta,i}(t+\sqrt{u}, t+u))} > \epsilon ). 
\end{align*}
 
And, on $K_\eta$, $\phi$ is continuous and therefore uniformly continuous. Therefore, there exists $\delta$ such that for all $x,y$ in $K_\eta$, \[ \norm{x-y} \le \delta \text{ implies } \norm{ \phi(x)-\phi(y) } \le \epsilon. \] So for all $u \ge 0$, for all $t \ge t_0$, 
\begin{align*}  
    \bb{P}\qty( \norm{\phi(U_{\theta,i}(t+u))-\phi(\widetilde{U}_{\theta,i}      (t+\sqrt{u},t+u)) } \ge \epsilon ) \leq \eta + \bb{P}\qty( \norm{U_{\theta,i}(t+u)-\widetilde{U}_{\theta,i}(t+\sqrt{u},t+u) } \ge \delta)
\end{align*} 

Since $\eta$ is chosen arbitrarily and $\sup_{s\ge 0 }\bb{P}\qty( \norm{U_{\theta,i}(t+u)-\widetilde{U}_{\theta,i}(t+\sqrt{u},t+u) } \ge \delta) \to 0 $. We can deduce that $\sup_{t \geq 0 } \bb{P}\qty( \norm{\phi(U_{\theta,i}(t+u))-\phi(\widetilde{U}_{\theta,i}(t+\sqrt{u},t+u)) } > \epsilon) \to 0.$\\

Let us now conclude the proof.
Let us fix $\epsilon>0$. For all $u,t\ge 0$, we have,
\begin{align*}
   & \quad \sup_{ t\ge 0}\bb{E}\qty( \norm{ \phi\qty( U_{\theta,i}(u+t)) - \phi\qty( \widetilde{U}_{\theta,i}(\sqrt{u}+T,u+t))  }^2 )  \\
    &\leq \epsilon^2  + \sup_{ t\ge 0} \bb{E}\qty( \norm{ \phi\qty( U_{\theta,i}(u+t)) - \phi\qty( \widetilde{U}_{\theta,i}(\sqrt{u}+T,u+t))}^2 \indic_{ \norm{\phi\qty( U_{\theta,i}(u+t)) - \phi\qty( \widetilde{U}_{\theta,i}(\sqrt{u}+T,u+t))} \ge \epsilon } )  \\
    & \leq \epsilon^2 + \norm{\phi}_{L^\infty}\sup_{ t\ge 0} \bb{P}\qty(  \norm{\phi\qty( U_{\theta,i}(u+t)) - \phi\qty( \widetilde{U}_{\theta,i}(\sqrt{u}+T,u+t))} \ge \epsilon ) 
\end{align*}
As a result, $ \sup_{ t\ge 0}\bb{E}\qty( \norm{ \phi\qty( U_{\theta,i}(u+t)) - \phi\qty( \widetilde{U}_{\theta,i}(\sqrt{u}+T,u+t))  }^2 )$ converge to $0$. We thus have shown that $A_1^\theta(u)  \xrightarrow[u \to \infty ]{}  0  $.

\noindent \underline{ We can now turn to  the study of $A_2^\theta(u) $.}

\begin{align*}
     A_2^\theta(u)  & =   \sup_{t\geq 0 }\cov ( \psi(U_{\theta,i}(t)), 
        \phi(\widetilde{U}_{\theta,i}(t+\sqrt{u},t+u))  ) \\
        & =   \sup_{t\geq 0 } \bb{E}\qty( \qty[ \psi(U_{\theta,i}(t)) - \bb{E}\qty( \psi(U_{\theta,i}(t))) ] \times  \qty[ \phi(\widetilde{U}_{\theta,i}(t+\sqrt{u},t+u)) - \bb{E}\qty(\phi(\widetilde{U}_{\theta,i}(t+\sqrt{u},t+u)))  ] )\\
        & { \modif = \sup_{t\geq 0 } \bb{E}\qty( \qty[ \psi(U_{\theta,i}(t)) - \bb{E}\qty( \psi(U_{\theta,i}(t))) ] \times  \bb{E}\qty[ \phi(\widetilde{U}_{\theta,i}(t+\sqrt{u},t+u)) - \bb{E}\qty(\phi(\widetilde{U}_{\theta,i}(t+\sqrt{u},t+u))) \vert \s{F}_t ] ) }
\end{align*}
{\modif Thus, 

\begin{align*}
    \abs{A_2^\theta(u)} & \leq 2 \norm{\psi} \sup_{t\geq 0 } \bb{E}\qty(\abs{\bb{E}\qty( \phi(\widetilde{U}_{\theta,i}(t+\sqrt{u},t+u)) - \bb{E}\qty(\phi(\widetilde{U}_{\theta,i}(t+\sqrt{u},t+u))) \vert \s{F}_t)})
\end{align*}

Denotes $R_t(u) :=  \bb{E}\qty( \phi(\widetilde{U}_{\theta,i}(t+\sqrt{u},t+u)) - \bb{E}\qty(\phi(\widetilde{U}_{\theta,i}(t+\sqrt{u},t+u))) \vert \s{F}_t)$.

As $\widetilde{U}_{\theta,i}(t+\sqrt{u},t+u)$ is a function of $Z$ between $[t + \sqrt{u}, t+u]$, $\widetilde{U}_{\theta,i}(t+\sqrt{u},t+u) \vert \s{F}_{t + \sqrt{u}}$ can be written as $\Psi(Z_{t + \sqrt{u}})$ (with $\Psi$ bounded as $\psi$ is bounded itself). 
Thus $\bb{E}\qty( \phi(\widetilde{U}_{\theta,i}(t+\sqrt{u},t+u)) \vert \s{F}_t)  =  \bb{E} \qty( \Psi(Z_{t + \sqrt{u}}) \vert \s{F}_t) = P^Z_{\sqrt{u}}\Psi (Z_t)$ where $P^Z$ is the transition kernel of $Z$, defined in Theorem \ref{theo:exp_ergodicity}.

We can deduce the following set of inequalities 
\begin{align*}
     \abs{A_2^\theta(u)} & \leq 2 \norm{\psi} \sup_{t\geq 0 } \bb{E} \qty( \abs{ P^Z_{\sqrt{u}}\Psi (Z_t) - \bb{E}\qty(P^Z_{\sqrt{u}}\Psi (Z_t)) } )\\
     &  \leq 2 \norm{\psi} \sup_{t\geq 0 } \bb{E} \qty( \abs{ P^Z_{\sqrt{u}}\Psi(Z_t) - \pi(\Psi)}  + \abs{\pi(\Psi)  - \bb{E}\qty( P^Z_{\sqrt{u}}\Psi(Z_t) ) } )\\
     & \leq 4 \norm{\psi} B \eta^{\sqrt{u}} \sup_{t\geq 0 } \qty( \bb{E}\qty(  f(Z_t) ) )
\end{align*}

And, as a consequence of Theorem \ref{theo:exp_ergodicity}, $\sup_{t\geq 0 } \qty( \bb{E}\qty(  f(Z_t) ) ) < \infty$. 

Thus, $A_2^\theta(u) \xrightarrow[u \to \infty]{} 0$}

\subsection{Strong mixing property}
The ideas and structure of the proof of the strong mixing result are essentially identical to those in Lemma A.6 of \cite{clinet2017statistical}.
The only difference lies in the way the baseline terms are incorporated: in our setting, the baseline process explicitly appears in $U_{\theta,i}$ and its truncated version $\widetilde{U}_{\theta,i}(s,t)$, whereas it is not the case in \cite{clinet2017statistical}.
Nevertheless, the proof in \cite{clinet2017statistical} can be followed line by line, replacing what we denote $U_{\theta,i}$ and $\widetilde{U}_{\theta,i}(s,t)$ by their analogues $X^i(t,\theta)$ and $\widetilde{X}^i(s,t,\theta)$.
The reasons are as follows:
\begin{itemize}
    \item The covariance can be decomposed, by applying the Cauchy--Schwarz inequality, into a ``difference'' term (analogous to $A_1$ in the preceding proof) and a covariance term involving the truncated block (analogous to $A_2$ in the same proof),
    \item the truncated future block has the same asymptotic law as original process,
    \item uniform moment and control bounds hold for all relevant processes.
\end{itemize}
Under our assumptions, all these properties remain valid. We therefore refer the reader to \cite{clinet2017statistical} for the detailed argument.

\section{Proof of Theorem \ref{theo:AsymptoticNormality}}
\label{proof:AsymptoticNormality}

The proof relies on Theorem 3.14 of \cite{clinet2017statistical} which gives four necessary conditions, ensuring the asymptotic normality of the MLE in the general framework of a point process. First, the conditions of application are specified, then their verification in our setting is demonstrated.
\vspace{\baselineskip}
\begin{customAss}{B}
\label{ass:clinet}
\begin{enumerate}[label={(M\arabic*)}]    
        \item \label{ass:clinet1} The mapping $\theta \mapsto \lamb{\theta}{}(t)$ is $\s{C}^3$ on $\Theta$ and admits a continuous extension to $\overline{\Theta}$, the closure of $\Theta$.
        \item \label{ass:clinet2} for all $p>1$ and $i \in \llbracket 0,4 \rrbracket$ we have
$$\sup_{t \in \bb{R}_+}\bb{E}\qty(\norm{ \sup_\theta \nabla_\theta^i \lamb{\theta}{}(t)}^p_p)< + \infty.$$
\item \label{ass:clinet3} There exists $\gamma \in (0,1/2)$ such that, for all $i$, $U_{\theta,i}$ is $\s{D}(E,\bb{R})$-mixing uniformly in $\theta \in \Theta$, with a rate $\rho_u$ satisfying

$$\rho_u = o(u^{-\epsilon}) \quad \text{for some } \epsilon > \frac{2\gamma}{1-2\gamma}.$$

Furthermore, there exists a limiting process $\olamb{\theta,i}{}$ such that for all $\psi \in D(E,\bb{R})$,
$$\sup_{\theta \in \Theta} t^\gamma \norm{ \bb{E}\qty[\psi(U_{\theta,i}(t))] -  \bb{E}\qty[\psi(\olamb{\thetastar,i}{}, \olamb{\theta,i}{}, \partial_\theta \olamb{\theta,i}{})] } \xrightarrow[t \rightarrow \infty]{} 0.$$
with U given in Equation \eqref{eq:U}.
\item \label{ass:clinet4} The quantity $\bb{Y}(\theta)$, defined by
\begin{equation}
\label{eq:Y(theta)}
\bb{Y}(\theta) = \sum_{i=1}^d\bb{E}\qty[ \olamb{\thetastar,i}{} - \olamb{\theta,i}{} + \log\qty(\frac{\olamb{\theta,i}{}}{\olamb{\thetastar,i}{}})\olamb{\thetastar,i}{} ]
\end{equation}
verifies: $$ \inf_{\theta \in \Theta} - \frac{\bb{Y}(\theta)}{\norm{\theta- \thetastar}^2} \geq 0.$$
    \end{enumerate}
\end{customAss}

\begin{customTheo}{B}[Theorem 3.14 of \cite{clinet2017statistical}] \label{theo:clinet}Under the Condition \ref{ass:clinet}, the MLE estimator $\MLE$ verifies 
    \[ \bb{E}\qty( \psi\qty( \sqrt{T} \qty( \MLE - \thetastar) )) \xrightarrow[T\to \infty]{} \bb{E}\qty( \psi\qty(\s{N}\qty( 0, \Gamma^{-1})))\]
    for every $\psi$ with polynomial growth.
\end{customTheo}

In order to apply Theorem~\ref{theo:clinet}, we verify that Conditions \ref{ass:clinet} hold in our setting.

Condition~\ref{ass:clinet1} is straightforward given the intensity formula. Condition~\ref{ass:clinet2} follows from Assumption~\ref{ass:moment_g(X)}$(p,j)$, which holds for all $p>1$ and $j \in \{1,2,3,4\}$,  together with Assumption~\ref{ass:bounded_baseline3}  and Proposition~\ref{prop:bounded_moment_general_lambda}. The first part of Condition~\ref{ass:clinet3} has already been established in Theorem~\ref{theo:mixing}. 

We now proceed to prove the two final point.

\begin{lemme} 
\label{lemme:strong_stationnarity}
Under Assumption \ref{ass:bounded_baseline3}  and \ref{ass:moment_g(X)}$(p,1)$, for all $p > 1$, the process verifies \ref{ass:clinet3}.
\end{lemme}

\begin{proof}

The first point is a consequence of the mixing property stated in Theorem \ref{theo:mixing}.

Let us show the second point.  We first note that under Assumption \ref{ass:bounded_baseline3}, the constant $C_\theta$ defined in Proposition \ref{prop:stationary_intensity_param_family}, is independent of $\theta.$
As a result, 
$$ \bb{E}\qty( \norm{\lamb{\theta,i}{}(t) - \olamb{\theta,i}{}})\le C e^{-ct }.$$
Then, for all $\phi \in D(E,\bb{R})$,

\begin{align*}
t^\gamma \bb{E}\qty[\norm{\phi(\widetilde{U}_{\theta,i}(t))-\phi(\ov{U}_{\theta,i}(t))}] &\leq t^\gamma \bb{E}\qty[\norm{\nabla \phi(\xi^\theta(t)). \qty(U_{\theta,i}(t)-\ov{U}_{\theta,i}(t)) }] \\
&\leq C t^\gamma \bb{E}\qty[\norm{ U_{\theta,i}(t)-\ov{U}_{\theta,i}(t)}^2] <\infty
\end{align*}

for some $\xi^\theta(t) \in[U_{\theta,i}(t), \ov{U}_{\theta,i}(t)]$, and $C>0$ a constant which does not depend on $\theta$ and $t$ and where the inequality is due to $\sup_t \bb{E}\qty[\sup_\theta \nabla \phi(\xi^\theta(t))]<+ \infty$. 

Thanks to Proposition \ref{lemma:speed_cv_stationnary} and the fact that each term in $ U_{\theta,i}(t) - \ov{U}_{\theta,i}(t) $ can be written as a sum of integrals with respect to the measure $ dN - d\ov{N} $, and a term $\bb{E}\qty(\norm{\bas{\mu}{i}(X(t)) - \bas{\mu}{i} (\ov{X}(t) }) $ we deduce that $ \bb{E} \qty[ \norm{ U_{\theta,i}(t) - \ov{U}_{\theta,i}(t)}^2] $ decreases exponentially uniformly in $ \theta$.

This shows the convergence stated in \ref{ass:clinet3} and thus conclude the proof.
\end{proof}

\begin{lemme}
    Under Assumptions \ref{ass:ergodicitySDE},\ref{ass:identifiability_baseline}, the process defined by Equation \eqref{eq:modelLambda} verifies condition \ref{ass:clinet4}.
\end{lemme}

\begin{proof}
Note that $\bb{Y}(\theta)$, defined in \eqref{eq:Y(theta)} as the asymptotic difference of the log-likelihood, is given by:
$$\bb{Y}(\theta) = \sum_{i=1}^d\bb{E}\qty[ \olamb{\thetastar,i}{} - \olamb{\theta,i}{} + \log\qty(\frac{\olamb{\theta,i}{}}{\olamb{\thetastar,i}{}})\olamb{\thetastar,i}{} ],$$
 
where $\olamb{\theta,i}{}(t)$ denotes the intensity of the stationary version of the process given in Equation \eqref{eq:stationary_intensity}. We start by showing that $\bb{Y}(\theta) \leq 0 $ with equality iff $\theta = \thetastar$.

Note that
$$\bb{E}\qty[ \olamb{\theta,i}{}(t) - \olamb{\thetastar,i}{}(t) - \log\qty(\frac{\olamb{\theta,i}{}(t)}{\olamb{\thetastar,i}{}(t)}) \olamb{\thetastar,i}{}(t) ] = \bb{E}\qty[ \olamb{\thetastar,i}{}(t) f\qty(\frac{\olamb{\theta,i}{}(t)}{\olamb{\thetastar,i}{}(t)}) ]$$
with $f: x \mapsto 1 - x + \log(x)$ a negative function, null only in $x = 1$.
Therefore, $\bb{Y}(\theta) = 0$ implies that, for all $i\in \{1,\ldots,d\}$ for all $t$, $\olamb{\theta,i}{}(t) = \olamb{\thetastar,i}{}(t)$, $\bb{P}$-almost certainly.

Let us show that $\theta=\thetastar$. Both processes are càdlàg, implying that, for all $i$, we have $\ps{}$
 
\begin{equation}
\label{eq:identification}
    \bas{\mu}{i}(\ov{X}(t)) - \bastar{\mu}{i}(\ov{X}(t)) = \sum_{j=1}^d \int_{(-\infty,t)} \qty(\alphastar_{ij} e^{-\betastar_{ij}(t-s) } - \alpha_{ij} e^{-\beta_{ij}(t-s) }) d\ov{N}_j(s).
\end{equation}

The left member of this equality is a continuous process, whereas the right member is a jump process, which jumps have size $(\alpha_{ij}-\alphastar_{ij})_j$, implying that for all $i,j\in \bb{R}^d$ $\alpha_{ij} = \alphastar_{ij}$.

Then, the right-hand side of Equation~\eqref{eq:identification} is $\s{C}^1$ with respect to time. Hence, it has a finite variation and therefore a zero quadratic variation. 
Thus, the left-hand side of the equation has zero quadratic variation. Since for all $\mu$, the function $\bas{\mu}{i}$ is of class $\s{C}^2$, Itô's formula implies that the quadratic variation of the process $\qty(\bas{\mu}{i} - \bastar{\mu}{i})(X(t))$ is given by:

$$\int_{0}^{t} \nabla_x \qty(\bas{\mu}{i} - \bastar{\mu}{i})(\overline{X}(t)(s))^\top \, \sigma(\overline{X}(t)(s)) \sigma(\overline{X}(t)(s))^\top \, \nabla_x \qty(\bas{\mu}{i} - \bastar{\mu}{i})(\overline{X}(t)(s)) \, ds.$$

Since $\sigma(x) \sigma(x)^\top$ is positive definite for all $x$, the integrand must vanish almost everywhere. Therefore, for all $t$,
\[
\nabla_x \qty(\bas{\mu}{i} - \bastar{\mu}{i})(\overline{X}(t)) = 0.
\]
 
Then, by Assumption~\ref{ass:identifiability_baseline}, we deduce that
\[
\nabla_x \qty(\bas{\mu}{i} - \bastar{\mu}{i})(x) = 0 \quad \text{for all } x,
\]
so that $\bas{\mu}{i} - \bastar{\mu}{i}$ is a constant function. Using Equation \eqref{eq:identification} with this information, we have shown that for all $i \in \{ 1, \ldots,d\}$ \[
\sum_{j=1}^d \int_{(-\infty,t)} \qty(\alpha_{ij} e^{-\betastar_{ij}(t-s) } - \alpha_{ij} e^{-\beta_{ij}(t-s) }) d\ov{N}_j(s)
\]
is constant w.r.t time and therefore for all $i,j$, $\beta_{ij} = \betastar_{ij}$.
Hence, we now have that $\bas{\mu}{i}(\ov{X}(t)) - \bastar{\mu}{i}(X(t)) = 0$.
And, thanks to Assumption \ref{ass:identifiability_baseline}, we know that if $\mu_i \ne \mustar_i $, $\pi_X \qty(\qty(\bas{\mu}{i}- \bastar{\mu}{i})^{-1}\qty(\bb{R}^{d_i}\backslash\{0\}))>0$. Thanks to the Harris recurrence property of $X$, this ensures that $X$ returns an  infinite number of times in $\bb{R}^{d_i}\backslash\{0\}$. As a result, $\mu_i = \mustar_i$, which concludes the proof.

At this point, we have proven that the function $\bb{Y}$ is always negative except at $\theta^\star$, where the it value is null. 
To prove \ref{ass:clinet4}, it remains to show that the Fisher matrix is positive definite, using Assumption \ref{ass:identifiability_baseline}.

The Fisher matrix is given by $\Gamma = \sum_{i=1}^d \bb{E} \qty(\frac{\partial_\theta \olamb{\thetastar,i}{} (\partial_\theta \olamb{\thetastar,i}{})^T }{\olamb{\thetastar,i}{}(t)})$.

Let us consider a vector \(v:= (v_i)_i \), where for each \(i \in \{1, \dots, d\} \), the vector \(v_i \in \bb{R}^{q_i} \) has the same dimension and structure as the parameter block \(\theta_i = (\mu_i, (\alpha_{ij})_j, (\beta_{ij})_j) \). More precisely, each \(v_i \) can be written as:
\[
v_i = (v_{\mu_i}, (v_{\alpha_{ij}})_j, (v_{\beta_{ij}})_j),
\]
where \(v_{\mu_i} \in \bb{R}^{d_i} \), \(v_{\alpha_{ij}} \in \bb{R} \), and \(v_{\beta_{ij}} \in \bb{R} \) for all \(j = 1, \dots, d \).

Suppose now that the quadratic form satisfies:
\[
v^\top \Gamma v = 0.
\]
As $v^T \Gamma v = \bb{E}\qty(v^T_i \frac{\partial_\theta \olamb{\thetastar,i}{}(t) \partial_\theta \olamb{\thetastar,i}{}(t)^T }{\olamb{\thetastar,i}{}(t)} v_i)  = \bb{E}\qty(\frac{ \norm{ \partial_\theta \olamb{\thetastar}{}(t)^T v}^2_2 }{ \partial_\theta \olamb{\thetastar,i}{}(t) }) $, it implies that for all $i$, $\partial_\theta \olamb{\theta,i}{}(t) v^T=0$ $\bb{P}-a.s$, hence $\ps{},$ we have that for all $t$,

\begin{equation}
\label{eq:FisherMat}
    \nabla_\mu \bas{\mu}{i}(\ov{X}(t)) v_{\mu_i}^T +  \sum_{j=1}^d v_{\alpha_{ij}} \int_{(-\infty, t) }  e^{-\beta_{ij}(t-u)}d\ov{N}_j(u) - v_{\beta_{ij}} \alpha_{ij} \int_{(-\infty, t) }(t-u) e^{-\beta_{ij}(t-u)}d\ov{N}_j(u)= 0.
\end{equation}

The function $ t \mapsto \nabla_\mu \bas{\mu}{i}(\ov{X}(t)) v_{\mu_i}^T - \sum_{j=1}^d v_{\beta_{ij}} \alpha_{ij} \int_{(-\infty, t) } (t-u) e^{-\beta_{ij}(t-u)}d\ov{N}_j(u) $ is continuous, whereas the term $\sum_{j=1}^d v_{\alpha_{ij}} \int_{(-\infty, t) }  e^{-\beta_{ij}(t-u)}d\ov{N}_j(u)$ is a jumping process, with discontinuities of size $ v_\alpha $. Thus, for the left-hand side of the equation to be continuous almost surely, it is necessary that $ v_\alpha = 0 $.
As a result, using Equation \eqref{eq:FisherMat}, 
 
$$\nabla_{\mu} \,\bas{\mu}{i}(\ov{X}(t)) v_{\mu_i}^T - \sum_{j=1}^d v_{\beta_{ij}} \alpha_{ij} \int_{(-\infty, t) }(t-u) e^{-\beta_{ij}(t-u)}d\ov{N}_j(u) =0 \quad \ps{}$$
which, using the derivation w.r.t the parameter $\alpha$, ensures $v_\beta=0$.
Finally, as $\pi_X\qty((\nabla_\mu \bas{\mu}{i})^{-1}\qty(\bb{R}^\star))>0$, we can conclude that $v_{\mu_i} =0$.

\end{proof}

\section{Proof of Theorem \ref{theo:gof_corrige}}
\label{proof:BaarsTest}

Let us recall the notations 
$$E_i := \Lamb{\thetastar}{}(T_i) - \Lamb{\thetastar}{}(T_{i-1}) \sim \mathcal{E}(1),
\quad \w{E}_i(T) := \Lamb{\wMLE}{}(T_i) - \Lamb{\wMLE}{}(T_{i-1}) + \frac{1}{\sqrt{T}} (T_i - T_{i-1}) \, \chi_T \, \widetilde{I}_T \, \widetilde{X}. $$

We proceed in three steps:

\begin{enumerate}
    \item First, we show that for any $i$, $ \w{E}_i(T)  \xrightarrow[T \to \infty]{\s{L}} \s{E}(1)$.
    \item Next, we show that for any $k$, $\qty(\w{E}_1(T), ... , \w{E}_k(T)) \xrightarrow[T \to \infty]{\s{L}} \bigotimes_{j=1}^k \s{E}(1)$.
    \item Finally, we show that $(\w{E}_i(T))$ is tight. Specifically, we show that there exist $M_i$ such that
    $$    \sum_i \sup_T \bb{P}\qty( \vert\w{E}_i(T) \vert > M_i ) < \infty.$$
\end{enumerate}

\paragraph{Step 1} Using  \cite{baars2025asymptotically} (see Equation 16 of Section 4.2):
\begin{align*}
\Lamb{\wMLE}{}(\T{i}{}) -\Lamb{\wMLE}{}(\T{i-1}{}) & = \Lamb{\thetastar}{}(\T{i}{}) -\Lamb{\thetastar}{}(\T{i-1}{})+ \frac{1}{\sqrt{T}} \Delta \T{i}{}  \chi \Gamma^{-1/2} X  + o_{\bb{P}}\qty( \frac{1}{\sqrt{T}} ),
\end{align*}
and hence
\begin{equation}
\label{cv:theobaars}
\Lamb{\wMLE}{}(\T{i}{}) -\Lamb{\wMLE}{}(\T{i-1}{}) - E_i - \frac{1}{\sqrt{T}}\Delta \T{i}{} \chi \Gamma^{-1/2} X =   o_{\bb{P}}\qty( \frac{1}{\sqrt{T}} ).
\end{equation}

For any $\eta>0$,
\begin{align*}
\bb{P}\qty( \norm{ \w{E}_i(T) -  E_i}>\eta) & \leq \bb{P}\qty( \norm{\Lamb{\wMLE}{}(\T{i}{}) -\Lamb{\wMLE}{}(\T{i-1}{}) -   \Lamb{\thetastar}{}(\T{i}{}) +\Lamb{\thetastar}{}(\T{i-1}{})  + \frac{1}{\sqrt{T}}\Delta \T{i}{} \chi\Gamma^{-1/2} X }> \eta/2 ) \\
& + \bb{P}\qty( \norm{   \frac{1}{\sqrt{T}}\Delta \T{i}{}\qty( \chi\Gamma^{-1/2} X  -  \chi_T \widetilde{I}_T \widetilde{X} )} >\eta/2 ).
\end{align*}

The first term tends to 0 according to Equation \eqref{cv:theobaars}. For the second term, we have
\begin{align*}
    &\bb{E}\qty( \norm{   \frac{1}{\sqrt{T}}\Delta \T{i}{}\qty( \chi\Gamma^{-1/2} X  -  \chi_T \widetilde{I}_T \widetilde{X} )} ) \\
    & \quad  = \frac{1}{\sqrt{T}}\bb{E}\Delta \T{i}{} \times \qty[ \bb{E}\qty( \norm{\chi\Gamma^{-1/2}} ) \bb{E}\qty(\norm{X}) +  \bb{E}\qty( \norm{\chi_T \widetilde{I}_T} ) \bb{E}\qty(\norm{\widetilde{X}}) ],
\end{align*} 
where the equality holds since $(\T{i}{})$ are independent of $X,\widetilde{X}$, of $\chi_T, \widetilde{I}_T$ (which depend only on $\wN{}{}$), and of $\chi,\Gamma$ (also independent of $N$). 

Under the assumptions of moment convergence for the MLE, $\bb{E}\qty(\norm{\chi_T}^2)$ and $\bb{E}\qty(\norm{ \widetilde{I}_T }^2)$ are bounded (see Proposition \ref{prop:stationary_intensity_param_family}). We thus conclude that for all $i$, $\w{E}_i(T)$ converges to $E_i$, hence to an exponential distribution with rate 1.

\paragraph{Step 2.} Now we show that $\qty(\w{E}_1(T), ... , \w{E}_k(T)) \xrightarrow[T \to \infty]{\s{L}} \bigotimes_{j=1}^k \s{E}(1)$.

We proceed by induction. First, we show that for any $k<l$, $(\w{E}_k(T),\w{E}_l(T))\xrightarrow[T \to \infty]{\s{L}}  \s{E}(1) \bigotimes \s{E}(1) $.

Let $t_k, t_l \in \bb{R}$,
\begin{align*}
    \bb{E}\qty( e^{i t_l \w{E}_l(T) + i t_k \w{E}_k(T)}) &= \bb{E}\qty( \bb{E}\qty( e^{it_l \w{E}_l(T) } e^{it_k \w{E}_k(T) } \big| (\wT{j}{}),(\T{j}{})_{j\leq l-1}  )).
\end{align*}

For all $j < l$, $\Delta T_l \indep T_j \vert (\T{j}{})_{j\leq l-1}$, and since we also condition on $(\wT{j}{})$ and $\widetilde{X}$, $\w{E}_l(T)$ and $\w{E}_k(T)$ depend only on the inter-arrival times. Therefore,
\begin{align*}
  \bb{E}\qty( e^{i t_l \w{E}_l(T) + i t_k \w{E}_k(T)}) & =  \bb{E}\qty( \bb{E}\qty( e^{it_l \w{E}_l(T) } \big| (\wT{j}{}),(\T{j}{})_{j\leq l-1}, \widetilde{X} ) \bb{E}\qty( e^{it_k \w{E}_k(T) } \big| (\wT{j}{}),(\T{j}{})_{j\leq l-1}, \widetilde{X} )) \\
  & = \bb{E}\qty( e^{it_k \w{E}_k(T) } \bb{E}\qty(  e^{it_l \w{E}_l(T)}  \big| (\wT{j}{}),(\T{j}{})_{j\leq l-1}, \widetilde{X} )  ).
\end{align*}

Since $\w{E}_l(T)$ converges in probability to $E_i$, by the mapping theorem $e^{it_l \w{E}_l(T)}$ converges in probability to $e^{it_l E_l}$, and as $\norm{e^{it_l \w{E}_l(T)}} <1$, we have convergence in $\s{L}^1$. Similar reasoning apply for $e^{it_k \w{E}_k(T)}$ which lead to

$$\bb{E}\qty( e^{i t_l \w{E}_l(T) + i t_k \w{E}_k(T)}) \xrightarrow[T\to \infty]{\s{L}^1} \bb{E}\qty(e^{it_k E_k }) \bb{E}\qty(  e^{it_l E_l}  ).$$

Thus, $(\w{E}_k(T),\w{E}_l(T))\xrightarrow[T \to \infty]{\s{L}}  \s{E}(1) \bigotimes \s{E}(1)$. By induction, for any $k>0$, $\qty(\w{E}_1(T), ... , \w{E}_k(T)) \xrightarrow[T \to \infty]{\s{L}} \bigotimes_{j=1}^k \s{E}(1)$.

\paragraph{Step 3} We show that there exist $(M_i)$ such that
$$\sum_i \sup_T \bb{P}\qty( \vert \w{E}_i(T) \vert > M_i ) < \infty.$$

To prove this inequality, we show that $\sup_T \bb{E}\qty( \vert  \w{E}_i(T) \vert )  \leq c_i$ for some $c_i>0$. Indeed,
\begin{equation}
\label{eq:borne_supEi}
\bb{E}\qty(  \vert  \w{E}_i(T) \vert) \leq \bb{E}\qty(\Delta_i \Lamb{\wMLE}{}) +\frac{1}{\sqrt{T}} \bb{E}\qty(\Delta \T{i}{}   \norm{\chi_T \widetilde{I}_T \widetilde{X}} ).
\end{equation}

The second term tends to 0 since $\bb{E}\qty(\norm{\chi_T \widetilde{I}_T})$ is bounded and $\bb{E}\qty(\Delta T_i) < \infty$, so there exists $c_i^1$ such that

$$\frac{1}{\sqrt{T}} \bb{E}\qty(\Delta \T{i}{}   \norm{\chi_T \widetilde{I}_T \widetilde{X}} ) \leq c_i^1.$$

For the first term in \eqref{eq:borne_supEi}, conditioning on $(\wT{j}{})$, $\lamb{\wMLE}{}$ is a process with mean $\widetilde{\mu}_T + \widetilde{K} \qty(\mathrm{I} - K^\star )^{-1} \mustar$.

Thus,
\begin{align*}
    \bb{E}\qty( \Lamb{\wMLE}{}(T_i) - \Lamb{\wMLE}{}(T_{i-1}))  & = \bb{E}\qty( \Delta T_i \qty(\widetilde{\mu}_T + \widetilde{K}_T \qty(\mathrm{I} - K^\star )^{-1} \mustar )) \\
    & = \bb{E}\qty( \Delta T_i ) \bb{E}\qty( \widetilde{\mu}_T + \widetilde{K}_T \qty(\mathrm{I} - K^\star )^{-1} \mustar ),
\end{align*}
which is bounded. Hence $\sup_T \bb{E}\qty( \Lamb{\wMLE}{}(T_i) - \Lamb{\wMLE}{}(T_{i-1}) )  < + \infty$.

This concludes the third step, and thus the proof.

\appendixheaderon
\section{T-chain property for a general d}
\label{proof:Tchaingeneraldim}

We resume the proof of \ref{proof:Tchaindim2} from Equation \eqref{eq:lower_bound_kernel}. As in the case $d=2$, we begin by bounding $\bb{E}_{(x,y)} \qty(\indic_B(Y(T))  \vert  \s{F}_T^W )$ from below.
We thus work on the event
$$A_T:= \{ T_1 < T_2 < \ldots < T_{d^2} < T < T_{d^{2}+1} \} \cap \{ K_1 = 1, \ldots K_d = 1 , K_{d+1}=2 \ldots..., K_{2d} = 2, \ldots K_{d^2} = d \}.$$

In other words, we require that each component jumps $d$ times before time $T$, 
and that the jumps of $N_1$ occur before those of component $N_2$ which occur before those of $N_3$ and so on.
Then, we have

\begin{align*}
    & \bb{E}_{(x,y)} \qty(\indic_B(Y(T))  \vert  \s{F}_T^W ) \\
    &\geq \bb{E}_{(x,y)} \qty(\indic_B(Y(T)) \indic_{A_T}  \vert  \s{F}_T^W )\\
    & \int f^{ (\Delta T_{d^{2+1}}, K_{d^2 +1}, \dots, \Delta T_{1}, K_{1})}(t_{d^2 +1 }, k_{d^2+1}, \dots, t_1, k_1)  \indic_{B}(\gamma_y(T)) \indic_{t_1+\ldots +t_{d^2}< T } \indic_{t_1+\ldots +t_{d^2+1} > T }    dt_1 \dots dt_{d^2+1},
\end{align*}
where:
\begin{itemize}
    \item $f^{ (\Delta T_{1}, K_{1}, \dots, \Delta T_{5}, K_{5})}(t_5, k_5, \dots, t_1, k_1)$ is the joint density of the inter-arrival times $(\Delta T_i)$ and the marks $(K_i)$, conditional on the trajectory of $W$;
    \item $\gamma_y(t) \in \bb{R}^{d^2}$ denotes the value of $Y(t)$, for $t \in [0,T]$, conditional on the event
    $\{ Y(0) = y \} \cap \{ \Delta T_1 = t_1, K_1 = 1, \dots, \Delta T_4 = t_4, K_4 = 2 \}$ and conditional to $W$.
\end{itemize}

\paragraph{Computation of $\gamma_y(T)$.}

Using the recurrence property of  $Y$ we have:
\begin{equation}
\label{eq:kernel_val_generaldim}
\gamma_y(T) = \qty(y_{ij} e^{-\betastar_{ij} T} + \alphastar_{ij} \sum_{i=jd+1}^{(j+1)d}   e^{-\betastar_{ij} (T - t_1 - \dots - t_i)} )_{ij}.
\end{equation}

We denote by $(\gamma_y^{kj}(t))_{1 \leq k,j \leq 2}$ the entries of the matrix $\gamma_y(t)$, and define, for each $k$,
$$ \gamma_y^k(s) := \sum_{j=1}^d \gamma_y^{ij}(s),$$

\paragraph{Computation of the joint density $f $.}

We aim to make explicit the joint density $f $ of the random vector $(\Delta T_{d^2+1}, K_{d^2+1}, \dots, \Delta T_1, K_1)$. Using Bayes' formula, we have:

$$f^{(\Delta T_{d^2}, K_{d^2}, \dots, \Delta T_1, K_1)} =\prod_{k=0}^{d-1} \; \prod_{i=dk+1}^{(k+1)d} f^{\qty(\Delta T_i, K_i \mid \Delta T_{<i}, K_{<i})}.$$

where the notation $f^{\qty(\Delta T_i, K_i \mid \Delta T_{<i}, K_{<i})}$ corresponds to the conditional density of the pair $(\Delta T_i, K_i)$ given the previous history.

To compute the conditional terms $f^{\qty(\Delta T_i, K_i \mid \Delta T_{<i}, K_{<i})} $, we first consider the survival function of the $i $-the inter-arrival time, given the past. For all $i \in \llbracket dk+1, d(k+1) \rrbracket $, we have:

\begin{align*}
\bb{P}(\Delta T_i > t_i \mid \Delta T_{i-1}, \dots, \Delta T_1, K_1, \s{F}_W^T) 
&= \bb{P}(N_k([T_{i-1}, T_{i-1} + t_i]) = 0 \mid \Delta T_{i-1}, \dots, \Delta T_1, K_1, \s{F}_W^T) \\
&= \exp\qty(- \int_{T_{i-1}}^{T_{i-1} + t_i} g_{k}^\star(X(s)) + \gamma_y^{k}(s)ds),
\end{align*}

where $(\gamma_{y}^{ij}(s))_{ij} $, the coordinates of $\gamma^T_y(s) $ defined in Equation \eqref{eq:kernel_val}, represents the value of the Hawkes process $Y(s) $ conditionally on the Brownian motion, the initial condition $y $, and the jump times before time $s $.

Hence, for all $i$, the conditional density of $\Delta T_i$ is given by

$$t \mapsto \qty(g^{\star}_{k}(X(T_{i-1}+t)) + \gamma_y^k(T_{i-1}+t)) \exp\qty(- \int_{T_{i-1}}^{T_{i-1}+t}\qty[ g^{\star}_{k}(X(s)) + \gamma_y^k(s)] ds ),$$

 Note that for $i=d^2+1$, $k_{d^2+1}$ plays no role as we have no condition imposed on it.

For the conditional mark density, given $\Delta T_i, \dots, \Delta T_1, K_1, \s{F}_W^T $, the probability that the event occurs in component $k $ is:

$$\bb{P}(K_i = k \mid \Delta T_i, \dots, \Delta T_1, K_1, \s{F}_W^T ) = \frac{\lamb{k}{}(T_i^-)}{\sum_{\ell=1}^d \lamb{l}{}(T_i^-)} 
= \frac{g^{\star}_{k}(X(T_i^-)) +  \gamma_y^{\ell}(T_i^-)}{\sum_{\ell=1}^d \qty(g^{\star}_{l}(X(T_i^-)) + \gamma_y^{\ell}(T_i^-) )}.$$

We can then write the full conditional density $f^{\qty(\Delta T_i, K_i \mid \Delta T_{<i}, K_{<i})} $ as:

\begin{align*}
    & f^{\qty(\Delta T_i, K_i \mid \Delta T_{<i}, K_{<i})}(t_i,k_i, t_{i-1}, k_{i-1}\dots, t_1,k_1) \\
    &= \frac{g^{\star}_{k_i}(X(T_i^-)) + \gamma_y^{k_i}(T_i^-)}{\sum_{\ell=1}^d \qty(g^{\star}_{l}(X(T_i^-)) +  \gamma_y^{ l}(T_i^-) )} \\
    & \times \qty(g^{\star}_{k_i}(X(T_{i-1} + t_i)) + \gamma_y^{k_i}(T_{i-1} + t_i) )  \times \exp\qty(- \int_{T_{i-1}}^{T_{i-1} + t_i}  g^{\star}_{k_i}(X(s)) +  \gamma_y^{k_i}(s)ds).
\end{align*}

Now, using the fact that the baseline intensities $\qty(g^{\star}_{k}(X(t)) )_k $ are bounded above and below (see Assumption \ref{ass:bounded_baseline2}), we can deduce a lower bound:

\begin{align*}
    f^{\qty(\Delta T_i, K_i \mid \Delta T_{<i}, K_{<i})}(t_i,k_i, t_{i-1}, k_{i-1}\dots, t_1,k_1)&\geq \frac{g_{-}+ \gamma_y^{k}(T_i^-)} {\sum_{\ell=1}^d \qty(g_{+}+  \gamma_y^{k_i}(T_i^-) )} \\
    &\quad \times  \qty(g_{-}+ \gamma_y^{k_i}(T_{i-1} + t_i) )  \times \exp\qty(- \int_{T_{i-1}}^{T_{i-1} + t_i}  g_{+}+  \gamma_y^{k_i}(s)ds).
\end{align*}

We are thus once again in the setting of a multidimensional Hawkes process with constant baseline. The same arguments as in Lemma A.3 of \cite{clinet2017statistical} apply, showing that:

\begin{eqnarray*}
    Q_{2,T}(y,B)&:=& \int f^{ (\Delta T_{d^2+1} , K_{d^2+1}, \dots, \Delta T_{1}, K_{1})}(t_{d^2+1}, k_{d^2+1}, \dots, t_{1}, k_{1})   \\
    &&\quad \quad \indic_{B}(\gamma_y(T)) \indic_{t_1+\ldots +t_{d^2}< T } \indic_{t_1+\ldots +t_{d^2+1} > T }    dt_1 \dots dt_{d^2+1}
\end{eqnarray*}
is a kernel that is non-trivial and lower semi-continuous.

Combining this with Equation \eqref{eq:lower_bound_kernel} and using Assumption \ref{ass:bounded_baseline2}, we obtain:

$$P^{Z}_{T}\qty((x,y), A\times B) \geq Q_T \qty((x,y), A\times B):= Q_{1,T }(x,A) Q_{2,T}(y,B)$$

with $Q_{1,T }(x,A) =\bb{E}_x\qty(\indic_A(X(T)) )$. This inequality proves that $P $ is a $T $-chain.

The rest of the proof is the same as in \ref{proof:Tchaindim2}.

\end{appendix}


\end{changemargin}

\end{document}